\documentclass{tac}

\usepackage{xy}

\usepackage{amsfonts}

\usepackage{amssymb}

\usepackage{amsmath}

\usepackage{bbm}

\input diagxy

\thirdleveltheorems

\author {Marco Grandis and Robert Par\'e}

\thanks{Work partly supported by a research contract of the University of Genoa.}

\address{Dipartimento di Matematica, Universit\`a di Genova. Via Dodecaneso 35, 16146-Genova, Italy\\Department of Mathematics and Statistics, Dalhousie University,
 Halifax, NS, Canada, B3H 4R2}

\title {Intercategories: \\A framework for three-dimensional category theory}

\copyrightyear{2014}


\eaddress{grandis@dima.unige.it\\pare@mathstat.dal.ca}

\def\ic{\bf\sf}

\mathrmdef{id}
\mathrmdef{inj}
\mathrmdef{Id}
\mathrmdef{Lan}

\def\CCat{{\mathbb C}{\rm at}}

\def\Doub{{\mathbb D}{\rm bl}}

\def\ICat{{\ic ICat}}

\def\DblSt{{\cal P}{\it s}{\cal D}{\it bl}}

\def\Dlax{{\cal L}{\it x}{\cal D}{\it bl}}
\def\Dcolax{{\cal C}{\it x}{\cal D}{\it bl}}

\def\Mlax{{\cal L}{\it x}{\cal M}{\it on}}

\mathrmdef{Ob}

\def\Cat{\mbox{{$\cal C$}\it at}}

\newcommand{\ps}[1]{\makebox[0pt]{$#1$}}

\newcommand{\ov}[1]{{\bar{#1}}}
\newcommand{\ovv}[1]{{\tilde{#1}}}

\newcommand{\todd}[2]{\xymatrix@1{\ar[r]|\bb^{#1}_{#2}&}}
\newcommand{\todo}[1]{\xymatrix@1{\ar[r]|\bb^{#1}&}}
\newcommand{\todu}[1]{\xymatrix@1{\ar[r]|\bb_{#1}&}}
\newcommand{\todol}[1]{\xymatrix@1@C=40pt{\ar[r]|\bb^{#1}&}}

\newcommand{\tod}{\xymatrix@1{\ar[r]|\bb&}}

\newcommand{\toc}{\xymatrix@1{\ar[r]|\cc&}}

\newtheoremrm{definition}{Definition}
\newtheorem{theorem}{Theorem}
\newtheorem{proposition}{Proposition}

\newtheorem{corollary}{Corollary}
\newtheoremrm{remark}{Remark}
\newtheoremrm{remarks}{Remarks}
\newtheoremrm{example}{Example}
\newtheoremrm{notation}{Notation}
\newtheoremrm{``proof''}{``Proof''}
\newtheoremrm{note}{Note}
\newtheoremrm{examples}{Examples}

\input xy
\xyoption{all}

\mathrmdef{Hom}
\mathrmdef[colim]{\varinjlim\,}

\mathbfdef{Set}

\def\SSpan{{\mathbb S}{\rm pan}}

\def\CCat{{\mathbb C}{\rm at}}
\def\Cosp{{\mathbb C}{\rm osp}}

\def\Cat{{\cal C}{\it at}}

\def\H{{\ic H}}

\newbox\bbox
\setbox\bbox = \hbox to 0pt{\hss $\scriptstyle\bullet$\hss}
\def\bb{\usebox{\bbox}}

\newbox\cbox
\setbox\cbox = \hbox to 0pt{\hss $\circ$\hss}
\def\cc{\usebox{\cbox}}

\begin{document}

\maketitle

\begin{abstract} We show how the notion of intercategory encompasses a wide variety of three-dimensional structures from the literature, notably duoidal categories, monoidal double categories, cubical bicategories, double bicategories and Gray categories. Variations on the notion of span provide further examples of interest, an important one being the intercategory of sets. We consider the three kinds of morphism of intercategory as well as the cells binding them with applications to the above structures. In particular hom functors are studied.

\end{abstract}

\tableofcontents

\section*{Introduction}

In this paper we propose intercategories, introduced in \cite{ICI}, as a conceptual framework for three-dimensional category theory. Many notions of three-dimensional category already appear in the literature, each with its own use, and no doubt many more will appear as the theory develops. Some of the more established ones which we discuss here are duoidal categories \cite{AM, BCZ, BS}, monoidal double categories \cite{Sh}, cubical bicategories \cite{GaGu}, Verity's double bicategories \cite{V} and Gray categories \cite{Gray}. As will be seen below, these can all be considered as special intercategories. The range and variety of examples is evidence of the unifying role intercategories can play. Not only do they provide an effective organization and unification of a large number of three-dimensional structures from the literature, but by putting these in a common setting it is possible to consider morphisms between them and study how they relate to each other. The more encompassing context often suggests useful generalizations which are more natural and actually come up in practice.

There are many ways of looking at intercategories each with its own intuition, complementing and augmenting the others. At the most basic level an intercategory is a laxified triple category, having three kinds of arrows, three kinds of cells relating these in pairs, and cubes relating these. We haven't striven for the most general laxity possible, but rather a specific choice, informed by the examples. This will become apparent as the paper progresses.

Intercategories are a natural extension of double categories to the next dimension. Just as double categories can be conceptually understood as two categories with the same objects, intercategories can be thought of as two double categories with a common horizontal category. Of course, the two double categories are related in more than this shared horizontal structure. There are three-dimensional cells, called cubes, and this is where the laxity comes in. This point of view is exploited in Section \ref{spans} where various double categories of spans or cospans interact.

A different connection to double categories, explored in Section \ref{Verity}, is that intercategories may be thought of as double categories whose horizontal {\em and} vertical compositions are weak, with an added transversal direction which is strict and used to express the coherence conditions. Weak double categories are strict in the horizontal direction; yet there are many cases where both compositions are weak, e.g. quintets in a bicategory. Trying to formalize this as a $2$-dimensional structure leads to a vicious circle.

One of the motivating examples for our definition of intercategory was $2$-monoidal or duoidal categories \cite{AM, BCZ, BS}. These are categories with two monoidal structures in which one of the tensors and its unit are monoidal with respect to the other. It is tempting to try to place this in the context of double categories, i.e. a double category with one object. Even if we allow weak composition in both directions, this doesn't work. It is only at the level of intercategories that it does. So we can think of an intercategory as a ``duoidal category with several objects''. This point of view is studied in Section \ref{duoidal}.

The main reason for defining $2$-monoidal categories in \cite{AM} was to study their morphisms and thus put some order in the large number of structures arising in Hopf algebra theory. There are three types of morphism corresponding to our lax-lax, colax-lax, and colax-colax functors. The ``intercategory as two double categories'' point of view would suggest four types, but the lax-colax ones (lax in the horizontal direction and colax in the vertical) don't come up. In fact they don't make sense! Having three types of morphism suggests  that there is a triple category of intercategories, which is indeed the case. There are appropriate $2$-cells and commutative cubes all fitting together nicely in a strict triple category. This is perhaps the main theorem of \cite{ICI}.

This suggests a further point of view, namely that the triple category of intercategories is a universe for doing higher category theory. Thus in Section \ref{Set} we study ``hom functors'' for intercategories and single out one particular intercategory where they take their values, which we believe can rightfully be called ${\ic Set}$, the intercategory of sets. 

Conspicuously lacking in our examples are tricategories \cite{GPS}. Intercategories are a competing notion. They are simpler because the associativity and unit constraints are isomorphisms rather than equivalences. The added freedom of having three different kinds of morphism allows many interesting structures, normally viewed as forming tricategories to be arranged into intercategories. The equivalences appear as a result of requiring globularity. This is the point of view put forth by Garner and Gurski in \cite{GaGu} with their cubical bicategories and Shulman in \cite{Sh} with monoidal double categories, both examples of intercategories. 

Each of the examples mentioned above is studied in detail below. The theory of intercategories suggests natural generalizations, supported by examples. We also introduce general constructions such as spans or quintets producing new intercategories. There are furthermore interesting morphisms of various kinds between all of the above and cells between these. All of this takes place inside the intercategory of intercategories, $ {\ic ICat} $, and is the basis of our claim of the unification of three-dimensional category theory.

\section{Preliminaries} \label{prelim}

In \cite{ICI} we introduced intercategories and their morphisms and exposed their basic properties. We gave three equivalent presentations. The first as pseudocategories in the $2$-category $\Dlax$ of weak double categories, lax functors and horizontal transformations. The second as pseudocategories in $ \Dcolax $, the $2$-category of weak double categories, colax functors and horizontal transformations. In fact, it is better, as far as morphisms are concerned, to consider these as horizontal (and, respectively, vertical) pseudocategories in $ \Doub $, the strict double category of weak double categories, lax functors, colax functors and their cells. These presentations are short and clear and an obvious generalization of duoidal categories, but a more intuitive presentation is as a double pseudocategory in $ \Cat  $. A pseudocategory in $ \Cat $ is a weak double category, so this presentation shows an intercategory as two double categories sharing a common horizontal structure. This is like thinking of a double category as two categories with the same objects. Of course the two structures are related, which is where interchange appears.

At a more basic level, an intercategory $ {\ic A} $ has a class of objects, and three kinds of arrows, horizontal, vertical and transversal each with their own composition ($ \circ$, $\bullet$, $\cdot $) and identities ($\id$, $\Id$, $1$, resp.). These are related in pairs by double cells as depicted in the diagram
$$
\bfig\scalefactor{.8}

\square(0,300)/@{>}|{\cc}`@{>}|{\bb}``/[A ` B `\ov{A} `;h ` v ` `]

\square(300,0)/@{>}|{\cc}`@{>}|{\bb}`@{>}|{\bb}`@{>}|{\cc}/[A' ` B' `\ov{A'}`\ov{B'}; h' ` v' `w' `\ov{h'}]

\morphism(0,800)|b|<300,-300>[A` A';f]

\morphism(500,800)<300,-300>[B ` B';g]

\morphism(0,300)|b|<300,-300>[\ov{A}`\ov{A'};\ov{f}]

\place(550,250)[\scriptstyle \alpha']

\place(150,400)[\scriptstyle \psi]

\place(400,650)[\scriptstyle \phi]

\efig
$$
Here the $ h $, $ h' $, $\ov{h'} $ are horizontal, $ v $, $ v'$, $ w' $ are vertical and $ f $, $ \ov{f} $, $ g $ transversal. Cells whose boundaries are horizontal and transversal, such as $ \phi $ above, are called {\em horizontal}, those whose boundaries are vertical and transversal, such as $ \psi $, are called {\em vertical}, and those like $ \alpha' $ with horizontal and vertical boundaries are {\em basic}. Each of the three types of cells has two compositions like in a double category. In fact horizontal (resp.\ vertical) cells are the double cells of a weak double category. The fundamental unit of structure is the {\em cube}, as depicted above. Cubes have three compositions: horizontal, vertical and transversal. Transversal composition is strictly associative and unitary, giving four transversal categories. Horizontal and vertical composition are associative and unitary up to coherent transversal isomorphism.

The first feature of intercategories which distinguishes them from what one might imagine a weak triple category would be, is that both horizontal and vertical composition are bicategorical in nature, rather than having one of the composites associative and unitary up to equivalence as for tricategories. So in this sense they are a stricter notion. But in another sense they are laxer. The interchange law for basic cells doesn't hold. Instead there is a comparison, the interchanger
$$
 \chi : (\alpha \circ \beta) \bullet (\ov{\alpha}\circ \ov{\beta}) \to (\alpha \bullet \ov{\alpha}) \circ (\beta \bullet \ov{\beta}) .
$$
$ \chi $ is a {\em special cube} meaning a cube whose horizontal and vertical faces are transversal identities. There will be many examples given below. The two-dimensional notation
$$
\chi : \frac{\alpha | \beta}{\ov{\alpha} | \ov{\beta}}     \to \left.\dfrac{\alpha}{\ov{\alpha}}\right|\dfrac{\beta}{\ov{\beta}}
$$
is often used. In it the variables (cells) don't change place. There are also {\em degenerate interchangers}:  
$$  
\mu : \dfrac{\id_v}{\id_{\ov{v}}} \to \id_{\frac{v}{\ov{v}}} 
$$
$$   
\delta : \Id_{h|h'} \to \Id_h | \Id_{h'}
$$
$$      
\tau : \Id_{\id_A} \to \id_{\Id_A}
$$
All these satisfy a number of coherence conditions, which can be found in \cite{ICI}.

If all interchangers are identities as well as the associativity and unit isomorphisms, we have a triple category. If they are all isomorphisms we talk of {\em weak triple category}. A case of special importance is when the $ \delta $, $ \mu $, $ \tau$ are identities while the $ \chi $ is allowed to be arbitrary. We call this a chiral triple category. It will play a central role in \cite{GP2}.

There are three general types of morphism of intercategory. They all  preserve the transversal structure on the nose. In the horizontal and vertical directions they can be lax or colax. We can have laxity in both directions, which we call lax-lax morphisms. Similarly there are colax-colax morphisms. The colax-lax morphisms are colax in the horizontal direction and lax in the vertical. The lax-colax doesn't come up. In fact the obvious coherence conditions produce diagrams in which none of the arrows compose.

\section{Duoidal categories}\label{duoidal}

\subsection{Duoidal categories as intercategories}\label{duoidal.1} 

\ 

Duoidal categories were introduced in \cite{AM} under the name of $2$-monoidal categories as a generalization of braided monoidal categories and motivated by various kinds of morphisms between these.

The classical Eckmann-Hilton argument says that a monoid in the category of monoids is a commutative monoid and we might think then that a pseudomonoid in the $2$-category of monoidal categories and strong monoidal functors could be, for similar reasons, a symmetric monoidal category. This is not quite true. What emerges is the important notion of  braided monoidal category as exposed in the now classical paper \cite{JS}.

If instead we consider pseudomonoids in the $2$-category of monoidal categories and lax monoidal functors we get categories equipped with two tensor products related by interchange morphisms. These morphisms express the fact that the second tensor is given by a lax functor with respect to the first, but could equally well be understood as saying that the first tensor is colax with respect to the second in a way that reminds us of the definition of bialgebra. This is the notion of {\em duoidal category} (or {\em $2$-monoidal category}).

Duoidal categories have been studied (apart from {\em loc.\ cit.}) in \cite{BS,BCZ}, where many examples are given.

Our notion of intercategory is partly modeled after this, so it will be no surprise that duoidal categories can be considered as special intercategories just as monoidal categories can be viewed as one-object bicategories. However, it is perhaps not in the first way one might try. One could think that, as a monoidal category is a one-object bicategory, what we've got are two related bicategories sharing the same objects which are perhaps part of a Verity double category. This would translate into an intercategory with one object and identity transversal arrows. The horizontal and vertical arrows would be the objects of our duoidal category with horizontal and vertical composition given by the two tensors. A general cube might look something like
$$
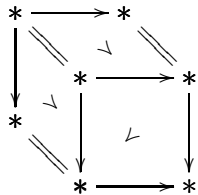
\bfig\scalefactor{.7}

\square(0,300)/>`>``/[*`*`*`;` ` `]

\square(300,0)/>`>`>`>/[*`*`*`*; ```]

\morphism(0,800)|b|/=/<300,-300>[*`*;]

\morphism(500,800)/=/<300,-300>[* ` *;]

\morphism(0,300)|b|/=/<300,-300>[*`*;]


\morphism(600,300)/=>/<-90,-90>[`;]

\morphism(100,450)/=>/<90,-90>[`;]

\morphism(350,700)/=>/<90,-90>[`;]

\efig
$$
This doesn't work. 

Definition 2.2 of \cite{ICI} says that an intercategory is a pseudocategory in $ \Dlax $, the $2$-category of weak double categories with lax functors and horizontal transformations. A monoidal category may be considered as a weak double category with one object and one horizontal morphism, the identity. Then lax functors are lax monoidal functors and horizontal transformations are monoidal natural transformations. So we have a full sub $2$-category $\Mlax $ of $ \Dlax $. Also, a pseudomonoid is a pseudocategory whose object of objects is $ {\mathbbm 1} $. In this way a duoidal category $ {\bf D} $, which is a pseudomonoid in $ \Mlax $, can be considered as a special intercategory. It will have one object, only identity horizontal, vertical and transversal arrows, and the horizontal and vertical cells are also identities. The only nontrivial parts are the basic cells which are the objects of $ {\bf D} $ and the cubes which are its morphisms. A general cube will look like
$$
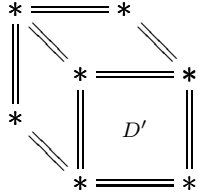
\bfig\scalefactor{.7}

\square(0,300)/=`=``/[*`*`*`;` ` `]

\square(300,0)/=`=`=`=/[*`*`*`*; ```]

\morphism(0,800)|b|/=/<300,-300>[*`*;]

\morphism(500,800)/=/<300,-300>[* ` *;]

\morphism(0,300)|b|/=/<300,-300>[*`*;]

\place(550,250)[\scriptstyle D']

\morphism(100,450)/=/<70,-70>[`;]

\morphism(400,700)/=/<70,-70>[`;]

\efig
$$ 
with a morphism of $ {\bf D} $, $ d : D \to D' $, in it. The first tensor gives horizontal composition and the second tensor, the vertical.

As a double pseudocategory in $ {\cal CAT} $, ${\ic D} $ can be described by a diagram as in Section 3 of \cite{ICI}:

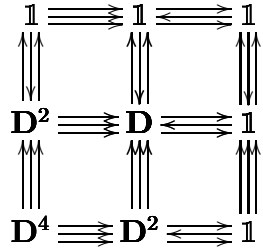
$$
\bfig\scalefactor{.7}

\square/>`<-`<-`>/[{\bf D}^2 `{\bf D}`{\bf D}^4`{\bf D}^2;```]

\square|allb|/@{>}@<-3pt>`@{<-}@<-3pt>`@{<-}@<-3pt>`@{>}@<-3pt>/[{\bf D}^2 `{\bf D}`{\bf D}^4`{\bf D}^2;```]

\square|allb|/@{>}@<3pt>`@{<-}@<3pt>`@{<-}@<3pt>`@{>}@<3pt>/[{\bf D}^2 `{\bf D}`{\bf D}^4`{\bf D}^2;```]

\square(0,500)/>`>`>`>/[{\mathbbm 1}`{\mathbbm 1}`{\bf D}^2`{\bf D};```]

\square(0,500)|allb|/@{>}@<-3pt>`@{<-}@<-3pt>`@{<-}@<-3pt>`@{>}@<-3pt>/[{\mathbbm 1}`{\mathbbm 1}`{\bf D}^2`{\bf D};```]

\square(0,500)|allb|/@{>}@<3pt>`@{<-}@<3pt>`@{<-}@<3pt>`@{>}@<3pt>/[{\mathbbm 1}`{\mathbbm 1}`{\bf D}^2`{\bf D};```]

\square(500,0)/<-`<-`<-`<-/[{\bf D} `{\mathbbm 1} `{\bf D}^2 `{\mathbbm 1};```]

\square(500,0)|allb|/@{>}@<-3pt>`@{<-}@<-3pt>`@{<-}@<-3pt>`@{>}@<-3pt>/[{\bf D} `{\mathbbm 1} `{\bf D}^2 `{\mathbbm 1};```]

\square(500,0)|allb|/@{>}@<3pt>`@{<-}@<3pt>`@{<-}@<3pt>`@{>}@<3pt>/[{\bf D} `{\mathbbm 1} `{\bf D}^2 `{\mathbbm 1};```]

\square(500,500)/<-`>`>`<-/[{\mathbbm 1} `{\mathbbm 1} `{\bf D} `{\mathbbm 1};```]

\square(500,500)|allb|/@{>}@<-3pt>`@{<-}@<-3pt>`@{<-}@<-3pt>`@{>}@<-3pt>/[{\mathbbm 1} `{\mathbbm 1} `{\bf D} `{\mathbbm 1};```]

\square(500,500)|allb|/@{>}@<3pt>`@{<-}@<3pt>`@{<-}@<3pt>`@{>}@<3pt>/[{\mathbbm 1} `{\mathbbm 1} `{\bf D} `{\mathbbm 1};```]

\efig
$$

Furthermore, the bilax, double lax and double colax morphisms of \cite{AM} correspond to our colax-lax, lax-lax and colax-colax functors.

A ready supply of duoidal categories can be gotten from monoidal categories ($ {\bf V}, \otimes, I$) with finite products. Indeed, ($ {\bf V}, \times, 1, \otimes, I$) is Example 6.19 of \cite{AM}. (Note however that, contrary to {\it loc.\ cit.}, we list the horizontal structure, product here, first.) No coherence between $ \otimes $ and product is assumed. Dually if $ {\bf V} $ has finite coproducts, then $ ({\bf V}, \otimes, I, +, 0) $ is a duoidal category. In particular, for any category $ {\bf A} $ with finite products and coproducts, we get a duoidal category $ ({\bf A}, \times, 1, +, 0) $.

\subsection{Matrices in a monoidal category}\label{duoidal.2}

\ 

A closely related intercategory is the following. Let $ {\bf V} $ be a monoidal category with coproducts preserved by $ \otimes $ in each variable, and with pullbacks. We construct an intercategory $ {\ic SM}({\bf V})$ whose objects are sets, whose transversal arrows are functions, whose horizontal arrows are spans, and whose vertical arrows are matrices of $ {\bf V} $ objects. Specifically, a vertical arrow $ A \tod B $ is an $ A \times B $ matrix $ [V_{ab}] $ of objects $ V_{ab} $ of $ {\bf V} $. Horizontal cells are span morphisms and vertical cells are matrices of morphisms of $ {\bf V} $. A basic cell is a span of matrices
$$
\bfig\scalefactor{.9}

\square/<-`@{>}|{\bb}`@{>}|{\bb}`<-/[A`S`B`T;\sigma_0`{[}V_{ab}{]}`{[}W_{st}{]}`\tau_0]

\morphism(180,200)|a|/<=/<180,0>[`;\scriptstyle {[}f_{st}{]}]

\square(500,0)/>``@{>}|{\bb}`>/[S`A'`T`B';\sigma_1 ``{[}V'_{a',b'}{]}`\tau_1]

\morphism(720,200)|a|/=>/<180,0>[`;\scriptstyle {[}g_{st}{]}]

\efig
$$
where
$$
V_{\sigma_0 s, \tau_0 t}  \to/<-/^{f_{st}} W_{st} \to^{g_{st}} V'_{\sigma_1 s, \tau_1 t}
$$
are morphisms of $ {\bf V} $. A general cube
$$
\bfig\scalefactor{.7}

\square(0,300)/<-`@{>}|{\bb}``/[A`S`B`;```]

\square(300,0)/<-`@{>}|{\bb}`@{>}|{\bb}`<-/[C`R`D`U;```]

\morphism(0,800)|b|<300,-300>[A`C;]

\morphism(500,800)<300,-300>[S` R;]

\morphism(0,300)|b|<300,-300>[B`D;]

\morphism(500,800)<500,0>[S`A';]

\morphism(1000,800)<300,-300>[A'`C';]

\square(800,0)/>`@{>}|{\bb}`@{>}|{\bb}`>/[R`C'`U`D';```]

\morphism(500,250)/<=/<150,0>[`;]

\morphism(1000,250)/=>/<150,0>[`;]

\morphism(120,400)/=>/<100,-100>[`;]

\morphism(500,750)/-/<0,-200>[`;]

\morphism(1000,750)/-/<0,-200>[`;]

\efig
$$
is a morphism of spans of matrices, i.e.
$$
\bfig\scalefactor{.7}

\square/>`@{>}|{\bb}`@{>}|{\bb}`>/[S`R`T`U;`{[}W_{st}{]}`{[}X_{ru}{]}`]

\morphism(200,250)/=>/<150,0>[`;]

\efig
$$
forming two commutative cubical diagrams.

In Section \ref{spans} we shall give a general construction showing, in particular, that this is indeed an intercategory. Unless $ \otimes $ preserves pullback, the interchanger $ \chi $ is not invertible. The identities are as follows. The horizontal identity $ \id_{[V_{ab}]} $ is
$$
\bfig\scalefactor{.7}

\square/<-`@{>}|{\bb}`@{>}|{\bb}`<-/[A`A`B`B;1`{[}V_{ab}{]}`{[}V_{ab}{]}`1]

\morphism(180,200)|a|/<=/<180,0>[`;1]

\square(500,0)/>``@{>}|{\bb}`>/[A`A`B`B;1 ``{[}V_{ab}{]}`1]

\morphism(720,200)|a|/=>/<180,0>[`;1]

\efig
$$
These compose vertically so
$$
\mu : \frac{\id_{{[}V_{ab}{]}}}{\id_{{[}W_{bc}{]}}}\ \to \id_{ {[}V_{ab}{]} \otimes {[}W_{bc}{]} }
$$
is equality.

\noindent The vertical identity $ \Id_S $ is
$$
\bfig\scalefactor{.7}

\square/<-`@{>}|{\bb}`@{>}|{\bb}`<-/[A`S`A`S;`\Id_A`\Id_S`]

\morphism(180,240)|a|/<=/<180,0>[`;]

\square(500,0)/>``@{>}|{\bb}`>/[S`A'`S`A';``\Id_{A'}`]

\morphism(700,240)|a|/=>/<180,0>[`;]

\efig
$$
where $ \Id_X : X \tod X $ is given by
$$
(\Id_X)_{x,x'} = \left\{ \begin{array}{ll}
   I & \mbox{if\ \ } x = x'\\
  0 & \mbox{otherwise}
\end{array}
\right.
$$
and for $ f : X \to Y $
$$
\bfig\scalefactor{.7}

\square/>`@{>}|{\bb}`@{>}|{\bb}`>/[X`Y`X`Y;f`\Id_X `\Id_Y`f]

\morphism(200,200)|a|/=>/<150,0>[`;\Id_f]

\efig
$$
is given by
$$
(\Id_f)_{x,x'} = \left\{ \begin{array}{ll}
    1_I : I \to I   & \mbox{if\ \ } x = x' \\
    ! :\ \  0 \to I & \mbox{if\ \ } x \neq x' \mbox{\ \ and\ \ } fx = fx' \\
    1_0 : 0 \to 0 & \mbox{if\ \ } fx \neq fx'.
\end{array}
\right.
$$

The horizontal composition $ \Id_S | \Id_{S'} $ will usually involve the pullback
$$
\bfig\scalefactor{.7}

\square[0\times_I 0 ` 0 `0 `I;```]

\efig
$$
and unless this is $ 0 $ (i.e.\ $ 0 \to I $ is mono), $ \delta : \Id_{S \otimes S'} \to \Id_{S} | \Id_{S'} $ will not be invertible. Finally $ \tau : \Id_{\id_A} \to \id_{\Id_A} $ is always the identity.

By contrast, all of the interchangers $ \chi, \mu, \delta, \tau $ are generally not invertible for the cartesian product/tensor duoidal category of Subsection \ref{duoidal.1}.

\subsection{Embedding the duoidal category of $ {\bf V} $ into matrices}\label{duoidal.3}

\ 

If $ {\bf V} $ has a terminal object, we can embed $ ({\bf V}, \times, \otimes) $, considered as an intercategory, into $ {\ic SM}({\bf V}) $ as follows. A basic cell of $ {\bf V} $ is embedded as
$$
\bfig\scalefactor{.7}

\square/=`=`=`=/[*`*`*`*;```]

\place(250,250)[V]

\place(700,250)[\longmapsto]

\square(1000,0)/<-`@{>}|{\bb}`@{>}|{\bb}`<-/[1`1`1`1;`{[}1{]}`{[}V{]}`]

\morphism(1180,240)|a|/<=/<180,0>[`;]

\square(1500,0)/>``@{>}|{\bb}`>/[1`1`1`1;``{[}1{]}`]

\morphism(1680,240)|a|/=>/<180,0>[`;]

\efig
$$
The extension in the transversal direction is obvious. $ {\bf V} $ cannot be a strict subintercategory of $ {\ic SM}({\bf V}) $ as this would imply that $ \mu $ and $ \tau $ for $ {\bf V} $ are identities.  Indeed, the vertical arrow $ [1] : 1 \tod 1 $ is the $ 1 \times 1 $ matrix whose sole entry is $ 1 $, the terminal object of $ {\bf V} $, whereas the vertical identity is the one whose entry is $ I $, the unit for $ \otimes$. What we have is an inclusion $ F : {\bf V} \to {\ic SM}({\bf V}) $ which is strong in the horizontal direction and lax in the vertical direction. So it can be considered as a lax-lax or a colax-lax morphism. 

There is also a morphism in the opposite direction, $ G : {\ic SM}({\bf V}) \to {\bf V} $, taking a basic cell
$$
\bfig\scalefactor{.7}

\square/<-`@{>}|{\bb}`@{>}|{\bb}`<-/[A`S`B`T;`{[}U_{ab}{]}`{[}V_{st}{]}`]

\morphism(180,240)|a|/<=/<180,0>[`;]

\square(500,0)/>``@{>}|{\bb}`>/[S`A'`T`B';``{[}U'_{a',b'}{]}`]

\morphism(700,240)|a|/=>/<180,0>[`;]

\efig
$$ 
to the cell
$$
\bfig\scalefactor{.8}

\square/=`=`=`=/[*`*`*`*;```]

\place(250,250)[\scriptstyle\sum_{s,t} V_{st}]

\efig
$$
with the obvious extension in the transversal direction. $ G $ is colax-colax. For example, the identity structure morphisms are
$$
G(\id_{[U_{ab}]}) = \sum_{a,b} U_{a,b} \to 1
$$
and 
$$ G(\Id_S) = \nabla : \sum_S I \to I 
$$
the codiagonal.
$ G $ is left adjoint to $ F $ in the following sense. $ F $ may be considered as a lax-lax morphism or a colax-lax morphism, i.e.\ as a horizontal or a transversal arrow in $\ICat $, the triple category of intercategories. As transversal arrows are generally better let's consider $ F $ as such. Then $ F $ and $ G $ are horizontal and vertical arrows in the double category of transversal and vertical arrows of $ \ICat $, i.e.\ in $ {\mathbb P}{\rm s}{\mathbb C}{\rm at} (\Dcolax) $; moreover, $ F $ and $ G $ are conjoint arrows in the latter.

To see this we need double cells
$$
\bfig\scalefactor{.7}

\square/>`=`>`=/[{\bf V} `{\ic SM}({\bf V})`{\bf V} `{\bf V};F``G`]

\place(250,250)[\scriptstyle \alpha]

\square(1200,0)/=`>`=`>/<600,500>[{\ic SM}({\bf V})`{\ic SM}({\bf V})`{\bf V} `{\ic SM}({\bf V});`G``F]

\place(1500,250)[\scriptstyle \beta]

\efig
$$
satisfying the ``triangle equalities''. Such double cells take objects, horizontal and vertical arrows, and basic cells to transversal arrows, horizontal and vertical cells, and cubes respectively. $ GF $ is the identity on all elements and $ \alpha : GF \to \id \cdot \Id $ is taken to be the appropriate identity.

The various components of $ \beta : \Id \cdot \id \to F \cdot G $ can be read off from its action on a basic cell
$$
\bfig\scalefactor{.8}

\square/<-`@{>}|{\bb}`@{>}|{\bb}`<-/[A`S`B`T;`{[}U_{ab}{]}`{[}V_{st}{]}`]

\morphism(180,240)|a|/<=/<180,0>[`;]

\square(500,0)/>``@{>}|{\bb}`>/[S`A'`T`B';``{[}U'_{a',b'}{]}`]

\morphism(700,240)|a|/=>/<180,0>[`;]

\efig
$$ 
which produces the cube
$$
\bfig\scalefactor{.8}

\square(0,300)/<-`@{>}|{\bb}``/[A`S`B`;`{[}U_{ab}{]}``]

\square(300,0)/<-`@{>}|{\bb}`@{>}|{\bb}`<-/[1`1`1`1;`{[}1{]}`{[}\sum_{st} V_{st}{]}`]

\morphism(0,800)|b|<300,-300>[A`1;]

\morphism(500,800)<300,-300>[S` 1;]

\morphism(0,300)|b|<300,-300>[B`1;]

\morphism(500,800)<500,0>[S`A';]

\morphism(1000,800)<300,-300>[A'`1;]

\square(800,0)/>`@{>}|{\bb}`@{>}|{\bb}`>/[1`1`1`1;``{[}1{]}`]



\morphism(500,800)/-/<0,-250>[`;]

\morphism(1000,800)/-/<0,-250>[`;]

\efig
$$
where the middle cell is given by the coproduct injections
$$
\bfig\scalefactor{.7}

\square/>`@{>}|{\bb}`@{>}|{\bb}`>/[S`1`T`1;`{[}V_{st}{]} `{[}\sum_{s,t} V_{st}{]}`]

\place(250,250)[\scriptstyle {[}j_{st}{]}]

\efig
$$
(All the other morphisms are uniquely determined.)

Checking that $ \alpha $ and $ \beta $ are double cells and that they satisfy the conjoint equations is straightforward and omitted.





\section{Monoidal double categories and cubical bicategories}\label{MonDbl}

\subsection{Monoidal double categories}\label{MonDbl.1} 

\ 

In \cite{Sh}, Shulman uses a notion of monoidal double category to construct monoidal bicategories. The notion of monoidal double category is simpler because the coherence morphisms are isomorphisms rather than equivalences, which makes the coherence conditions much easier. In {\em loc.\ cit.} many examples are given building a strong case for the point of view that the seemingly more complicated notion of double category is in fact simpler than that of bicategory.

A {\em monoidal double category} \cite{Sh} is a pseudomonoid in the $2$-category $\DblSt $ of (weak) double categories with pseudo functors and horizontal transformations:
$$
\otimes : {\mathbb D} \times {\mathbb D} \to {\mathbb D},
$$
$$
I : {\mathbbm 1} \to {\mathbb D}.
$$

As $ \DblSt $ is a sub $2$-category of $\Dlax $ (and $\Dcolax$) and intercategories involve only pullbacks of strict double functors, which are in $ \DblSt $, it follows that a weak category object in $ \DblSt $ is also one in $ \Dlax $ (and $\Dcolax $), i.e.\ an intercategory. It is one in which the interchangers $ \chi $, $ \mu $, $\delta $, $\tau $ are all isomorphisms. So a monoidal double category is an intercategory of the form
$$
{\mathbb D} \times {\mathbb D}  \threepppp/>`>`>/<400>^{} |{} _{} {\mathbb D} \three/>`<-`>/^{} | {}_{} {\mathbbm 1}
$$
with strong interchangers (isomorphisms).

It is an intercategory with one object, one transversal arrow, one vertical arrow and one vertical cell, all identities of course.

Furthermore interchange holds up to isomorphism. As a double pseudocategory in $ {\cal CAT} $, it looks like

$$
\bfig\scalefactor{.7}

\square/>`<-`<-`>/[{\bf D}^2_1 `{\bf D}_1 `{\bf D}^2_2 `{\bf D}_2;```]

\square|allb|/@{>}@<-3pt>`@{<-}@<-3pt>`@{<-}@<-3pt>`@{>}@<-3pt>/[{\bf D}^2_1 `{\bf D}_1 `{\bf D}^2_2 `{\bf D}_2;```]

\square|allb|/@{>}@<3pt>`@{<-}@<3pt>`@{<-}@<3pt>`@{>}@<3pt>/[{\bf D}^2_1 `{\bf D}_1 `{\bf D}^2_2 `{\bf D}_2;```]

\square(0,500)/>`>`>`>/[{\bf D}^2_0 `{\bf D}_0 `{\bf D}^2_1 `{\bf D}_1;```]

\square(0,500)|allb|/@{>}@<-3pt>`@{<-}@<-3pt>`@{<-}@<-3pt>`@{>}@<-3pt>/[{\bf D}^2_0 `{\bf D}_0 `{\bf D}^2_1 `{\bf D}_1;```]

\square(0,500)|allb|/@{>}@<3pt>`@{<-}@<3pt>`@{<-}@<3pt>`@{>}@<3pt>/[{\bf D}^2_0 `{\bf D}_0 `{\bf D}^2_1 `{\bf D}_1;```]

\square(500,0)/<-`<-`<-`<-/[{\bf D}_1`{\mathbbm 1}`{\bf D}_2 `{\mathbbm 1};```]

\square(500,0)|allb|/@{>}@<-3pt>`@{<-}@<-3pt>`@{<-}@<-3pt>`@{>}@<-3pt>/[{\bf D}_1`{\mathbbm 1}`{\bf D}_2 `{\mathbbm 1};```]

\square(500,0)|allb|/@{>}@<3pt>`@{<-}@<3pt>`@{<-}@<3pt>`@{>}@<3pt>/[{\bf D}_1`{\mathbbm 1}`{\bf D}_2 `{\mathbbm 1};```]

\square(500,500)/<-`>`>`<-/[{\bf D}_0 `{\mathbbm 1} `{\bf D}_1 `{\mathbbm 1};```]

\square(500,500)|allb|/@{>}@<-3pt>`@{<-}@<-3pt>`@{<-}@<-3pt>`@{>}@<-3pt>/[{\bf D}_0 `{\mathbbm 1} `{\bf D}_1 `{\mathbbm 1};```]

\square(500,500)|allb|/@{>}@<3pt>`@{<-}@<3pt>`@{<-}@<3pt>`@{>}@<3pt>/[{\bf D}_0 `{\mathbbm 1} `{\bf D}_1 `{\mathbbm 1};```]

\efig
$$

The conditions (iv) of {\em loc.\ cit.\ }correspond to (24), (26), (25) in \cite[Sect. 4]{ICI}, conditions (v) to (27), (28), and conditions (vi) to (31), (30), (29), (32). Our conditions (21), (22), (23) don't appear there because the structural isomorphisms of the double category $ {\mathbb D} $ were treated as identities.

A monoidal double category can equally well be viewed as an intercategory with one object, one transversal arrow, one horizontal arrow and one horizontal cell by using the inverse interchangers. The $ 3\times 3 $ diagram of categories would then be the transpose of the above.

\subsection{Horizontal and vertical monoidal double categories}\label{MonDbl.2}

\ 

In the present context, it seems natural to generalize Shulman's notion of monoidal double category by removing the restriction that the interchangers be isomorphisms. We then get two distinct generalizations of monoidal double category corresponding to the cases just mentioned. One in which $ \otimes : {\mathbb D} \times {\mathbb D} \to {\mathbb D} $ and $ I : {\mathbbm 1} \to {\mathbb D} $ are lax, which we call {\em horizontal monoidal double category}, and the other where $ \otimes $ and $ I $ are colax, which we call {\em vertical}.

Let us examine this in more detail. For notational convenience we look at vertical monoidal double categories. That $ \otimes $ and $ I $ are colax means that we have comparison cells

$$
\bfig\scalefactor{.8}

\square/``@{>}|{\bb}`=/<500,350>[`\ov{A} \otimes \ov{B} ` \ovv{A} \otimes \ovv{B} ` \ovv{A} \otimes \ovv{B};``\ov{v}\otimes \ov{w}`]

\square(0,350)/=``@{>}|{\bb}`/<500,350>[A\otimes B ` A\otimes B``\ov{A}\otimes \ov{B};``v\otimes w`]

\square/=`@{>}|{\bb}``=/<500,700>[A \otimes B `A\otimes B`\ovv{A}\otimes\ovv{B} `\ovv{A} \otimes \ovv{B};`(v\bullet \ov{v})\otimes (w \bullet \ov{w})``]

\place(250,350)[\scriptstyle \chi]

\square(1300,150)/=`@{>}|{\bb}`@{>}|{\bb}`=/[A \otimes B ` A \otimes B ` A \otimes B `A\otimes B;`\id_A \otimes \id_B `\id_{A \otimes B} `]

\place(1550,400)[\scriptstyle \mu]

\efig
$$
$$
\bfig\scalefactor{.8}

\square/``@{>}|{\bb}`=/<500,350>[`I` I` I;``I_\id`]

\square(0,350)/=``@{>}|{\bb}`/<500,350>[I`I``I;``I_\id`]

\square/=`@{>}|{\bb}``=/<500,700>[I `I `I`I;`I_\id``]

\place(250,350)[\scriptstyle \delta]

\square(1300,150)/=`@{>}|{\bb}`@{>}|{\bb}`=/[I`I`I`I;`I_\id `\id_I`]

\place(1550,400)[\scriptstyle \tau]

\efig
$$
satisfying the conditions (21)-(32) of Section 3 in \cite{ICI}.

This definition encompasses duoidal categories. Another example is a double category with a lax choice of finite products, as in \cite{GP99}.

\subsection{Endomorphisms in an intercategory}\label{MonDbl.3}

\ 

Just like the set of endomorphisms of an object in a category has a monoid structure, if we fix an object $ A $ of an intercategory $ {\ic A} $ we get two monoidal double categories of endomorphisms, a horizontal one $ {\mathbb H}{\mathbb E}{\rm nd} (A) $ and a vertical one $ {\mathbb V}{\mathbb E}{\rm nd} (A) $ (or $ {\mathbb H}{\mathbb E}{\rm nd}_{\ic A} (A) $ and $ {\mathbb V}{\mathbb E}{\rm nd}_{\ic A} (A)$ if there are several intercategories). As an intercategory, $ {\ic HEnd} (A) $ has the same structure as $ {\ic A} $ except that we only consider the one object $ A $ as well as only the identity transversal arrow $ 1_A $, the identity vertical arrow $ \Id_A $, and the identity vertical cell $1_{\Id_A} $. So a general cube would be an $ {\ic A} $ cube that looks like 
$$
\bfig\scalefactor{.7}

\square(0,300)/>`=``/[A ` A `A `;f ` ` `]

\square(300,0)/>`=`=`>/[A ` A`A `A; f' `  ` `g']

\morphism(0,800)|b|/=/<300,-300>[A `A;]

\morphism(500,800)/=/<300,-300>[A` A;]

\morphism(0,300)|b|/=/<300,-300>[A`A;]

\place(550,250)[\scriptstyle \phi']


\place(400,650)[\scriptstyle \alpha]

\morphism(100,450)/=/<70,-70>[`;]

\efig
$$
As a monoidal double category, $ {\mathbb H}{\mathbb E}{\rm nd}(A) $ has objects the horizontal endomorphisms of $ {\ic A} $, horizontal arrows the horizontal cells, and vertical arrows the basic cells. The tensor product is given by horizontal composition. This indeed gives us what we are calling a horizontal monoidal double category. It will only be a monoidal double category in the sense of \cite{Sh} if the interchangers $ \chi, \delta, \mu, \tau $ are isomorphisms when restricted to basic cells of the form
$$
\bfig\scalefactor{.7}

\square/>`=`=`>/[A`A`A`A;```]

\place(250,250)[\scriptstyle \phi]

\efig
$$

The construction of $ {\mathbb V}{\mathbb E}{\rm nd}(A) $ is dual, and considers only cubes of the form
$$
\bfig\scalefactor{.7}

\square(0,300)/=`@{>}|{\bb}``/[A ` A `A `; `v ` `]

\square(300,0)|arrb|/=`@{>}|{\bb}`@{>}|{\bb}`=/[A ` A`A `A;  `v'  `w' `]

\morphism(0,800)|b|/=/<300,-300>[A `A;]

\morphism(500,800)/=/<300,-300>[A` A;]

\morphism(0,300)|b|/=/<300,-300>[A`A;]

\place(550,250)[\scriptstyle \phi']

\place(150,400)[\scriptstyle \beta]

\morphism(350,700)/=/<70,-70>[`;]


\efig
$$
This produces a vertical monoidal double category.

\subsection{Matrices in a duoidal category}\label{MonDbl.4}

\ 

We outline an interesting example of a horizontal monoidal double category constructed from a duoidal category $ ({\bf D}, \otimes, I, \boxtimes, J) $ having coproducts over which $ \boxtimes $ distributes. The double category $ {\mathbb D} $ has sets as objects, functions as horizontal arrows, matrices of $ {\bf D}$-objects as vertical arrows and matrices of $ {\bf D}$-morphisms as cells. Vertical composition is given by matrix multiplication using $ \boxtimes $, and vertical identities $ \Id $ are ``scalar matrices'' with $ J $ on the diagonal.  This is what we called $ {\bf V}$-${\mathbb S}{\rm et} $ in \cite{P} with $ {\bf V} = ({\bf D}, \boxtimes, J) $.

The tensor product $ \otimes : {\mathbb D} \times {\mathbb D} \to {\mathbb D} $ is cartesian product on objects (sets) and horizontal arrows (functions). For vertical arrows, it is defined pointwise using the $ \otimes $ of $ {\bf D} $
$$
\bfig\scalefactor{.7}

\square/`@{>}|{\bb}`@{>}|{\bb}`/<400,500>[A`A'`B`B';`{[}V_{ab}{]}`{[}V'_{a'b'}{]}`]

\morphism(850,250)/|->/<300,0>[`;\otimes]

\square(1500,0)|arrb|/`@{>}|{\bb}``/<400,500>[A \times A'``B \times B'`;`{[}V_{ab}\otimes V'_{a'b'}{]}``]

\efig
$$
with the obvious extension to cells. The unit for $ \otimes $ is the $ 1 \times 1 $ matrix with entry $ I $.

The laxity morphisms of $ \otimes $ are as follows. Suppose $ [W_{bc}] : B \tod C $ and $ [W'_{b'c'}] : B' \tod C' $ are two more vertical arrows of $ {\mathbb D} $. Then
$$
\chi : (V \otimes V') \bullet (W \otimes W') \to (V \bullet W) \otimes (V' \bullet W')
$$
has as its $ (a, a')$, $(c,c')$ component the composite
$$
\sum_{(b,b')} (V_{ab} \otimes V'_{a'b'}) \boxtimes (W_{bc} \otimes W'_{b'c'}) \to^{\sum_{(b,b')} \chi}
$$
$$
 \sum_{(b,b')} (V_{ab} \boxtimes W_{bc}) \otimes (V'_{a'b'} \boxtimes W'_{b'c'}) \to^{[j_b \otimes j_{b'}]}
$$
$$
(\sum_b V_{ab} \boxtimes W_{bc}) \otimes (\sum_{b'} V'_{a'b'} \boxtimes W'_{b'c'})
$$
where $ j_b $, $ j_{b'} $ are coproduct injections.

The $ (a,b)$, $(a',b') $ component of
$$
\delta : \Id_{A \times B} \to \Id_A \otimes \Id_B
$$
is given by
$$
\begin{array}{lll}
 \delta : J \to J\otimes J & if & a = a', b = b'\\
  ! : 0 \to J \otimes 0     & if & a = a', b \neq b'\\
  ! : 0 \to 0 \otimes J     & if & a \neq a', b = b'\\
  ! : 0 \to 0 \otimes 0    & if & a \neq a', b \neq b'
\end{array}
$$
The laxity morphisms for $ I : {\mathbbm 1} \to {\mathbb D} $ are given by
$$
\mu : I \boxtimes I \to I
$$
and
$$
\tau : J \to I.
$$
The routine calculations showing that we actually get a horizontal monoidal double category are omitted.

\subsection{Locally cubical bicategories}\label{MonDbl.5}

\ 

A multiobject version of monoidal double categories is given by Garner and Gurski's {\em locally cubical bicategories} \cite{GaGu}. These are categories weakly enriched in the monoidal (cartesian) $2$-category $\DblSt$. So a class $ \mbox{Ob}{\cal A} $ of objects is given, and for each pair $ A, B \in \mbox{Ob}{\cal A} $ a weak double category $ {\cal A}(A,B) $. For each object $ A $ there is given a strong functor
$$
\Id_A : {\mathbbm 1} \to {\cal A}(A,A)
$$
and for any three objects $ A, B, C $, a strong functor
$$
\bullet : {\cal A} (A,B) \times {\cal A} (B,C) \to {\cal A} (A,C).
$$
This composition is unitary and associative up to coherent isomorphism (see {\em loc.\ cit.\ }for details).

One can get a good feel for this structure by considering the {\em category ${\bf StDbl}$} of strict double categories and strict functors. These are category objects and their functors in $ {\bf Cat} $ and so form a cartesian closed category. That is, for any two double categories $ {\mathbb A} $ and ${\mathbb B} $ we have a double category $ {\mathbb B}^{\mathbb A} $ of morphisms from $ {\mathbb A} $ to $ {\mathbb B} $. One can easily work out what $ {\mathbb B}^{\mathbb A} $ looks like. Its objects are strict functors, its horizontal arrows are the horizontal transformations we have been using, its vertical arrows are the dual notion of vertical transformation, and its cells are modifications. This makes ${\bf StDbl}$ into a category enriched in itself, i.\ e.\ a strict locally cubical bicategory.

Returning to the non strict case, we can combine the whole structure into a pseudocategory in $ \DblSt $:
$$
\sum_{A,B,C} {\cal A}(A,B) \times {\cal A} (B,C) \threepppp/>`>`>/<400>^{p_1} |{\bullet} _{p_2} \sum_{A,B}{\cal A}(A,B) \three/>`<-`>/^{\partial_0} | {\id}_{\partial_1} \mbox{Ob} {\cal A}
$$
where $ \mbox{Ob}{\cal A} $ is a discrete double category.

Thus we see that a locally cubical bicategory is an intercategory in which the only transversal and vertical arrows are identities as well as vertical cells, and for which the interchangers are isomorphisms.

A general cell might be pictured as
$$
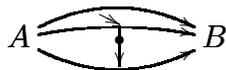
\bfig\scalefactor{.9}

\morphism(0,0)|a|/{@{>}@/^1em/}/<700,0>[A`B;]

\morphism(0,0)|a|/{@{>}@/^.37em/}/<700,0>[A`B;]

\morphism(0,0)|b|/{@{>}@/_1em/}/<700,0>[A `B;]

\morphism(260,100)/>/<100,-50>[`;]

\morphism(360,90)/@{>}|{\bb}/<0,-190>[`;]

\efig
$$
(with a double cell of $ {\cal A}(A,B) $ inside it) which looks more like a cube if we insert the horizontal and vertical identities
$$
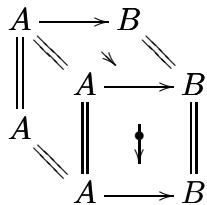
\bfig\scalefactor{.7}

\square(0,300)/>`=``/[A ` B `A `; `  ` `]

\square(300,0)/>`=`=`>/[A ` B `A `B;  `  ``]

\morphism(0,800)|b|/=/<300,-300>[A` A;]

\morphism(500,800)/=/<300,-300>[B ` B;]

\morphism(0,300)|b|/=/<300,-300>[A`A;]

\morphism(100,450)/=/<70,-70>[`;]

\morphism(550,400)/@{>}|{\bb}/<0,-250>[`;]

\morphism(340,700)/>/<100,-100>[`;]

\efig
$$

In a locally cubical bicategory, the composition and identity operations are pseudofunctors, but in our discussion above we consider them as lax morphisms. We could of course extend the definition of locally cubical bicategory to lax composition and identity, and there is no problem getting an intercategory this way as it is just by definition. The cubes look the same. It is just the interchangers that are not invertible now. However, we don't have any good examples and so have not developed this further.

We could also view composition and identity as colax functors, and get a transposed representation of a locally cubical bicategory as an intercategory in which cubes look like
$$
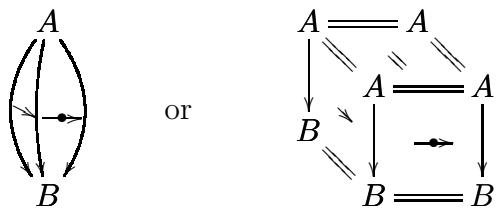
\bfig\scalefactor{.7}

\morphism(0,800)/{@{>}@/_1.2em/}/<0,-800>[A`B;]

\morphism(0,800)/{@{>}@/^-.37em/}/<0,-800>[A`B;]

\morphism(0,800)/{@{>}@/^1.2em/}/<0,-800>[A`B;]

\morphism(-195,440)/>/<140,-70>[`;]

\morphism(-60,365)/@{>}|{\bb}/<250,0>[`;]

\place(600,400)[\mbox{or}]

\square(1200,300)/=`>``/[A ` A `B `; `  ` `]

\square(1500,0)/=`>`>`=/[A ` A `B `B;  `  ``]

\morphism(1200,800)|b|/=/<300,-300>[A` A;]

\morphism(1700,800)/=/<300,-300>[A ` A;]

\morphism(1200,300)|b|/=/<300,-300>[B`B;]

\morphism(1300,450)/>/<100,-100>[`;]

\morphism(1650,250)/@{>}|{\bb}/<250,0>[`;]

\morphism(1540,700)/=/<100,-100>[`;]

\efig
$$
and again we can relax the conditions on composition and identity to being merely colax.

In view of this it is tempting, as the referee has suggested, to replace the basic cells and cubes by quintets and get a nicer, more symmetric, representation of a locally cubical bicategory as an intercategory. A cube would look like
$$
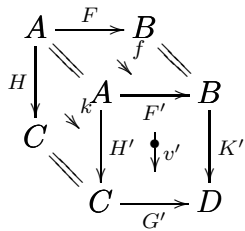
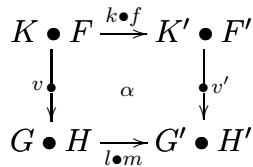
\bfig\scalefactor{.7}

\square(0,300)/>`>``/[A ` B `C `; F`H  ` `]

\square(300,0)|brrb|/>`>`>`>/[A ` B `C `D;F'`H' `K'`G']

\morphism(0,800)|b|/=/<300,-300>[A` A;]

\morphism(500,800)/=/<300,-300>[B ` B;]

\morphism(0,300)|b|/=/<300,-300>[C`C;]

\morphism(100,450)/>/<100,-100>[`;k]

\morphism(550,400)|r|/@{>}|{\bb}/<0,-250>[`;v']

\morphism(340,700)/>/<100,-100>[`;f]











\efig
$$
with a cell
$$
\bfig\scalefactor{.7}

\square/>`@{>}|{\bb}`@{>}|{\bb}`>/<700,500>[K \bullet F `K' \bullet F' `G \bullet H `G' \bullet H';k \bullet f `v ` v' ` l \bullet m]

\place(350,250)[\scriptstyle \alpha]

\efig
$$
This seems to work although the details, which are formidable, have not been completely checked and they don't appear to follow from the general machinery we have developed so far. This will have to await further work.

\section{Verity double bicategories}\label{Verity}

\subsection{Double bicategories}\label{Verity.1}

\  

Double bicategories are, at least in part, an answer to the problem of making double categories weak in both directions. For example, we could take quintets in a bicategory $ {\cal B} $. This structure has the same objects as $ {\cal B} $ with horizontal and vertical arrows the arrows of $ {\cal B} $ and with double cells diagrams
$$
\bfig\scalefactor{.7}

\square[A`B`C`D;f`h`k`g]

\morphism(320,290)/=>/<-100,-100>[`;t]

\efig
$$
where $ t : k f \to g h $ is a $2$-cell. Such cells can be pasted horizontally and vertically, and everything works well (including interchange) except that neither horizontal nor vertical composition is associative or unitary on the nose.

A simpler example is the transpose of a weak double category, where horizontal and vertical are interchanged. This is a useful duality for strict double categories but is not available for weak ones.

Attempts at a direct definition of double categories, weak in both directions, just lead to vicious circles. The problem lies with the special cells used in the coherence conditions for the definition of weak double category. These are cells whose vertical domains and codomains are horizontal identities, but if these identities are not strict identities, then horizontal composition of special cells would require the use of vertical special cells, and now the same problem arises. The resolution is achieved by formalizing special cells. This is done by giving as extra structure, cells between arrows whose domains (and codomains) are the {\em same}, i.e.\ {\em globular cells} as well as the double ones. Although the special cells involved in the definition of weak double category are all isomorphisms, non invertible ones come up in the definition of lax and colax functor.

We sketch Verity's definition of double bicategory. The reader is referred to \cite{V} for details. Section 3.2 of \cite{VM} also gives a very readable account.

To start with we are given two bicategories $ {\cal H} $ and $ {\cal V} $ which share the same class of objects $ A $ and then we are given a class of squares $ {\cal S} $ with boundaries
$$
\bfig\scalefactor{.7}

\square/>`@{>}|{\bb}`@{>}|{\bb}`>/[a`a'`\ov{a}`\ov{a'};h `v`v'`\ov{h}]

\place(250,250)[\scriptstyle \sigma]

\efig
$$
$ h$, $ \ov{h} $ arrows of $ {\cal H} $ and $ v $, $ v' $ arrows of $ {\cal V} $. There are furthermore left and right actions of the $2$-cells of $ {\cal V} $ on the $ \sigma $ and top and bottom actions of those of $ {\cal H} $ on them as well, e.g.
$$
\bfig\scalefactor{.7}

\square/`@{>}|{\bb}`@{>}|{\bb}`>/[a`a'`\ov{a}`\ov{a'};`v`v'`\ov{h}]

\place(250,200)[\scriptstyle \sigma]

\place(850,250)[\longmapsto]

\square(1200,0)/>`@{>}|{\bb}`@{>}|{\bb}`>/[a`a'`\ov{a}`\ov{a'};h' `v `v'`\ov{h}]

\place(1450,250)[\scriptstyle \alpha *_V \sigma]

\morphism(0,500)|a|/{@{>}@/^1em/}/<500,0>[a`a';h']

\morphism(0,500)|b|/{@{>}@/_1em/}/<500,0>[a `a';h]

\morphism(250,600)|l|/=>/<0,-180>[`;\alpha]

\efig
$$
These four actions commute (strictly). Finally the squares can be pasted horizontally and vertically. Horizontal and vertical pasting is associative and unitary once the structural isomorphisms of $ {\cal H} $ (or ${\cal V}$) are factored in so as to make domains and codomains agree. The interchange law for squares holds strictly.

We already have the beginning of an intercategory
$$
\bfig\scalefactor{.7}

\square(0,500)/>```/<700,500>[{\bf H} \times_A {\bf H} ` {\bf H}``{\bf S};```]

\square(0,500)|allb|/@{>}@<-3pt>```/<700,500>[{\bf H} \times_A {\bf H} ` {\bf H}``{\bf S};```]

\square(0,500)|allb|/@{>}@<3pt>```/<700,500>[{\bf H} \times_A {\bf H} ` {\bf H}``{\bf S};```]

\square(700,0)/``<-`/<700,500>[{\bf S} `{\bf V}``{\bf V} \times_A {\bf V};```]

\square(700,0)|allb|/@{>}@<-3pt>``@{<-}@<-3pt>`/<700,500>[{\bf S} `{\bf V}``{\bf V} \times_A {\bf V};```]

\square(700,0)|allb|/@{>}@<3pt>``@{<-}@<3pt>`/<700,500>[{\bf S} `{\bf V}``{\bf V} \times_A {\bf V};```]

\square(700,500)/<-``>`/<700,500>[{\bf H} `A `{\bf S} `{\bf V};```]

\square(700,500)|allb|/@{>}@<-3pt>`@{<-}@<-3pt>`@{<-}@<-3pt>`@{>}@<-3pt>/<700,500>[{\bf H} `A `{\bf S} `{\bf V};```]

\square(700,500)|allb|/@{>}@<3pt>`@{<-}@<3pt>`@{<-}@<3pt>`@{>}@<3pt>/<700,500>[{\bf H} `A `{\bf S} `{\bf V};```]

\efig
$$
$ {\bf H} $ (resp.\ $ {\bf V} $) is the category of arrows and $2$-cells of $ {\cal H} $ (resp.\ ${\cal V}$). $ {\bf S} $ is the category whose objects are squares with morphisms described below.

Given a double bicategory ($ A, {\cal H}, {\cal V}, {\cal S}, \ldots $) we construct an intercategory $ {\ic D} $ as follows.

\noindent (1) The objects are the elements of $ A $.

\noindent (2) The transversal arrows are identities.

\noindent (3) The horizontal (vertical) arrows are the arrows of $ {\cal H} $ (resp.\ $ {\cal V} $).

\noindent (4) The horizontal (vertical) cells are the $2$-cells of $ {\cal H} $ (resp.\ $ {\cal V} $).

\noindent (5) The basic cells are the elements of $ {\cal S} $.

\noindent (6) There is a single cube with boundary as below
$$
\bfig\scalefactor{.7}

\square(0,300)/>`@{>}|{\bb}`@{>}|{\bb}`>/[a`a'`\ov{a}`\ov{a'};```]

\square(300,0)/``@{>}|{\bb}`>/[`a'`\ov{a}`\ov{a'};```]

\morphism(500,800)/=/<300,-300>[a'`a';]

\morphism(0,300)|b|/=/<300,-300>[\ov{a}`\ov{a};]

\morphism(500,300)/=/<300,-300>[\ov{a'}`\ov{a'};]

\place(250,550)[\scriptstyle \sigma]

\place(400,150)[\scriptstyle \ov{\alpha}]

\place(650,400)[\scriptstyle \ov{\beta'}]

\square(1400,300)/>`@{>}|{\bb}``/[a `a'`\ov{a}`;```]

\square(1700,0)/>`@{>}|{\bb}`@{>}|{\bb}`>/[a `a' `\ov{a} `\ov{a'};```]

\morphism(1400,800)|b|/=/<300,-300>[a `a;]

\morphism(1900,800)/=/<300,-300>[a'`a';]

\morphism(1400,300)|b|/=/<300,-300>[\ov{a}`\ov{a};]

\place(1950,250)[\scriptstyle \sigma']

\place(1550,400)[\scriptstyle \beta]

\place(1800,650)[\scriptstyle \alpha]

\efig
$$
if
$$
(\sigma *_H \beta')\star_V \ov{\alpha} = \alpha *_V (\beta *_H \sigma'),
$$
otherwise there are none.

This last condition tells us what the morphisms of $ {\bf S} $ are: a morphism $ \sigma \to \sigma' $ is a quadruple ($ \alpha, \beta, \ov{\alpha}, \beta'$) as above.

The interchangers $ \chi $, $ \delta $, $ \mu $, $ \tau $ are all identities.

Apart from the fact that transversal arrows are all identities, there is a more important special feature of intercategories $ {\ic D} $ arising in this way:
$$
\bfig\scalefactor{.7}

\Atriangle/>`>`/[{\bf S} ` {\bf V} \times {\bf H} ` {\bf V} \times {\bf H};(\partial_1, \partial_1) ` (\partial_0, \partial_0)`]

\efig
$$
is a discrete bifibration. Let's call this property $ (*) $. It implies, in particular, that every horizontal and every vertical cell has a basic companion and conjoint. It also implies that the interchangers are identities.

\begin{theorem} There is a natural correspondence between double bicategories and intercategories satisfying $(*) $ and whose transversal arrows are identities.

\end{theorem}

\subsection{Double categories as intercategories}\label{Verity.2}

\ 

One thing this example gives us is a different way of looking at intercategories. They are a weakening of double categories so as to allow both horizontal and vertical composition to be bicategorical in nature. And thus it gives us a preferred way to consider a double category as an intercategory.

Let $ {\mathbb A} $ be a weak double category. Horizontal composition is strictly associative and unitary whereas vertical composition is so only up to coherent isomorphism. This is reflected in the fact that morphisms of double categories can be lax, colax, strong or strict in the vertical direction but are always required to be strict in the horizontal direction. To encode this in a Verity double bicategory we take $ {\cal H} $ to be the locally discrete bicategory (i.e.\ just the category) of objects and horizontal arrows of $ {\mathbb A} $. For $ {\cal V} $ we take the bicategory of objects, vertical arrows and special (= globular) cells of $ {\mathbb A} $. The class of squares $ {\cal S} $ is the class of all double cells of $ {\mathbb A} $. We now turn this into an intercategory $ {\ic I}({\mathbb A}) $ thus placing it in the same environment (i.\ e.\ the triple category ${\ic ICat}$) as the other examples. Thus we have for $ {\ic I}({\mathbb A}) $

\noindent -- objects are those of $ {\mathbb A} $

\noindent -- transversal arrows are identities

\noindent -- horizontal arrows are those of $ {\mathbb A} $

\noindent -- vertical arrows are those of $ {\mathbb A} $

\noindent -- horizontal cells are identities

\noindent -- vertical cells are special cells of $ {\mathbb A} $

\noindent -- basic cells are the double cells of $ {\mathbb A} $

\noindent -- cubes are commutative cylinders $ \phi' \alpha = \beta \phi $

$$
\bfig\scalefactor{.7}

\square(0,300)/>`@{>}|{\bb}``/[A ` B `\ov{A} `;f ` v ` `]

\square(300,0)|brrb|/>`@{>}|{\bb}`@{>}|{\bb}`>/[A ` B `\ov{A} `\ov{B}; f` v' `w' `g]

\morphism(0,800)/=/<300,-300>[A` A;]

\morphism(500,800)/=/<300,-300>[B ` B;]

\morphism(0,300)|b|/=/<300,-300>[\ov{A}`\ov{A};]

\place(550,250)[\scriptstyle \phi']

\place(150,400)[\scriptstyle \alpha]

\place(330,630)[\scriptstyle \phi]

\place(570,600)[\scriptstyle \beta]

\morphism(500,790)/-/<0,-250>[`;]

\efig
$$

The three kinds of morphisms of intercategory $ {\ic I}({\mathbb A}) \to {\ic I}({\mathbb B}) $, lax-lax, colax-lax, colax-colax, correspond respectively to lax, lax, colax functors $ {\mathbb A} \to {\mathbb B} $.

Of course there are other ways of considering a double category as an intercategory. The two examples mentioned at the beginning of the section, quintets in a bicategory and the transpose of a weak double category, require the horizontal cells to be horizontally special as well. We would take the $ {\cal H} $ to be the bicategory of objects, horizontal arrows and horizontally special cells of $ {\mathbb A} $, with $ {\cal V} $ and ${\cal S} $ the same as above. The new intercategory $ {\ic I'} ({\mathbb A}) $ will have cubes that involve six cells
$$
\bfig\scalefactor{.7}

\square(0,300)/>`@{>}|{\bb}`@{>}|{\bb}`>/[A ` B `\ov{A} `\ov{B}; f` v`v' `g]

\square(300,0)/``@{>}|{\bb}`>/[` B `\ov{A}`\ov{B};` `w' `g']

\morphism(500,800)/=/<300,-300>[B ` B;]

\morphism(0,300)|b|/=/<300,-300>[\ov{A}`\ov{A};]

\morphism(500,300)/=/<300,-300>[\ov{B} `\ov{B};]

\place(250,550)[\scriptstyle \phi]

\place(400,150)[\scriptstyle \ov{\theta}]

\place(650,400)[\scriptstyle \beta]

\square(1500,300)/>`@{>}|{\bb}``/[A ` B `\ov{A} `;f ` v ` `]

\square(1800,0)/>`@{>}|{\bb}`@{>}|{\bb}`>/[A ` B `\ov{A} `\ov{B}; f' ` v' `w' `g']

\morphism(1500,800)|b|/=/<300,-300>[A` A;]

\morphism(2000,800)/=/<300,-300>[B ` B;]

\morphism(1500,300)|b|/=/<300,-300>[\ov{A}`\ov{A};]

\place(2050,250)[\scriptstyle \phi']

\place(1650,400)[\scriptstyle \alpha]

\place(1900,650)[\scriptstyle \theta]

\efig
$$
making the cube commute, i.e.
$$
\bfig\scalefactor{.7}

\square/``@{>}|{\bb}`=/[`\ov{A} `\ov{A}`\ov{A};```]

\square(0,500)/=``@{>}|{\bb}`/[A`A``\ov{A};```]

\square(0,0)/=`@{>}|{\bb}``=/<500,1000>[A`A`\ov{A}`\ov{A};`v``]

\place(250,500)[\scriptstyle \lambda^{-1}]

\square(500,0)/>`@{>}|{\bb}`@{>}|{\bb}`>/[\ov{A}`\ov{B}`\ov{A}`\ov{B};`\Id`\Id`]

\place(750,250)[\scriptstyle \theta]

\square(500,500)/>`@{>}|{\bb}`@{>}|{\bb}`>/[A`B`\ov{A}`\ov{B};`v`v'`]

\place(750,750)[\scriptstyle \phi]

\square(1000,0)/=``@{>}|{\bb}`=/[\ov{B}`\ov{B}`\ov{B}`\ov{B};``\Id`]

\place(1250,250)[\scriptstyle 1]

\square(1000,500)/=``@{>}|{\bb}`/[B`B`\ov{B}`\ov{B};``w'`]

\place(1250,750)[\scriptstyle \beta]

\square(1500,0)/```=/[\ov{B}``\ov{B}`\ov{B};```]

\square(1500,500)/=`@{>}|{\bb}``/[B`B`\ov{B}`;```]

\square(1500,0)/=``@{>}|{\bb}`=/<500,1000>[B`B`\ov{B}`\ov{B};``w'`]

\place(1750,500)[\scriptstyle \lambda]

\place(2200,500)[=]

\square(2400,0)/``@{>}|{\bb}`=/[`A`\ov{A}`\ov{A};``v`]

\square(2400,500)/=``@{>}|{\bb}`/[A`A``A;``\Id`]

\square(2400,0)/=`@{>}|{\bb}``=/<500,1000>[A`A`\ov{A}`\ov{A};`v``]

\place(2650,500)[\scriptstyle \rho^{-1}]

\square(2900,0)/=``@{>}|{\bb}`=/[A`A`\ov{A}`\ov{A};``v'`]

\place(3150,250)[\scriptstyle \alpha]

\square(2900,500)/=``@{>}|{\bb}`=/[A`A`A`A;``\Id`]

\place(3150,750)[\scriptstyle 1]

\square(3400,0)/>``@{>}|{\bb}`>/[A`B`\ov{A}`\ov{B};``w'`]

\place(3650,250)[\scriptstyle \phi']

\square(3400,500)/>``@{>}|{\bb}`/[A`B`A`B;``\Id`]

\place(3650,750)[\scriptstyle \theta]

\square(3900,0)/`@{>}|{\bb}``=/[B``\ov{B}`\ov{B};```]

\square(3900,500)/=```/[B`B`\ov{B}`;```]

\square(3900,0)/=``>`=/<500,1000>[B`B`\ov{B}`\ov{B};``w'`]

\place(4150,500)[\scriptstyle \rho]

\efig
$$

We will not check the tedious though straightforward details showing that this is indeed an intercategory.

The above example suggests the following generalization which does not, however, arise as a Verity double bicategory. From a weak double category $ {\mathbb A} $ we construct an intercategory $ {\ic I''}({\mathbb A}) $ which is like $ {\ic I'}({\mathbb A}) $ except that we allow its transversal morphisms to be horizontal arrows of $ {\mathbb A} $. So a general cube will look like
$$
\bfig\scalefactor{.7}

\square(0,300)/>`@{>}|{\bb}`@{>}|{\bb}`>/[A ` B `\ov{A} `\ov{B}; f` v`v' `g]

\square(300,0)/``@{>}|{\bb}`>/[` B' `\ov{A'}`\ov{B'};` `w' `]

\morphism(500,800)|r|/>/<300,-300>[B ` B';t]

\morphism(0,300)|b|/>/<300,-300>[\ov{A}`\ov{A'};\ov{g}]

\morphism(500,300)|a|/>/<300,-300>[\ov{B} `\ov{B'};\ov{t}]

\place(250,550)[\scriptstyle \phi]

\place(400,150)[\scriptstyle \ov{\theta}]

\place(650,400)[\scriptstyle \beta]

\square(1500,300)/>`@{>}|{\bb}``/[A ` B `\ov{A} `;f ` v ` `]

\square(1800,0)|blrb|/>`@{>}|{\bb}`@{>}|{\bb}`>/[A' ` B' `\ov{A'} `\ov{B'}; f' ` ` `g']

\morphism(1500,800)|a|/>/<300,-300>[A` A';s]

\morphism(2000,800)|a|/>/<300,-300>[B ` B';t]

\morphism(1500,300)|b|/>/<300,-300>[\ov{A}`\ov{A'};\ov{g}]

\place(2050,250)[\scriptstyle \phi']

\place(1650,400)[\scriptstyle \alpha]

\place(1900,650)[\scriptstyle \theta]

\efig
$$
such that

$$
\bfig\scalefactor{.7}

\square/``@{>}|{\bb}`=/[`\ov{A} `\ov{A}`\ov{A};``\Id`]

\square(0,500)/=``@{>}|{\bb}`/[A`A``\ov{A};``v`]

\square(0,0)/=`@{>}|{\bb}``=/<500,1000>[A`A`\ov{A}`\ov{A};`v``]

\place(250,500)[\scriptstyle \lambda^{-1}]

\square(500,0)/>`@{>}|{\bb}``>/[\ov{A}`\ov{B}`\ov{A}`\ov{A'};```]


\square(500,500)/>`@{>}|{\bb}`@{>}|{\bb}`>/[A`B`\ov{A}`\ov{B};```]

\place(750,750)[\scriptstyle \phi]

\square(1000,0)/>``@{>}|{\bb}`>/[\ov{B}`\ov{B'}`\ov{A'}`\ov{B'};``\Id`]

\place(1000,250)[\scriptstyle \ov{\theta}]


\square(1000,500)/>`@{>}|{\bb}`@{>}|{\bb}`>/[B`B'`\ov{B}`\ov{B'};``w'`]

\place(1250,750)[\scriptstyle \beta]

\square(1500,0)/```=/[\ov{B'}``\ov{B'}`\ov{B'};```]

\square(1500,500)/=`@{>}|{\bb}``/[B'`B'`\ov{B'}`;```]

\square(1500,0)/=``@{>}|{\bb}`=/<500,1000>[B'`B'`\ov{B'}`\ov{B'};``w'`]

\place(1750,500)[\scriptstyle \lambda]

\place(2200,500)[=]

\square(2400,0)/``@{>}|{\bb}`=/[`A`\ov{A}`\ov{A};``v`]

\square(2400,500)/=``@{>}|{\bb}`/[A`A``A;``\Id`]

\square(2400,0)/=`@{>}|{\bb}``=/<500,1000>[A`A`\ov{A}`\ov{A};`v``]

\place(2650,500)[\scriptstyle \rho^{-1}]

\square(2900,0)/>`@{>}|{\bb}`@{>}|{\bb}`>/[A`A'`\ov{A}`\ov{A'};```]

\place(3150,250)[\scriptstyle \alpha]

\square(2900,500)/>`@{>}``>/[A`B`A`A';```]


\square(3400,0)/>``@{>}|{\bb}`>/[A'`B'`\ov{A'}`\ov{B'};``w'`]

\place(3650,250)[\scriptstyle \phi']

\square(3400,500)/>``@{>}|{\bb}`>/[B`B'`A'`B';``\Id`]

\place(3400,750)[\scriptstyle \theta]

\square(3900,0)/`@{>}|{\bb}``=/[B'``\ov{B'}`\ov{B'};```]

\square(3900,500)/=`@{>}|{\bb}``/[B'`B'`B'`;```]

\square(3900,0)/=``@{>}|{\bb}`=/<500,1000>[B'`B'`\ov{B'}`\ov{B'};``w'`]

\place(4150,500)[\scriptstyle \rho]

\efig
$$

\subsection{Quintets in a double category}\label{Verity.3}

\ 

We end this section with a somewhat dual construction to the previous one, quintets in a double category. A weak double category may be thought of as a bicategory (vertically) with some extra arrows (horizontal) which serve to rigidify it in a sense. The quintet construction for bicategories mentioned above can be performed on an arbitrary weak double category.

Let $ {\mathbb A} $ be a weak double category. The intercategory of quintets, $ {\ic Q}({\mathbb A}) $, has the following:

\noindent -- objects, those of $ {\mathbb A} $

\noindent -- transversal arrows, the horizontal arrows of $ {\mathbb A} $

\noindent --  horizontal and vertical arrows, the vertical arrows of $ {\mathbb A} $

\noindent -- horizontal and vertical cells, the double cells of $ {\mathbb A} $

\noindent -- basic cells
$$
\bfig\scalefactor{.6}

\square/@{>}|{\bb}`@{>}|{\bb}`@{>}|{\bb}`@{>}|{\bb}/[A`B`\ov{A}`\ov{B};u`v`v'`\ov{u}]

\place(250,250)[\scriptstyle \phi]

\efig
$$
are quintets, i.e.\ special cells of $ {\mathbb A} $
$$
\bfig\scalefactor{.55}

\square/`@{>}|{\bb}`@{>}|{\bb}`=/[B`\ov{A}`\ov{B}`\ov{B};`v'`\ov{u}`]

\square(0,500)/=`@{>}|{\bb}`@{>}|{\bb}`/[A`A`B`\ov{A};`u`v`]

\place(250,500)[\scriptstyle \phi]

\efig
$$

\noindent -- cubes consist of cells as follows

$$
\bfig\scalefactor{.7}

\square(0,300)/@{>}|{\bb}`@{>}|{\bb}`@{>}|{\bb}`@{>}|{\bb}/[A ` B `\ov{A} `\ov{B}; u` v`w `\ov{u}]

\square(300,0)/``@{>}|{\bb}`@{>}|{\bb}/[` B' `\ov{A'}`\ov{B'};` `w' `\ov{u'}]

\morphism(500,800)|r|/>/<300,-300>[B ` B';g]

\morphism(0,300)|b|/>/<300,-300>[\ov{A}`\ov{A'};\ov{f}]

\morphism(500,300)|a|/>/<300,-300>[\ov{B} `\ov{B'};\ov{g}]

\place(250,550)[\scriptstyle \phi]

\place(400,150)[\scriptstyle \ov{\theta}]

\place(650,400)[\scriptstyle \beta]

\square(1500,300)/@{>}|{\bb}`@{>}|{\bb}``/[A ` B `\ov{A} `;u ` v ` `]

\square(1800,0)|arrb|/@{>}|{\bb}`@{>}|{\bb}`@{>}|{\bb}`@{>}|{\bb}/[A' ` B' `\ov{A'} `\ov{B'}; u' `v' `w' `\ov{u'}]

\morphism(1500,800)|a|/>/<300,-300>[A` A';f]

\morphism(2000,800)|a|/>/<300,-300>[B ` B';g]

\morphism(1500,300)|b|/>/<300,-300>[\ov{A}`\ov{A'};\ov{f}]

\place(2050,250)[\scriptstyle \phi']

\place(1650,400)[\scriptstyle \alpha]

\place(1900,650)[\scriptstyle \theta]

\efig
$$
such that

$$
\bfig\scalefactor{.7}

\square/>`@{>}|{\bb}`@{>}|{\bb}`>/[B`B'`\ov{B}`\ov{B'};`w`w'`\ov{g}]

\place(250,250)[\scriptstyle \beta]

\square(0,500)/>`@{>}|{\bb}`@{>}|{\bb}`>/[A`A'`B`B';f'`u`u'`g]

\place(250,750)[\scriptstyle \theta]

\square(500,0)/``@{>}|{\bb}`=/[B'`\ov{A'}`\ov{B'}`\ov{B'};``\ov{u'}`]

\square(500,500)/=``@{>}|{\bb}`/[A'`A'`B'`\ov{A'};``v'`]

\place(750,500)[\scriptstyle \phi']

\place(1300,500)[=]

\square(1600,0)/`@{>}|{\bb}`@{>}|{\bb}`=/[B`\ov{A}`\ov{B}`\ov{B};`w`u'`]

\square(1600,500)/=`@{>}|{\bb}`@{>}|{\bb}`/[A`A`B`\ov{A};`u`v`]

\place(1850,500)[\scriptstyle \phi]

\square(2100,0)/>`@{>}|{\bb}`@{>}|{\bb}`>/[\ov{A}`\ov{A'}`\ov{B}`\ov{B'};``\ov{u'}`\ov{g}]

\place(2350,250)[\scriptstyle \ov{\theta}]

\square(2100,500)/>``@{>}|{\bb}`>/[A`A'`\ov{A}`\ov{A'};f``v'`\ov{f}]

\place(2350,750)[\scriptstyle \alpha]

\efig
$$
Again we omit the straightforward verifications.

\subsection{Morphisms of double categories}\label{Verity.4}

\ 

In the preceding sections we gave four ways to consider a double category $ {\mathbb A} $ as an intercategory: $ {\ic I}({\mathbb A}) $, $ {\ic I'}({\mathbb A}) $, $ {\ic I''} ({\mathbb A}) $, and $ {\ic Q}({\mathbb A}) $. The first two come from considering $ {\mathbb A} $ as a special Verity double bicategory, whereas the other two are natural extensions in the intercategory context. There are other possibilities: we can switch transversal and horizontal in the first two cases and get a different embedding, or we can switch the horizontal and the vertical in all four. We could also restrict the horizontal or vertical cells to special isomorphisms. 

The referee pointed out this gives us a way of specifying how lax or colax we want to allow our morphisms to be. As mentioned in Section 4.2 above, the lax-lax, colax-lax, and colax-colax intercategory morphisms $ {\ic I} ({\mathbb A}) \to {\ic I} ({\mathbb B}) $ correspond exactly to lax, lax, and colax double functors $ {\mathbb A} \to {\mathbb B} $, respectively.

It is interesting that, while a direct definition of a double category weak in both directions doesn't work, there is no problem in defining morphisms that are weak in both directions, and these have already proved useful (see definition 6.1 of \cite{Shl2011} where they are called {\em double pseudofunctors}). In fact, one can just as easily define double functors which are lax or colax in either direction independently. Perhaps the only question is whether there should be some coherence between the horizontal and vertical structural morphisms.

Because $ {\ic I'}({\mathbb A}) $ allows non trivial horizontal cells as well as the vertical ones, we see that lax-lax, colax-lax, and colax-colax intercategory morphisms  $ {\ic I'}({\mathbb A}) \to {\ic I'} ({\mathbb B}) $ correspond to double functors $ {\mathbb A} \to {\mathbb B} $ that are lax-lax, colax-lax (colax on the horizontal arrows and lax on the vertical), colax-colax respectively. The coherence conditions are given explicitly for the lax-lax case in Section 6 of \cite{ICI}. (1) and (2) give vertically special cells expressing laxity for vertical arrows, and (5), (6) and (7) are the usual laxity coherence conditions. (3), (4, (8), (9) and (10) are the corresponding ones for horizontal laxity. (11)-(14) are the ones relating the horizontal and vertical laxity and are automatically satisfied. This is because a cube is simply a commutativity condition, so there is at most one for a given boundary. Thus any diagram of cubes will commute if the boundaries agree, and in the present case this follows by transversal functoriality. So there is no interaction between the horizontal and vertical laxity.

In the same way, the colax-lax double functors have colaxity special cells for horizontal composition satisfying the usual coherence conditions, and laxity special cells for vertical composition with no interaction between the two.

We can also define colax-colax double functors and there are double cells relating these three kinds of double functors in pairs, producing a strict triple category of weak double categories $ {\ic Doub} $, a full (at all levels) sub triple category of $ \ICat $ (Theorem 6.3 in \cite{ICI}).

We could equally well define lax-colax double functors, which weren't mentioned in the above discussion. Because the interchangers in $ {\ic I'} ({\mathbb A}) $ and $ {\ic I'} ({\mathbb B}) $ are all isomorphisms (identities, in fact) we can transpose the horizontal and vertical while inverting the interchangers to get new transposed intercategories. Then the three types of double functor described above give lax-lax, lax-colax, colax-colax double functors, with their double cells too, giving another triple category. But the two don't mix: any attempt to define double cells bounded by colax-lax and lax-colax double functors is doomed to failure!

An example of a different nature is provided by a double category $ {\mathbb A} $ in which each horizontal arrow $ f $ has a companion $ f_* $. Let $ {\ic Q'} ({\mathbb A}) $ be the intercategory of coquintets in $ {\mathbb A} $, i.e. the transpose of $ {\ic Q}({\mathbb A}) $. A basic cell in $ {\ic Q'} ({\mathbb A}) $
$$
\bfig\scalefactor{.6}

\square/@{>}|{\bb}`@{>}|{\bb}`@{>}|{\bb}`@{>}|{\bb}/[A ` B ` \ov{A} `\ov{B};u `v `v'`\ov{u}]

\place(250,250)[\scriptstyle \psi]

\efig
$$
is a special cell
$$
\bfig\scalefactor{.6}

\square/`@{>}|{\bb}`@{>}|{\bb}`=/[\ov{A} ` B `\ov{B} `\ov{B};`\ov{u} `\ov{v'} `]

\square(0,500)/=`@{>}|{\bb}`@{>}|{\bb}`/[A `A `\ov{A} `B;`v `u`]

\place(250,500)[\scriptstyle \psi]

\efig
$$
Then we get a pseudo-strict morphism of intercategories
$$
(\ )_* : {\ic I''}({\mathbb A}) \to {\ic Q'}({\mathbb A})
$$
taking a basic cell $ \phi $ to $ \phi_*$
$$
\bfig\scalefactor{.6}

\square(-200,500)/>`@{>}|{\bb}`@{>}|{\bb}`>/[A `B `\ov{A} `\ov{B};f `v `v'`\ov{f}]

\place(50,750)[\scriptstyle \phi]

\place(650,750)[\longmapsto]

\square(1000,500)/@{>}|{\bb}`@{>}|{\bb}`@{>}|{\bb}`@{>}|{\bb}/[A `B `\ov{A} `\ov{B};f_* `v `v'`\ov{f_*}]

\place(1250,750)[\scriptstyle \phi_*]

\place(1800,750)[=]

\square(2200,0)/`@{>}|{\bb}`=`=/[\ov{A} `\ov{B} `\ov{B} `\ov{B};`\ov{f_*} ``]

\square(2200,500)/>`@{>}|{\bb}`@{>}|{\bb}`>/[A `B `\ov{A} `\ov{B};f`v`v'`\ov{f}]

\square(2200,1000)/=`=`@{>}|{\bb}`/[A `A `A `B;``f_*`]

\place(2450,750)[\scriptstyle \phi]

\place(2450,250)[\lrcorner]

\place(2450,1250)[\ulcorner]

\efig
$$
where $ \ulcorner $ and $ \lrcorner $ denote the companion binding cells. The morphism $ (\ )_* $ embodies all of the functorial properties of companionship.

\section{True Gray categories}\label{Gray}

\subsection{Gray's original tensor}\label{Gray.1}

\ 

Gray categories came into prominence with the work of Gordon, Power and Street \cite{GPS} on tricategories. Whereas every bicategory is biequivalent to a $2$-category, the corresponding result for tricategories, that they be triequivalent to $3$-categories, is false as this would imply, as a special case, that every symmetric monoidal category is equivalent to a strict one. Their coherence result was that every tricategory is triequivalent to one in which everything is strict except for interchange which only holds up to isomorphism, a notion they called ``Gray category''. In fact, they introduced a monoidal structure on the category of $2$-categories which encodes this failure of interchange. The resulting monoidal category they called ${\bf Gray}$, and a Gray category is a category enriched in $ {\bf Gray} $.

As the name suggests, this was strongly influenced by a similar monoidal structure introduced by Gray in \cite{Gray}. His tensor product encodes the possibility that interchange ``hold'' only up to a comparison morphism. It arose, via adjointness, from a natural internal hom on the category of $2$-categories, which we briefly outline.

We consider the category $2\mbox{-}{\bf Cat} $ of $2$-categories and $2$-functors. Between $2$-functors we have various kinds of transformations, of which lax (natural) transformations are an important class. A lax transformation $ t : F \to G $, for $ F, G : {\cal A} \to{\cal B} $ $2$-functors, is given by

\noindent (1)  a ${\cal B}$-morphism $ t A : FA \to GA $ for each object $ A $

\noindent (2) a $2$-cell of $ {\cal B} $
$$
\bfig\scalefactor{.7}

\square[FA  ` GA  ` FA'  `GA';tA   ` Ff  ` Gf `tA']

\morphism(250,350)|r|/=>/<0,-150>[`;\ tf]

\efig
$$
for each arrow $ f : A \to A' $ in ${\cal A} $.
\noindent These satisfy well known conditions \cite{laxtransf}. Between lax transformations, there are modifications $ \mu : t \to u $ given by $2$-cells
$$
\bfig\scalefactor{.9}

\morphism(0,0)|a|/{@{>}@/^1em/}/<500,0>[FA `GA;tA]

\morphism(0,0)|b|/{@{>}@/_1em/}/<500,0>[FA ` GA;uA]

\morphism(230,90)|r|/=>/<0,-150>[`;\mu A]

\efig
$$
again satisfying obvious conditions. In this way we get a $2$-category $\mbox{Fun} ({\cal A}, {\cal B}) $, an internal hom for $2\mbox{-}{\bf Cat}$.

But it doesn't make $2\mbox{-}{\bf Cat} $ cartesian closed because composition
$$
\mbox{Fun} ({\cal A},{\cal B}) \times \mbox{Fun} ({\cal B}, {\cal C}) \to \mbox{Fun} ({\cal A}, {\cal C})
$$
isn't a $2$-functor. Composition of $2$-functors poses no problem. But for lax transformations
$$
\bfig\scalefactor{.9}

\morphism(0,0)|a|/{@{>}@/^1em/}/<500,0>[{\cal A}`{\cal B};F]

\morphism(0,0)|b|/{@{>}@/_1em/}/<500,0>[{\cal A}`{\cal B};G]

\morphism(250,100)|l|/=>/<0,-150>[`;t]

\morphism(500,0)|a|/{@{>}@/^1em/}/<500,0>[{\cal B}`{\cal C};H]

\morphism(500,0)|b|/{@{>}@/_1em/}/<500,0>[{\cal B}`{\cal C};K]

\morphism(750,100)|r|/=>/<0,-150>[`;v]

\efig
$$
we have two possible choices for $ (vt)A : HFA \to KGA$, either the top or bottom composite in
$$
\bfig\scalefactor{.7}

\square<800,500>[HFA ` KFA ` HGA `KGA;vFA `HtA `KtA `vGA]

\morphism(300,350)|r|/=>/<0,-150>[`;v(tA)]

\efig
$$
Each choice extends to a lax transformation via
$$
\bfig\scalefactor{.7}

\square<800,500>[HFA ` KFA `HFA' `KFA';vFA  `HFf `KFf `vFA']

\morphism(350,350)|r|/=>/<0,-150>[`;vFf]

\square(800,0)<800,500>[KFA `KGA `KFA'`KGA';KtA ``KGf`KtA']

\morphism(1150,350)|r|/=>/<0,-150>[`;Ktf]

\efig
$$
and
$$
\bfig\scalefactor{.7}

\square<800,500>[HFA `HGA`HFA'`HGA';HtA `HFf`HGf`HtA']

\morphism(350,350)|r|/=>/<0,-150>[`;Htf]

\square(800,0)<800,500>[HGA ` KGA`HGA'`KGA';vGA ``KGf`vGA']

\morphism(1150,350)|r|/=>/<0,-150>[`;vGf]

\efig
$$
and each of these composites is associative and unitary, and functorial with respect to modifications. But neither satisfies interchange. Whiskering, on the other hand, works well as there is no interchange involved, and the two composites come from that in the standard way. There is furthermore a comparison between the two. Clearly there is a lot of nice structure here and it is a question of organizing it properly. The key to this is Gray's tensor product, obtained from $\mbox{Fun} $ by adjointness.

We would like a $2$-category $ {\cal A} \otimes {\cal B} $ so that there is a $2$-natural bijection
$$
\frac{2\mbox{-functors}\quad  {\cal A} \otimes {\cal B} \to {\cal C}}{2\mbox{-functors}\quad  {\cal B} \to \mbox{Fun}({\cal A}, {\cal C})}
$$
Analyzing what a $2$-functor $ {\cal B} \to \mbox{Fun} ({\cal A}, {\cal C}) $ is, we get what Gray calls a quasi-functor of two variables $ H : {\cal A} \times {\cal B} \to {\cal C}$, i.e.

\noindent (1) a $2$-functor $ H(A,-) : {\cal B} \to {\cal C} $ for every $ A $ in ${\cal A}$;

\noindent (2) a $2$-functor $ H (-,B) : {\cal A} \to {\cal C} $ for every $ B $ in $ {\cal B}$;

\noindent (3) $ H(A,-)(B) = H(-,B)(A) $, written $ H(A,B) $;

\noindent (4) for every $ f : A \to A' $ and $ g : B \to B' $ a $2$-cell
$$
\bfig\scalefactor{.7}

\square<1000,500>[H(A,B)  `H(A',B)  `H(A,B')  `H(A',B');H(f,B)   `H(A,g)  `H(A',g) `H(f,B')]

\morphism(400,200)|a|/=>/<200,0>[`;h(f,g)]

\efig
$$
satisfying compatibility conditions for composition of the $f$'s (and the $g$'s). 

The term {\em cubical functor} is used in \cite{GPS} and \cite{GaGu} for quasi-functors of two variables in which the $h$'s are isomorphisms.

It is easy to imagine what $ {\cal A} \otimes {\cal B} $ is. It is the $2$-category with pairs $ (A,B) $, $A$ in $ {\cal A} $, $B$ in $ {\cal B} $, as objects, arrows generated by $ (f,B) : (A,B) \to (A',B) $ and $ (A,g) : (A,B) \to (A, B') $ subject to the equations
$$
(f',B)(f,B) = (f'f,B)
$$
$$
(A,g') (A,g) = (A, g'g)
$$
$$
(1_A,B)=1_{(A,B)} = (A,1_B).
$$
The $2$-cells are generated by those of $ {\cal A} $, those of $ {\cal B} $, and formal cells
$$
\bfig\scalefactor{.7}

\square<900,500>[(A,B)  `(A',B)  `(A,B')   `(A',B');(f,B)  `(A,g)  `(A',g)  `(f,B')]

\morphism(350,200)|a|/=>/<200,0>[`;\gamma(f,g)]

\efig
$$
subject to the expected equations. See \cite{Gray} for a more detailed description. It is of course complicated but just knowing it exists and that it gives a monoidal structure on $2\mbox{-}{\bf Cat} $ is enough. It is easier to use its universal property as classifying quasi-functors of two variables. This monoidal structure is biclosed, with $ \mbox{Fun}({\cal A},-)$ being the right adjoint to $ {\cal A} \otimes (-) $. The other adjoint is $ \mbox{Fun}^*({\cal B},-) $ given by
$$
\mbox{Fun}^* ({\cal B},{\cal C}) = \mbox{Fun}({\cal B}^{co}, {\cal C}^{co})^{co}.
$$
$\mbox{Fun}^* ({\cal B}, {\cal C}) $ has $2$-functors as objects, colax transformations as arrows, and modifications as $2$-cells.

\begin{definition} We call a category enriched in $2\mbox{-}{\bf Cat}$ with this tensor a {\em true Gray category}.

\end{definition}

Thus a true Gray category has objects, arrows ($1$-cells), $2$-cells and $3$-cells with domains and codomains like for $3$-categories. There is a strictly associative and unitary composition of arrows. $2$-cells and $3$-cells compose well inside the hom $2$-categories, but there is no horizontal composition of $2$-cells. Only whiskering on both sides by arrows, related by $3$-cells as above. This last aspect suits our purposes well as a measure of the failure of interchange. But we do need composition of $2$-cells and $3$-cells across the hom $2$-categories.

As hinted at above, there are two related ways of getting a composition, a lax and a colax one. The roots of this lie in the following result, which is essentially Gray's I.4.8 \cite{Gray}, the idea for which he credits Mac~Lane.

\begin{proposition} \label{GrayMacLane} There is a canonical bijection between the following three notions:

\noindent (a) Quasi-functors of two variables $ H : {\cal A} \times {\cal B} \to {\cal C}$, 

\noindent (b) lax functors $H^\wedge : {\cal A} \times {\cal B} \to {\cal C} $ for which the laxity morphisms

\noindent \quad (i)  $ H^\wedge (f,1) H^\wedge (f',g') \to H^\wedge (ff',g') $

\noindent \quad (ii) $ H^\wedge (f,g) H^\wedge (1,g') \to H^\wedge (f,gg') $

\noindent \quad (iii) $ 1_{H^\wedge (A,B)} \to H^\wedge(1_A,1_B) $

\noindent \phantom{(b)}are identities,

\noindent (c) colax functors $ H^\vee : {\cal A} \times {\cal B} \to {\cal C} $ for which the colaxity morphisms

\noindent \quad (i) $ H^\vee (f',gg') \to H^\vee (1,g) H^\vee (f',g') $

\noindent \quad (ii) $ H^\vee (ff',g) \to H^\vee (f,g) H^\vee (f',1)$

\noindent \quad (iii) $ H^\vee (1_A,1_B) \to 1_{H^\vee (A,B)} $

\noindent \phantom{(b)}are identities.

\noindent Furthermore, Gray's quasi-natural transformations $ H \to K $ are in bijection with lax transformations $ H^\wedge \to K^\wedge $ and also with lax transformations $ H^\vee \to K^\vee $. A similar statement applies to modifications.

\end{proposition}

\begin{proof} (Sketch)

\noindent $H^\wedge (A,B) = H^\vee (A,B) = H (A,B) $,

\noindent $ H^\wedge (f,g) = H(f,B) H(A',g)$,

\noindent $ h^\wedge (f,g;f',g') = H(f,B) h(f',g) H(A'',g)$,

\noindent $ H^\vee (f,g) = H(A,g) H(f,B') $,

\noindent $ h^\vee (f,g; f,g) = H(A,g) h(f,g') H(f',B'') $,

\noindent $ H(A,g) = H^\wedge (1_A,g) = H^\vee (1_A,g) $,

\noindent $ H(f,B) = H^\wedge (f, 1_B) = H^\vee (f, 1_B)$.

\noindent It is now just a question of direct calculation to verify all the equations.

\end{proof}

\subsection{Gray categories as intercategories -- lax case}\label{Gray.2}

\ 

It follows from the proposition that composition in a true Gray category may be considered as a special kind of lax functor
$$
{\cal A} (A,B) \times {\cal A} (B,C) \to {\cal A} (A,C)
$$
or, alternatively, as a colax functor. In this way we define the horizontal composition of $2$-cells and $3$-cells in two different ways. Thus we get two different types of ``$3$-category'' with lax or colax interchange. In \cite{GPS}, two different tricategories are gotten from a ${\bf Gray}$ category, which are called left and right, but are equivalent and one is chosen arbitrarily. For true Gray categories the two are quite different, the one represented vertically, the other horizontally as intercategories.

Consider first the lax case. We take $2$-categories as vertical double categories, which is forced because that is where the laxity occurs. Then we put all the homs together and get a category object
$$
\sum_{A,B,C \in \Ob ({\cal A})} {\cal A} (A,B) \times {\cal A} (B,C) \threepppp/>`>`>/<400>^{p_1} |{m} _{p_2} \sum_{A,B \in \Ob({\cal A})} {\cal A}(A,B) \three/>`<-`>/^{\partial_0} | {\id}_{\partial_1}\ \Ob({\cal A}) 
$$
in $ \Dlax $, which is of course an intercategory of a special sort. It remains only to determine how it is special in intercategory terms.

Referring to the table of Section 4 of \cite{ICI}, we see that a true Gray category $ {\cal A} $ gives an intercategory ${\ic A}_l$ as follows:

\noindent (1) The objects are those of $ {\cal A} $

\noindent (2) Transversal arrows are identities

\noindent (3) Horizontal arrows are the $1$-cells of $ {\cal A}$

\noindent (4) Vertical arrows are identities

\noindent (5) Horizontal cells are identities

\noindent (6) Vertical cells are identities

\noindent (7) Basic cells are the $2$-cells of $ {\cal A} $

\noindent (8) Cubes are the $3$-cells of $ {\cal A} $.

\noindent So a general cube would look like
$$
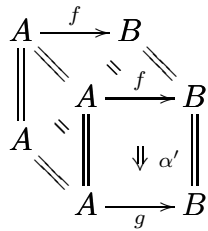
\bfig\scalefactor{.7}

\square(0,300)/>`=``/[A `B `A `;f` ``]

\square(300,0)/>`=`=`>/[A ` B `A `B; f```g]

\morphism(0,800)|b|/=/<300,-300>[A` A;]

\morphism(500,800)/=/<300,-300>[B ` B;]

\morphism(0,300)|b|/=/<300,-300>[A`A;]

\morphism(550,350)|r|/=>/<0,-180>[`;\ \alpha']


\morphism(130,430)/=/<80,-80>[`;]



\morphism(370,700)/=/<80,-80>[`;]

\efig
$$
with an $ \alpha $ in the back face and a $3$-cell $ \alpha \to \alpha' $ inside, corresponding to
$$
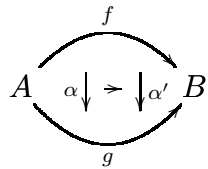
\bfig\scalefactor{.7}

\morphism(0,0)|a|/{@{>}@/^1.8em/}/<800,0>[A `B;f]

\morphism(0,0)|b|/{@{>}@/_1.8em/}/<800,0>[A ` B;g]

\morphism(300,150)|l|/>/<0,-250>[`; \alpha]

\morphism(350,0)/>/<150,0>[`;]

\morphism(550,150)|r|/>/<0,-250>[`;\alpha']

\efig
$$
in $ {\cal A} $.

Transversal and vertical composition come from the $2$-category homs of $ {\cal A} $. Horizontal composition of arrows is that of $ {\cal A} $, but for basic cells and cubes the decision was made, when we chose the lax case, to take
$$
\bfig\scalefactor{.7}

\square/>`=`=`>/[A`B`A`B;f```g]

\morphism(250,350)|r|/=>/<0,-150>[`;\ \alpha]

\square(500,0)/>`=`=`>/[B`C`B`C;h```k]

\morphism(750,350)|r|/=>/<0,-150>[`;\ \beta]

\place(1300,250)[=]

\square(1500,0)/>`=`=`>/<700,500>[A`C`A`C;fh```gk]

\morphism(1700,350)|r|/=>/<0,-150>[`;\ \alpha h \cdot g\beta]

\efig
$$
extended to cubes in the obvious way.

For interchange on
$$
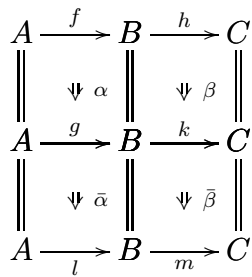
\bfig\scalefactor{.7}

\square/>`=`=`>/[A `B`A`B;g```l]

\square(0,500)/>`=`=`>/[A`B`A`B;f```]

\morphism(250,350)|r|/=>/<0,-150>[`;\ \ov{\alpha}]

\morphism(250,850)|r|/=>/<0,-150>[`;\ \alpha]

\square(500,0)/>`=`=`>/[B`C`B`C;k```m]

\square(500,500)/>`=`=`>/[B`C`B`C;h```]

\morphism(750,350)|r|/=>/<0,-150>[`;\ \ov{\beta}]

\morphism(750,850)|r|/=>/<0,-150>[`;\ \beta]

\efig
$$
$$
(\alpha \circ \beta) \bullet (\ov{\alpha} \circ \ov{\beta}) = \alpha h \cdot g \beta \cdot \ov{\alpha} k \cdot l \ov{\beta}
$$
and
$$
(\alpha \bullet \ov{\alpha}) \circ (\beta \bullet \ov{\beta}) = \alpha h \cdot \ov{\alpha} h \cdot l\beta \cdot l \ov{\beta}
$$
Grayness of composition gives a $2$-cell in $ {\cal A} (A,C) $
$$
x : g \beta \cdot \ov{\alpha} k \to \ov{\alpha} h \cdot l \beta
$$
and so a $2$-cell
$$
\alpha h \cdot x \cdot l \ov{\beta} : (\alpha \circ \beta) \bullet (\ov{\alpha} \circ \ov{\beta}) \to (\alpha \bullet \ov{\alpha}) \circ (\beta \bullet \ov{\beta})
$$
i.e.\ a special cube
$$
\chi   : \dfrac{\alpha | \beta}{\ov{\alpha} | \ov{\beta}} \ \to\  \left.\dfrac{\alpha}{\ov{\alpha}} \right|\dfrac{\beta}{\ov{\beta}}
$$
Conditions b(i) and b(ii) of Proposition \ref{GrayMacLane} say that if either $ \beta $ or $ \ov{\alpha} $ is $ 1 $, i.e.\ $\Id $, then $ \chi $ is equality. Condition b(iii) says that $ \delta : \Id \to \Id |\Id $ is an equality. $ \mu $ and $ \tau $ are also equalities because the hom's are $2$-categories. These conditions, together with the fact that all compositions are strict, characterize those intercategories arising from true Gray categories by considering the composition as a lax functor.

\subsection{Gray categories as intercategories -- colax case}\label{Gray.3}

\ 

The colax case is similar. Instead of using (b) of Proposition \ref{GrayMacLane} we use (c) to get a category object in $ \Dcolax $. We again get an intercategory $ {\ic A}_c $ except that now the $1$-cells of $ {\cal A} $ are made into the vertical arrows of  ${\ic A}_c $, its horizontal arrows being identities. A general cube is now

$$
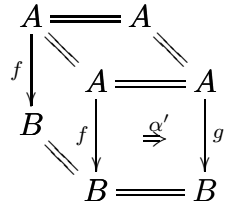
\bfig\scalefactor{.7}

\square(0,300)/=`>``/[A `A `B `;` f``]

\square(300,0)/=`>`>`=/[A ` A `B `B; `f`g`]

\morphism(0,800)|b|/=/<300,-300>[A` A;]

\morphism(500,800)/=/<300,-300>[A ` A;]

\morphism(0,300)|b|/=/<300,-300>[B`B;]

\morphism(480,250)|a|/=>/<175,0>[`;\ \alpha']

\morphism(130,430)/=/<75,-75>[`;]

\morphism(350,700)/=/<75,-75>[`;]

\efig
$$
with a $3$-cell $ \alpha \to \alpha' $ inside. Vertical composition is given by
$$
\bfig\scalefactor{.7}

\square/=`>`>`=/[B`B`C`C;`h`k`]

\morphism(210,250)|a|/=>/<175,0>[`;\beta]

\square(0,500)/=`>`>`=/[A`A`B`B;`f`g`]

\morphism(210,750)|a|/=>/<175,0>[`;\alpha]

\place(800,500)[=]

\square(1200,0)/=`>`>`=/<500,1000>[A`A`C`C;`fh`gk`]

\morphism(1370,500)|a|/=>/<250,0>[`;f\beta \cdot \alpha k]

\efig
$$
$$
\alpha \bullet \beta = f \beta \cdot \alpha k.
$$
Again, $\delta, \mu, \tau $ are equalities but now
$$
\chi : \dfrac{\alpha | \beta}{\ov{\alpha} | \ov{\beta}}  \ \to\ \left.\dfrac{\alpha}{\ov{\alpha}}\right|\dfrac{\beta}{\ov{\beta}}
$$
is equality if either $ \alpha $ or $ \ov{\beta} $ is $ 1 $, i.e.\ $\id$.

\subsection{Gray categories as intercategories -- symmetric case}\label{Gray.4}

\ 

Having two equally good ways of considering true Gray categories as intercategories is a bit unsatisfactory. There is a better way, not suggested by Proposition \ref{GrayMacLane} but rather by the theory of intercategories, where there is room to incorporate both aspects symmetrically.

Given a true Gray category $ {\cal A} $ we construct an intercategory $ {\ic A}_s $ as follows.

\noindent (1) Objects same as $ {\cal A} $

\noindent (2) Transversal arrows are identities

\noindent (3) Horizontal and vertical arrows are $1$-cells of $ {\cal A} $

\noindent (4) Horizontal and vertical cells are identities

\noindent (5) Basic cells are co-quintets
$$
\bfig\scalefactor{.7}

\square[A`A'`B`B';f`k`k'`g]

\morphism(200,250)|a|/=>/<175,0>[`;\alpha]

\efig
$$

\noindent (6) Cubes are $3$-cells $ \alpha \to \ov{\alpha} $

$$
\bfig\scalefactor{.7}

\square(0,300)/>`>``/[A `A' `B `;f`k ``]

\square(300,0)[A ` A' `B `B'; f`k`k'`g]

\morphism(0,800)|b|/=/<300,-300>[A` A;]

\morphism(500,800)/=/<300,-300>[A' ` A';]

\morphism(0,300)|b|/=/<300,-300>[B`B;]

\morphism(500,250)|a|/=>/<175,-0>[`;\ov{\alpha}]


\morphism(130,430)/=/<80,-80>[`;]



\morphism(370,700)/=/<80,-80>[`;]

\efig
$$
Transversal composition is either trivial or, for cubes, composition of $3$-cells. Horizontal composition of basic cells is given by
$$
\bfig\scalefactor{.7}

\square[A`A'`B`B';f`k`k'`g]

\morphism(200,230)|a|/=>/<175,0>[`;\alpha]

\square(500,0)[A'`A''`B'`B'';f'``k''`g']

\morphism(700,230)|a|/=>/<175,0>[`;\alpha']

\place(1400,250)[=]

\square(1800,0)<700,500>[A`A''`B`B'';ff'`k`k''`gg']

\morphism(2050,230)|a|/=>/<220,0>[`;\alpha g' \cdot f\alpha']

\efig
$$
$$
\phantom{Vertical composition}\alpha \circ \alpha' = (k gg' \to^{\alpha g'} fk'g' \to^{f\alpha'} ff'k'')\phantom{Vertical composition}(1) 
$$
Vertical composition of basic cells is
$$
\bfig\scalefactor{.7}

\square[B`B'`C`C';`l`l'`h]

\morphism(200,250)|a|/=>/<175,0>[`;\beta]

\square(0,500)[A`A'`B`B';f`k`k'`g]

\morphism(200,750)|a|/=>/<175,0>[`;\alpha]

\place(850,500)[=]

\square(1200,0)<500,1000>[A`A'`C`C';f`kl`k'l'`h]

\morphism(1350,500)|a|/=>/<200,0>[`;k\beta \cdot \alpha l']

\efig
$$
$$
\phantom{For cubes, horizontal and}\alpha \bullet \beta = (klh \to^{k\beta} kgl' \to^{\alpha l'} fk'l').\phantom{For cubes, horizontal and}(2)
$$
For cubes, horizontal and vertical composition are given by the same formulas but applied to $3$-cells.

Note that there is really no choice for these composites when we work with quintets and they reduce to the ones above when the horizontal or vertical domains and codomains are identities. However, the basic cells are oriented differently in the first case, which is unavoidable.

For interchange on
$$
\bfig\scalefactor{.7}

\square[B`B'`C`C';`l`l'`h]

\morphism(200,250)|a|/=>/<175,0>[`;\beta]

\square(0,500)[A`A'`B`B';f`k`k'`g]

\morphism(200,750)|a|/=>/<175,0>[`;\alpha]

\square(500,0)[B'`B''`C'`C'';``l''`h']

\morphism(700,250)|a|/=>/<175,0>[`;\beta']

\square(500,500)[A'`A''`B'`B'';f'``k''`g']

\morphism(700,750)|a|/=>/<175,0>[`;\alpha']

\efig
$$
$ \dfrac{\alpha | \alpha'}{\beta | \beta'} $ is the top composite and, $ \left.\dfrac{\alpha}{\beta}\right|\dfrac{\alpha'}{\beta'} $, the bottom in

$$
\bfig\scalefactor{.7}

\morphism(0,500)|a|/>/<800,0>[klhh'`kgl'h';k\beta h']

\Ctriangle(800,0)/<-``>/[kgg'l''`kgl'h'`fk'l'h';kg\beta'``\alpha l'h']

\Dtriangle(1300,0)/`>`<-/[kgg'l''`fk'g'l''`fk'l'h';`\alpha g'l''`fk'\beta']

\morphism(1800,500)|a|/>/<800,0>[fk'g'l''`ff'k''l'';f\alpha' l'']

\efig
$$
and
$$
\chi : \dfrac{\alpha | \alpha'}{\beta | \beta'} \ \to\ \left.\dfrac{\alpha}{\beta}\right|\dfrac{\alpha'}{\beta'}
$$
is given by 
$$
k\beta h' \cdot x \cdot f \alpha' l''
$$
where
$$
x : kg\beta' \cdot \alpha g' l'' \to \alpha l'h' \cdot fk'\beta'
$$
is the Gray morphism in the diamond above.

\begin{remark}\label{GrayCondition} Note that if either $ \alpha $ or $ \beta' $ are identities, then $ x $ is equality and so $ \chi $ is too, a fact we will use later.

\end{remark}
The straightforward calculations required to show that we do get an intercategory are left to the reader.

It is harder now to pin down the precise conditions required for an intercategory to be of the form ${\cal A}_s $. First of all we have to express, in intercategory terms, what it means for the horizontal and vertical arrows to be ``the same''. It is not a self-duality because co-quintets themselves are not self-dual, so the word symmetric is perhaps not appropriate. It rather has to do with companions but care has to be taken in the absence of the strict interchange law.

In what follows, we assume for convenience that all composites are strictly unitary and associative, and that transversal arrows and horizontal and vertical cells are all transversal identities, conditions which hold for $ {\ic A}_s $.

\begin{definition}\label{isotropic} An intercategory is {\em isotropic} if:

\noindent (i) there is given a functorial bijection between horizontal and vertical arrows
$$
\bfig\scalefactor{.7}

\square/`@{>}|{\bb}``/[A``B`;`f``]

\place(300,250)[\longleftrightarrow]

\morphism(600,250)/@{>}|{\cc}/[A`B;f_*]

\efig
$$
$$
(\Id_A)_* = \id_A \ \ \ \mbox{and}\ \ \ \left(\dfrac{f}{g}\right)_* = f_* | g_*
$$

\noindent (ii) for each $ f $ there are given basic cells
$$
\bfig\scalefactor{.7}

\square/@{>}|{\cc}`@{>}|{\bb}`@{>}|{\bb}`@{>}|{\cc}/[A `A `A `B;\id_A `\Id_A `f `f_*]

\place(250,250)[\scriptstyle \eta_f]

\place(800,250)[\mbox{and}]

\square(1100,0)/@{>}|{\cc}`@{>}|{\bb}`@{>}|{\bb}`@{>}|{\cc}/[A `B `B `B;f_* ` f `\Id_B `\id_B]

\place(1350,250)[\scriptstyle \epsilon_f]

\efig
$$
satisfying

\noindent (iii) (companionship)
$$
\eta_f | \epsilon_f = \Id_{f_*}\ \ \ \mbox{and}\ \ \ \dfrac{\eta_f}{\epsilon_f} = \id_f
$$

\noindent (iv) (functoriality)
$$
	\epsilon_{\Id_A} = \Id_{\id_A} \ \ \  \epsilon_{\frac{f}{g}} = \dfrac{\epsilon_f | \Id_{g_*}}{\id_g | \epsilon_g}
$$
$$
\eta_{\Id_A} = \id_{\Id_A} \ \ \ \eta_{\frac{f}{g}} = \left.\dfrac{\eta_f}{\Id_{f_*}} \right|\dfrac{\id_g}{\eta_g}
$$
(One may wonder why composites in the conditions for $ \epsilon_\frac{f}{g} $ and $ \eta_\frac{f}{g} $ are expressed in that order. It is merely for esthetic reasons. By ``whiskering'' below, the order is immaterial.)

\noindent (v) (whiskering)
$$
\chi \left(\bfig\square/```/<100,120>[*`{\scriptstyle \id}`{\scriptstyle \Id} `{\scriptstyle \eta};```]\efig\right) = 1 \ \ \ \mbox{and}\ \ \chi \left(\bfig\square/```/<100,120>[{\scriptstyle \epsilon} `{\scriptstyle \Id} `{\scriptstyle \id}`*;```]\efig\right) = 1
$$
\end{definition}

A word about condition (v): Given a basic cell $ \phi $ and a vertical arrow $ x $ as in
$$
\bfig\scalefactor{.7}

\square/@{>}|{\cc}`@{>}|{\bb}`@{>}|{\bb}`@{>}|{\cc}/[A `B `C `D;f `g `h `k]

\place(250,250)[\scriptstyle \phi]

\square(0,500)/`@{>}|{\bb}``/[x``A`B;`x``]

\efig
$$
we can form the following $ 2 \times 2 $ array of cells
$$
\bfig\scalefactor{.7}

\square/`@{>}|{\bb}``>/<700,500>[A `A `C `C;`g``\id_c]

\square(700,0)/@{>}|{\cc}`@{>}|{\bb}`@{>}|{\bb}`@{>}|{\cc}/<700,500>[A `B `C `C;f`g`h`k]

\square(0,500)/@{>}|{\cc}`@{>}|{\bb}`@{>}|{\bb}`@{>}|{\cc}/<700,500>[X `A `A `A;x_* `x ``\id_A]

\square(700,500)/@{>}|{\bb}`@{>}|{\bb}`@{>}|{\bb}`/<700,500>[A `B `A `B;f`\scriptstyle\Id_A `\Id_B `]

\place(350,250)[\scriptstyle \id_g]

\place(1050,250)[\scriptstyle \phi]

\place(350,750)[\scriptstyle \epsilon_x]

\place(1050,750)[\scriptstyle \Id_f]

\efig
$$
and condition (v) says that it can be evaluated in either order. The result may be thought of as whiskering $ \phi $ with $ x $.

\begin{proposition} $ {\ic A}_s $ is isotropic.

\end{proposition}

\begin{proof} (i) $ f^* $ is $ f $ so the conditions are trivially satisfied.

\noindent (ii) $ \eta_f $ and $ \epsilon_f $ are both $ 1_f $.

\noindent (iii) Formula (1) gives $ \eta_f | \epsilon_f = 1_F $ which is $ \Id_f $, and formula (2) gives $ \dfrac{\eta_f}{\epsilon_f} = 1_f $ which is also $ \id_f $.

\noindent (iv) All the cells involved here are $ 1 $, the identity $1$-cell in the hom $2$-categories of ${\cal A} $, so long as the domains and codomains correspond, any equation holds, and this is the case with functoriality.

\noindent (v) Both $ \eta $ and $ \epsilon $ are given by identities, so ``whiskering'' follows by Remark \ref{GrayCondition}.

\end{proof}

There is a further condition needed to characterize those intercategories of the form $ {\ic A}_s $.

\begin{definition} We say that $ \chi $ (or ${\ic A}$) satisfies the {\em Gray condition} if
$$
\chi \left(\bfig\square/```/<100,120>[{\scriptstyle \id}`*`*`*;```]\efig\right) = 1 \ \ \ \ \ \ \ \ \ \chi \left(\bfig\square/```/<100,120>[*`*`*`{\scriptstyle \id};```]\efig\right) = 1
$$
$$
\chi \left(\bfig\square/```/<100,120>[{\scriptstyle \Id}`*`*`*;```]\efig\right) = 1 \ \ \ \ \ \ \ \ \ \chi \left(\bfig\square/```/<100,120>[*`*`*`{\scriptstyle \Id};```]\efig\right) = 1
$$
where the $ * $ represent arbitrary (compatible) cells.
\end{definition}

By Remark \ref{GrayCondition}, $ {\ic A}_s $ satisfies the Gray condition. We can now state the main theorem of this section.

\begin{theorem}\label{TrueGray} There are three equivalent ways of representing a true Gray category as an intercategory. They all satisfy the following properties:

\noindent (a) transversal arrows and horizontal and vertical cells are transversal identities,

\noindent (b) all composites are strictly unitary and associative,

\noindent (c) the degenerate interchangers $ \tau $, $ \delta $, $ \mu $ are identities, i.e. the intercategory is chiral.

The three ways are characterized by:

\noindent (1) The lax case

\noindent (d) vertical arrows are identities,

\noindent (e) $ \chi \left(\bfig\square/```/<100,120>[*`{\scriptstyle \Id}`*`*;```]\efig\right) = 1 \ \ \ \ \mbox{and}\ \ \ \ \ \ \chi \left(\bfig\square/```/<100,120>[*`*`{\scriptstyle \Id}`*;```]\efig\right) = 1 $.

\noindent (2) The colax case

\noindent (d) horizontal arrows are identities,

\noindent (e) $ \chi \left(\bfig\square/```/<100,120>[{\scriptstyle \id}`*`*`*;```]\efig\right) = 1 \ \ \ \ \mbox{and}\ \ \ \ \ \ \chi \left(\bfig\square/```/<100,120>[*`*`*`{\scriptstyle \id};```]\efig\right) = 1 $.

\noindent (3) The symmetric case

\noindent (d) the intercategory is isotropic,

\noindent (e) $ \chi $ satisfies the Gray condition.

\end{theorem}

\begin{proof} The equivalence of true Gray categories with intercategories satisfying conditions (a)-(c), (1d), (1e) follows from the discussion of Section \ref{Gray.2}. The colax case is similar and touched upon in Section \ref{Gray.3}.

That the intercategory $ {\ic A}_s $ constructed from a true Gray category satisfies (a)-(c), (3d), (3e) follows from the discussion in this section. It remains to show that any intercategory $ {\ic B} $ satisfying these conditions corresponds to a true Gray category.

Given such a $ {\ic B} $ we get an intercategory $ {\ic A} $ satisfying (a)-(c), (2d), (2e) by restricting all elements to those involving no horizontal arrows except identities. This works because the composite of two horizontal identities is an identity. Such an $ {\ic A} $ comes from a unique true Gray category $ {\cal A} $ with the same ingredients, only packaged differently.

If we start with $ {\cal A} $ and construct $ {\ic A}_s $ and then restrict to identity horizontal arrows, we get $ {\ic A}_s $ which corresponds to $ {\cal A} $. The heart of the proof lies in showing that when we start with a $ {\ic B} $ satisfying the conditions (3), restrict to $ {\ic A} $ and then take $ {\ic A}_s $ we get an intercategory isomorphic to $ {\ic B} $.

The objects, vertical arrows and vertical cells (identities) are the same for $ {\ic A} $ and $ {\ic B}$. Define $ F : {\ic A}_s \to {\ic B} $ to be the identity on these. A horizontal arrow of $ {\ic A}_s $, $ f : A \toc B $ is the same as a vertical arrow so define $ F(f) = f_* $, which is functorial by condition (i) of Definition \ref{isotropic}. Horizontal cells are identities in both cases so there is no problem defining $ F $ on them. A basic cell in $ {\ic A}_s $ is a quintet 
$$
\bfig\scalefactor{.7}

\square[A `B `C `D;f `g `h `k]

\morphism(200,230)/=>/<175,0>[`;{\scriptstyle \phi}]

\efig
$$
Define $ F(\phi) $ to be 
$$
\bfig\scalefactor{.7}

\square/=`=``>/[C `C `C `D;```k_*]

\square(0,500)/=`>``=/[A `A `C `C;`g``]

\square(500,0)/`>`>`=/[C `B `D `D;`k`h`]

\square(500,500)/=`>`>`/[A `A `C `B;`g`f`]

\square(1000,0)/=``>`=/[B `B `D `D;``h`]

\square(1000,500)/>``=`=/[A `B `B `B;f_*```]

\place(250,250)[\scriptstyle \eta_k]

\place(250,750)[\scriptstyle \id_g]

\place(750,500)[\scriptstyle \phi]

\place(1250,250)[\scriptstyle \id_h]

\place(1250,750)[\scriptstyle \eta_f]

\efig
$$
which we will write as
$$
\setlength{\unitlength}{.5mm}
\begin{picture}(30,20)
\put(1,10){\line(1,0){8}}
\put(10,1){\line(0,1){18}}
\put(20,1){\line(0,1){18}}
\put(21,10){\line(1,0){8}}

\put(5,3){\ps{\eta}}

\put(4,13){\ps{\id_g}}
\put(15,8){\ps{\phi}}
\put(26,3){\ps{\id_h}}
\put(25,13){\ps{\epsilon}}

\end{picture}
$$
A cube in $ {\ic A}_s $ 
$$
\bfig\scalefactor{.7}

\square(0,300)/>`>``/[A `B `C `;f`g ``]

\square(300,0)[A ` B `C `D; f`g`h`k]

\morphism(0,800)|b|/=/<300,-300>[A` A;]

\morphism(500,800)/=/<300,-300>[B ` B;]

\morphism(0,300)|b|/=/<300,-300>[C`C;]

\place(500,250)[\ov{\phi}]

\morphism(130,430)/=/<80,-80>[`;]

\morphism(370,700)/=/<80,-80>[`;]

\efig
$$
is a $3$-cell $ {\frak a} : \phi \to \ov{\phi} $
$$
\bfig\scalefactor{.7}

\square(0,300)/`>``/[C``D`;`k``]

\square(0,800)/=`>``/[A`A`C`;`g``]

\square(300,0)/`>`>`=/[C`B`D`D;``h`]

\square(300,500)/=`>`>`/[A`A`C`B;``f`]

\morphism(0,300)/=/<300,-300>[D`D;]

\morphism(0,800)/=/<300,-300>[C`C;]

\morphism(0,1300)/=/<300,-300>[A`A;]

\morphism(500,1300)/=/<300,-300>[A`A;]

\place(550,500)[\ov{\phi}]

\place(-500,650)[{\frak a}:]

\efig
$$
Define
$$
\setlength{\unitlength}{.6mm}
\begin{picture}(60,20)
\put(0,10){$F({\frak a}) \ =$}
\put(31,10){\line(1,0){8}}
\put(40,1){\line(0,1){18}}
\put(50,1){\line(0,1){18}}
\put(51,10){\line(1,0){8}}

\put(35,3){\ps{1_{\eta_k}}}
\put(35,13){\ps{1_{\id_g}}}
\put(45,8){\ps{{\frak a}}}
\put(55,3){\ps{1_{\id_h}}}
\put(55,13){\ps{1_{\epsilon_f}}}

\end{picture}
$$
$$
\bfig\scalefactor{.7}

\morphism(0,800)/=/<0,-500>[C`C;]

\morphism(0,1300)/>/<0,-500>[A`C;g]

\morphism(0,1300)/=/<500,0>[A`A;]

\morphism(500,1300)/=/<500,0>[A`A;]

\morphism(1000,1300)/>/<500,0>[A`B;f_*]

\square(300,0)/=`=`>`>/[C`C`C`D;``k`k_*]

\square(300,500)|arra|/=`>`>`=/[A`A`C`C;`g`g`]

\square(1300,0)/=`>`>`=/[B`B`D`D;`h`h`]

\square(1300,500)/>`>`=`=/[A`B`B`B;f_*`f``]

\morphism(800,0)/=/<500,0>[D`D;]

\morphism(800,1000)/=/<500,0>[A`A;]

\morphism(0,300)/=/<300,-300>[C`C;]

\morphism(0,800)/=/<300,-300>[C`C;]

\morphism(0,1300)/=/<300,-300>[A`A;]

\morphism(500,1300)/=/<300,-300>[A`A;]

\morphism(1000,1300)/=/<300,-300>[A`A;]

\morphism(1500,1300)/=/<300,-300>[B`B;]

\morphism(115,420)/=/<70,-70>[`;]

\morphism(115,920)/=/<70,-70>[`;]

\morphism(350,1200)/=/<70,-70>[`;]

\morphism(850,1200)/=/<70,-70>[`;]

\morphism(1350,1200)/=/<70,-70>[`;]

\place(550,250)[\scriptstyle \eta_k]

\place(550,750)[\scriptstyle \id_g]

\place(1050,500)[\scriptstyle \ov{\phi}]

\place(1550,250)[\scriptstyle \id_h]

\place(1550,750)[\scriptstyle \epsilon_f]

\efig
$$
The horizontal identity
$$
\bfig\scalefactor{.7}

\square/@{>}|{\cc}`@{>}|{\bb}`@{>}|{\bb}`@{>}|{\cc}/[A `A `C `C;\id_A `g `g `\id_C]

\place(250,250)[\scriptstyle \id_g]

\efig
$$
in $ {\ic A}_s $ is
$$
\bfig\scalefactor{.7}

\square/`@{>}|{\bb}`@{>}|{\bb}`=/[C `A `C `C;`\Id_C `g`]

\square(0,500)/=`@{>}|{\bb}`@{>}|{\bb}`/[A `A `C `A;`g`\Id_A`]

\place(250,500)[\scriptstyle \id_g]

\efig
$$
so
$$
\setlength{\unitlength}{.8mm}
\begin{picture}(170,20)

\put(-8,8){\ps{F(\id_g)}}
\put(10,8){\ps{=}}

\put(21,10){\line(1,0){8}}
\put(30,1){\line(0,1){18}}
\put(40,1){\line(0,1){18}}
\put(41,10){\line(1,0){8}}

\put(25,3){\ps{\eta_\Id}}
\put(25,13){\ps{\id_g}}
\put(35,8){\ps{\id_g}}
\put(45,3){\ps{\id_g}}
\put(45,13){\ps{\epsilon_\Id}}

\put(60,8){\ps{=}}

\put(71,10){\line(1,0){8}}
\put(80,1){\line(0,1){18}}
\put(90,1){\line(0,1){18}}
\put(91,10){\line(1,0){8}}

\put(75,3){\ps{\id_\Id}}
\put(75,13){\ps{\id_g}}
\put(85,8){\ps{\id_g}}
\put(95,3){\ps{\id_g}}
\put(95,13){\ps{\Id_\id}}

\put(110,8){\ps{=}}

\put(130,6){\line(0,1){8}}
\put(140,6){\line(0,1){8}}

\put(125,8){\ps{\id_g}}
\put(135,8){\ps{\id_g}}
\put(145,8){\ps{\id_g}}

\put(160,8){\ps{=}}

\put(170,8){\ps{\id_g}}

\end{picture}
$$
The second equality uses functoriality of $ \eta $ and $ \epsilon $, and the third that $ \mu = 1 $ so that $ \dfrac{\id}{\id} = \id $.

The vertical identity
$$
\bfig\scalefactor{.7}

\square/@{>}|{\cc}`@{>}|{\bb}`@{>}|{\bb}`@{>}|{\cc}/[A `B `A `B;f `\Id_A `\Id_A `f]

\place(250,250)[\scriptstyle \Id_f]

\place(1500,250)[{\mbox{in}}\ \ {\ic A}_s]

\place(-1000,250)[\ ]

\efig
$$
is
$$
\bfig\scalefactor{.7}

\square/`@{>}|{\bb}`@{>}|{\bb}`=/[A `B `B `B;`f`\Id_B`]

\square(0,500)/=`@{>}|{\bb}`@{>}|{\bb}`/[A `A `A `B;`\Id_A `f`]

\place(250,500)[\scriptstyle \id_f]

\place(1500,500)[{\mbox{in}}\ \ {\ic A}]

\place(-1000,250)[\ ]

\efig
$$
so
$$
\setlength{\unitlength}{.8mm}
\begin{picture}(160,20)

\put(-8,8){\ps{F(\Id_f)}}
\put(10,8){\ps{=}}

\put(21,10){\line(1,0){8}}
\put(30,1){\line(0,1){18}}
\put(40,1){\line(0,1){18}}
\put(41,10){\line(1,0){8}}

\put(25,3){\ps{\eta_f}}
\put(25,13){\ps{\id}}
\put(35,8){\ps{\id_f}}
\put(45,3){\ps{\id}}
\put(45,13){\ps{\epsilon_f}}

\put(60,8){\ps{=}}

\put(80,6){\line(0,1){8}}
\put(90,6){\line(0,1){8}}

\put(75,8){\ps{\eta_f}}
\put(85,8){\ps{\id_f}}
\put(95,8){\ps{\epsilon_f}}

\put(110,8){\ps{=}}

\put(130,6){\line(0,1){8}}

\put(125,8){\ps{\eta_f}}
\put(135,8){\ps{\epsilon_f}}

\put(150,8){\ps{=}}

\put(160,8){\ps{\Id_{f_*}}}

\end{picture}
$$
where the second equality uses that $ \tau = 1 $, so both $\id$'s are $\Id$'s.

Horizontal composition in $ {\ic A}_s $ 
$$
\bfig\scalefactor{.7}

\square[A `B `C `D;f `g `h `k]

\square(500,0)/>``>`>/[B `E `D `F;l ``m `n]

\place(250,250)[\scriptstyle \phi]

\place(750,250)[\scriptstyle \psi]

\efig
$$
is given by
$$
\setlength{\unitlength}{.8mm}
\begin{picture}(20,30)

\put(10,1){\line(0,1){28}}
\put(1,10){\line(1,0){8}}
\put(11,20){\line(1,0){8}}

\put(5,3){\ps{\id_n}}
\put(5,18){\ps{\phi}}
\put(15,8){\ps{\psi}}
\put(15,23){\ps{\id_f}}

\put(30,13){in}
\put(40,13){\ps{{\ic A}}}

\end{picture}
$$
so
$$
\setlength{\unitlength}{.8mm}
\begin{picture}(150,30)

\put(0,13){\ps{F(\phi\circ\psi)}}
\put(20,13){\ps{=}}

\put(40,1){\line(0,1){28}}
\put(50,1){\line(0,1){28}}
\put(60,1){\line(0,1){28}}
\put(31,20){\line(1,0){8}}
\put(41,10){\line(1,0){8}}
\put(51,20){\line(1,0){8}}
\put(61,10){\line(1,0){8}}

\put(35,8){\ps{\eta_{\frac{k}{n}}}}
\put(35,23){\ps{\id_g}}
\put(45,3){\ps{\id_n}}
\put(45,18){\ps{\phi}}
\put(55,8){\ps{\psi}}
\put(55,23){\ps{\id_f}}
\put(65,3){\ps{\id_m}}
\put(65,18){\ps{\epsilon_{\frac{f}{l}}}}

\put(80,13){\ps{=}}
\put(80,10){\ps{{\scriptstyle (1)}}}

\put(100,1){\line(0,1){18}}
\put(100,21){\line(0,1){8}}
\put(110,1){\line(0,1){28}}
\put(120,1){\line(0,1){28}}
\put(130,1){\line(0,1){28}}
\put(140,1){\line(0,1){8}}
\put(140,11){\line(0,1){8}}
\put(140,21){\line(0,1){8}}
\put(91,10){\line(1,0){8}}
\put(91,20){\line(1,0){18}}
\put(101,10){\line(1,0){8}}
\put(111,10){\line(1,0){8}}
\put(121,20){\line(1,0){8}}
\put(131,10){\line(1,0){18}}
\put(131,20){\line(1,0){18}}

\put(95,3){\ps{\Id_{k_*}}}
\put(95,13){\ps{\eta_k}}
\put(95,23){\ps{\id_g}}
\put(105,3){\ps{\eta_n}}
\put(105,13){\ps{\id_k}}
\put(105,23){\ps{\id_g}}
\put(115,3){\ps{\id_n}}
\put(115,18){\ps{\phi}}
\put(125,8){\ps{\psi}}
\put(125,23){\ps{\id_f}}
\put(135,3){\ps{\id_m}}
\put(135,13){\ps{\id_l}}
\put(135,23){\ps{\epsilon_f}}
\put(145,3){\ps{\id_m}}
\put(145,13){\ps{\epsilon_l}}
\put(145,23){\ps{\Id_{l_*}}}

\end{picture}
$$
$$
\setlength{\unitlength}{.8mm}
\begin{picture}(160,30)

\put(0,13){\ps{=}}
\put(0,10){\ps{{\scriptstyle (2)}}}

\put(30,1){\line(0,1){28}}
\put(40,1){\line(0,1){28}}
\put(50,1){\line(0,1){28}}
\put(60,1){\line(0,1){28}}
\put(70,1){\line(0,1){8}}
\put(70,11){\line(0,1){18}}
\put(21,10){\line(1,0){8}}
\put(21,20){\line(1,0){8}}
\put(31,10){\line(1,0){8}}
\put(31,20){\line(1,0){8}}
\put(41,10){\line(1,0){8}}
\put(51,20){\line(1,0){8}}
\put(61,10){\line(1,0){18}}
\put(61,20){\line(1,0){8}}
\put(71,20){\line(1,0){8}}

\put(25,3){\ps{\Id_{k_*}}}
\put(25,13){\ps{\eta_k}}
\put(25,23){\ps{\id_g}}
\put(35,3){\ps{\eta_n}}
\put(35,13){\ps{\id_k}}
\put(35,23){\ps{\id_g}}
\put(45,3){\ps{\id_n}}
\put(45,18){\ps{\phi}}
\put(55,8){\ps{\psi}}
\put(55,23){\ps{\id_f}}
\put(65,3){\ps{\id_m}}
\put(65,13){\ps{\id_l}}
\put(65,23){\ps{\epsilon_f}}
\put(75,3){\ps{\id_m}}
\put(75,13){\ps{\epsilon_l}}
\put(75,23){\ps{\Id_{l_*}}}

\put(90,13){\ps{=}}
\put(90,10){\ps{{\scriptstyle (3)}}}

\put(110,1){\line(0,1){28}}
\put(120,1){\line(0,1){28}}
\put(130,1){\line(0,1){28}}
\put(140,1){\line(0,1){28}}
\put(150,1){\line(0,1){28}}
\put(101,10){\line(1,0){8}}
\put(101,20){\line(1,0){8}}
\put(111,10){\line(1,0){8}}
\put(111,20){\line(1,0){8}}
\put(121,10){\line(1,0){8}}
\put(131,20){\line(1,0){8}}
\put(141,10){\line(1,0){8}}
\put(141,20){\line(1,0){8}}
\put(151,10){\line(1,0){8}}
\put(151,20){\line(1,0){8}}

\put(105,3){\ps{\Id_{k_*}}}
\put(105,13){\ps{\eta_k}}
\put(105,23){\ps{\id_g}}
\put(115,3){\ps{\eta_n}}
\put(115,13){\ps{\id_k}}
\put(115,23){\ps{\id_g}}
\put(125,3){\ps{\id_n}}
\put(125,18){\ps{\phi}}
\put(135,8){\ps{\psi}}
\put(135,23){\ps{\id_f}}
\put(145,3){\ps{\id_m}}
\put(145,13){\ps{\id_l}}
\put(145,23){\ps{\epsilon_f}}
\put(155,3){\ps{\id_m}}
\put(155,13){\ps{\epsilon_l}}
\put(155,23){\ps{\Id_{l_*}}}

\end{picture}
$$
$$
\setlength{\unitlength}{.8mm}
\begin{picture}(140,30)

\put(0,13){\ps{=}}
\put(0,10){\ps{{\scriptstyle (4)}}}

\put(30,1){\line(0,1){28}}
\put(40,1){\line(0,1){8}}
\put(40,11){\line(0,1){18}}
\put(50,1){\line(0,1){28}}
\put(60,1){\line(0,1){18}}
\put(60,21){\line(0,1){8}}
\put(70,1){\line(0,1){28}}
\put(21,10){\line(1,0){8}}
\put(21,20){\line(1,0){8}}
\put(31,10){\line(1,0){18}}
\put(31,20){\line(1,0){8}}
\put(51,20){\line(1,0){18}}
\put(61,10){\line(1,0){8}}
\put(71,10){\line(1,0){8}}
\put(71,20){\line(1,0){8}}

\put(25,3){\ps{\Id_{k_*}}}
\put(25,13){\ps{\eta_k}}
\put(25,23){\ps{\id_g}}
\put(35,3){\ps{\eta_n}}
\put(35,13){\ps{\id_k}}
\put(35,23){\ps{\id_g}}
\put(45,3){\ps{\id_n}}
\put(45,18){\ps{\phi}}
\put(55,8){\ps{\psi}}
\put(55,23){\ps{\id_f}}
\put(65,3){\ps{\id_m}}
\put(65,13){\ps{\id_l}}
\put(65,23){\ps{\epsilon_f}}
\put(75,3){\ps{\id_m}}
\put(75,13){\ps{\epsilon_l}}
\put(75,23){\ps{\Id_{l_*}}}

\put(90,13){\ps{=}}
\put(90,10){\ps{{\scriptstyle (5)}}}

\put(110,1){\line(0,1){28}}
\put(120,1){\line(0,1){28}}
\put(130,1){\line(0,1){28}}
\put(101,10){\line(1,0){8}}
\put(101,20){\line(1,0){8}}
\put(111,10){\line(1,0){8}}
\put(121,20){\line(1,0){8}}
\put(131,10){\line(1,0){8}}
\put(131,20){\line(1,0){8}}

\put(105,3){\ps{\Id_{k_*}}}
\put(105,13){\ps{\eta_k}}
\put(105,23){\ps{\id_g}}
\put(115,3){\ps{\eta_n}}
\put(115,18){\ps{\phi}}
\put(125,8){\ps{\psi}}
\put(125,23){\ps{\epsilon_f}}
\put(135,3){\ps{\id_m}}
\put(135,13){\ps{\epsilon_l}}
\put(135,23){\ps{\Id_{l_*}}}

\end{picture}
$$
(1) uses functoriality of $ \eta $ and $ \epsilon $. (2) uses the Gray property of $ \chi $ on the left and isotropy of $ \epsilon $ on the right. (3) and (4) use the Gray property again. (5) uses that $ \mu : \dfrac{\id}{\id} \to \id $ is the identity.

On the other hand,
$$
\setlength{\unitlength}{.8mm}
\begin{picture}(160,30)

\put(0,13){\ps{F(\phi)\circ F(\psi)}}
\put(20,13){\ps{=}}

\put(40,11){\line(0,1){18}}
\put(50,11){\line(0,1){18}}
\put(60,1){\line(0,1){28}}
\put(70,1){\line(0,1){18}}
\put(80,1){\line(0,1){18}}

\put(31,10){\line(1,0){28}}
\put(31,20){\line(1,0){8}}
\put(51,20){\line(1,0){8}}
\put(61,10){\line(1,0){8}}
\put(61,20){\line(1,0){28}}
\put(81,10){\line(1,0){8}}

\put(45,3){\ps{\Id_{k_*}}}
\put(35,13){\ps{\eta_k}}
\put(35,23){\ps{\id_g}}
\put(45,18){\ps{\phi}}
\put(55,13){\ps{\id_h}}
\put(55,23){\ps{\epsilon_f}}
\put(65,3){\ps{\eta_n}}
\put(65,13){\ps{\id_h}}
\put(75,8){\ps{\psi}}
\put(75,23){\ps{\Id_{l_*}}}
\put(85,3){\ps{\id_m}}
\put(85,13){\ps{\epsilon_l}}

\end{picture}
$$
$$
\setlength{\unitlength}{.8mm}
\begin{picture}(160,30)

\put(10,13){\ps{=}}

\put(30,1){\line(0,1){8}}
\put(30,11){\line(0,1){18}}
\put(40,1){\line(0,1){8}}
\put(40,11){\line(0,1){18}}
\put(50,1){\line(0,1){28}}
\put(60,1){\line(0,1){18}}
\put(60,21){\line(0,1){8}}
\put(70,1){\line(0,1){18}}
\put(70,21){\line(0,1){8}}

\put(21,10){\line(1,0){28}}
\put(21,20){\line(1,0){8}}
\put(41,20){\line(1,0){8}}
\put(51,10){\line(1,0){8}}
\put(51,20){\line(1,0){28}}
\put(71,10){\line(1,0){8}}

\put(25,3){\ps{\Id_{k_*}}}
\put(25,13){\ps{\eta_k}}
\put(25,23){\ps{\id_g}}
\put(35,3){\ps{\id_\Id}}
\put(35,18){\ps{\phi}}
\put(45,3){\ps{\id_\Id}}
\put(45,13){\ps{\id_h}}
\put(45,23){\ps{\epsilon_f}}
\put(55,3){\ps{\eta_n}}
\put(55,13){\ps{\id_h}}
\put(55,23){\ps{\id_\Id}}
\put(65,8){\ps{\psi}}
\put(65,23){\ps{\id_\Id}}
\put(75,3){\ps{\id_m}}
\put(75,13){\ps{\epsilon_l}}
\put(75,23){\ps{\Id_{k_*}}}

\put(90,13){\ps{=}}
\put(90,10){\ps{\scriptstyle (1)}}

\put(110,1){\line(0,1){28}}
\put(120,1){\line(0,1){28}}
\put(130,1){\line(0,1){28}}
\put(140,1){\line(0,1){28}}
\put(150,1){\line(0,1){28}}

\put(101,10){\line(1,0){8}}
\put(101,20){\line(1,0){8}}
\put(111,10){\line(1,0){8}}
\put(121,10){\line(1,0){8}}
\put(121,20){\line(1,0){8}}
\put(131,10){\line(1,0){8}}
\put(131,20){\line(1,0){8}}
\put(141,20){\line(1,0){8}}
\put(151,10){\line(1,0){8}}
\put(151,20){\line(1,0){8}}

\put(105,3){\ps{\Id_{k_*}}}
\put(105,13){\ps{\eta_k}}
\put(105,23){\ps{\id_g}}
\put(115,3){\ps{\id_\Id}}
\put(115,18){\ps{\phi}}
\put(125,3){\ps{\id_\Id}}
\put(125,13){\ps{\id_h}}
\put(125,23){\ps{\epsilon_f}}
\put(135,3){\ps{\eta_n}}
\put(135,13){\ps{\id_h}}
\put(135,23){\ps{\id_\Id}}
\put(145,8){\ps{\psi}}
\put(145,23){\ps{\id_\Id}}
\put(155,3){\ps{\id_m}}
\put(155,13){\ps{\epsilon_l}}
\put(155,23){\ps{\Id_{l_*}}}

\end{picture}
$$
$$
\setlength{\unitlength}{.8mm}
\begin{picture}(160,30)

\put(10,13){\ps{=}}
\put(10,10){\ps{\scriptstyle (2)}}

\put(30,1){\line(0,1){28}}
\put(40,1){\line(0,1){28}}
\put(50,1){\line(0,1){18}}
\put(50,21){\line(0,1){8}}
\put(60,1){\line(0,1){28}}
\put(70,1){\line(0,1){28}}

\put(21,10){\line(1,0){8}}
\put(21,20){\line(1,0){8}}
\put(31,10){\line(1,0){8}}
\put(41,20){\line(1,0){18}}
\put(51,10){\line(1,0){8}}
\put(61,20){\line(1,0){8}}
\put(71,10){\line(1,0){8}}
\put(71,20){\line(1,0){8}}

\put(25,3){\ps{\Id_{k_*}}}
\put(25,13){\ps{\eta_k}}
\put(25,23){\ps{\id_g}}
\put(35,3){\ps{\id_\Id}}
\put(35,18){\ps{\phi}}
\put(45,8){\ps{\id_h}}
\put(45,23){\ps{\epsilon_f}}
\put(55,3){\ps{\eta_n}}
\put(55,13){\ps{\id_h}}
\put(55,23){\ps{\id_\Id}}
\put(65,8){\ps{\psi}}
\put(65,23){\ps{\id_\Id}}
\put(75,3){\ps{\id_m}}
\put(75,13){\ps{\epsilon_l}}
\put(75,23){\ps{\Id_{l_*}}}

\put(90,13){\ps{=}}

\put(110,1){\line(0,1){28}}
\put(120,1){\line(0,1){28}}
\put(130,1){\line(0,1){18}}
\put(130,21){\line(0,1){8}}
\put(140,1){\line(0,1){28}}
\put(150,1){\line(0,1){28}}

\put(101,10){\line(1,0){8}}
\put(101,20){\line(1,0){8}}
\put(111,10){\line(1,0){8}}
\put(121,10){\line(1,0){8}}
\put(121,20){\line(1,0){18}}
\put(141,20){\line(1,0){8}}
\put(151,10){\line(1,0){8}}
\put(151,20){\line(1,0){8}}

\put(105,3){\ps{\Id_{k_*}}}
\put(105,13){\ps{\eta_k}}
\put(105,23){\ps{\id_g}}
\put(115,3){\ps{\id_\Id}}
\put(115,18){\ps{\phi}}
\put(125,3){\ps{\eta_n}}
\put(125,13){\ps{\id_h}}
\put(125,23){\ps{\id_f}}
\put(135,8){\ps{\id_\frac{h}{n}}}
\put(135,23){\ps{\epsilon_f}}
\put(145,8){\ps{\psi}}
\put(145,23){\ps{\id_\Id}}
\put(155,3){\ps{\id_n}}
\put(155,13){\ps{\epsilon_l}}
\put(155,23){\ps{\Id_{l_*}}}

\end{picture}
$$
$$
\setlength{\unitlength}{.8mm}
\begin{picture}(160,30)

\put(10,13){\ps{=}}
\put(10,10){\ps{\scriptstyle (3)}}

\put(30,1){\line(0,1){28}}
\put(40,1){\line(0,1){28}}
\put(50,1){\line(0,1){28}}
\put(60,1){\line(0,1){28}}
\put(70,1){\line(0,1){28}}

\put(21,10){\line(1,0){8}}
\put(21,20){\line(1,0){8}}
\put(31,10){\line(1,0){8}}
\put(41,10){\line(1,0){8}}
\put(41,20){\line(1,0){8}}
\put(51,20){\line(1,0){8}}
\put(61,20){\line(1,0){8}}
\put(71,10){\line(1,0){8}}
\put(71,20){\line(1,0){8}}

\put(25,3){\ps{\Id_{k_*}}}
\put(25,13){\ps{\eta_k}}
\put(25,23){\ps{\id_g}}
\put(35,3){\ps{\id_\Id}}
\put(35,18){\ps{\phi}}
\put(45,3){\ps{\eta_n}}
\put(45,13){\ps{\id_h}}
\put(45,23){\ps{\id_f}}
\put(55,8){\ps{\id_\frac{h}{n}}}
\put(55,23){\ps{\epsilon_f}}
\put(65,8){\ps{\psi}}
\put(65,23){\ps{\id_\Id}}
\put(75,3){\ps{\id_m}}
\put(75,13){\ps{\epsilon_l}}
\put(75,23){\ps{\Id_{l_*}}}

\put(90,13){\ps{=}}
\put(90,10){\ps{\scriptstyle (4)}}

\put(110,1){\line(0,1){28}}
\put(120,1){\line(0,1){28}}
\put(130,1){\line(0,1){28}}

\put(101,10){\line(1,0){8}}
\put(101,20){\line(1,0){8}}
\put(111,10){\line(1,0){8}}
\put(121,20){\line(1,0){8}}
\put(131,10){\line(1,0){8}}
\put(131,20){\line(1,0){8}}

\put(105,3){\ps{\Id_{k_*}}}
\put(105,13){\ps{\eta_k}}
\put(105,23){\ps{\id_g}}
\put(115,3){\ps{\eta_n}}
\put(115,18){\ps{\phi}}
\put(125,8){\ps{\psi}}
\put(125,23){\ps{\epsilon_f}}
\put(135,3){\ps{\id_n}}
\put(135,13){\ps{\epsilon_l}}
\put(135,23){\ps{\Id_{l_*}}}

\end{picture}
$$
(1) uses the Gray property of $ \chi $ twice on the left and twice on the right. (2) uses isotropy for $ \epsilon $ and the fact that $ \tau $ is an identity so the $ \id $ to the immediate right of $ \epsilon $ is also $ \Id $, and that $ \mu $ is the identity so that $ \frac{\id}{\id} $ is $ \id $. (3) is another application of the Gray property. (4) uses isotropy of $ \eta $ and $ \epsilon $ when again the $\id$'s to the left of $ \eta $ and right of $ \epsilon $ are also $\Id$'s because $ \tau = 1 $.

That $ F $ preserves vertical composition is similar, though easier because the functoriality of $ \eta $ and $\epsilon $ doesn't come into it, so there are fewer cells. It does have to be checked though to make sure that all the identities are in the right places. We leave this as an amusing exercise.

That $ F $ preserves horizontal and vertical composition of cubes is the same except that the diagrams represent cubes rather than basic cells. $ F $ preserves transversal composition of cubes trivially.

It remains only to show that $ F  $ gives a bijection between quintets and cells in $ {\ic B} $.

For any cell
$$
\bfig\scalefactor{.7}

\square[A `B `C `D;f_* `g `h `k_*]

\place(250,250)[\scriptstyle \psi]

\efig
$$
define
$$
\setlength{\unitlength}{.8mm}
\begin{picture}(40,30)

\put(0,11){\ps{G(\psi)}}
\put(14,11){\ps{=}}

\put(25,8){\line(1,0){8}}
\put(25,16){\line(1,0){8}}

\put(29,3){\ps{\epsilon_k}}
\put(29,11){\ps{\psi}}
\put(29,19){\ps{\eta_f}}

\end{picture}
$$
$ G $ is defined on cubes by the same formula.
$$
\setlength{\unitlength}{.8mm}
\begin{picture}(160,40)

\put(0,18){\ps{GF(\phi)}}
\put(18,18){\ps{=}}

\put(40,1){\line(0,1){8}}
\put(40,11){\line(0,1){18}}
\put(50,11){\line(0,1){18}}
\put(50,31){\line(0,1){8}}

\put(31,10){\line(1,0){28}}
\put(31,20){\line(1,0){8}}
\put(31,30){\line(1,0){28}}
\put(51,20){\line(1,0){8}}

\put(35,3){\ps{\epsilon_k}}
\put(35,13){\ps{\eta_k}}
\put(35,23){\ps{\id_g}}
\put(40,33){\ps{\id_\Id}}
\put(45,18){\ps{\phi}}
\put(50,3){\ps{\id_\Id}}
\put(55,13){\ps{\id_h}}
\put(55,23){\ps{\epsilon_f}}
\put(55,33){\ps{\eta_f}}

\put(70,18){\ps{=}}
\put(70,15){\ps{\scriptstyle (1)}}

\put(90,1){\line(0,1){8}}
\put(90,11){\line(0,1){18}}
\put(100,11){\line(0,1){28}}

\put(81,10){\line(1,0){28}}
\put(81,20){\line(1,0){8}}
\put(81,30){\line(1,0){18}}
\put(101,20){\line(1,0){8}}
\put(101,30){\line(1,0){8}}

\put(85,3){\ps{\epsilon_k}}
\put(85,13){\ps{\eta_k}}
\put(85,23){\ps{\id_g}}
\put(90,33){\ps{\id_\Id}}
\put(95,18){\ps{\phi}}
\put(100,3){\ps{\id_\Id}}
\put(105,13){\ps{\id_h}}
\put(105,23){\ps{\epsilon_f}}
\put(105,33){\ps{\eta_f}}

\put(120,18){\ps{=}}
\put(120,15){\ps{\scriptstyle (2)}}

\put(140,1){\line(0,1){8}}
\put(140,11){\line(0,1){18}}
\put(150,11){\line(0,1){28}}

\put(131,10){\line(1,0){28}}
\put(131,20){\line(1,0){8}}
\put(131,30){\line(1,0){18}}
\put(151,20){\line(1,0){8}}

\put(135,3){\ps{\epsilon_k}}
\put(135,13){\ps{\eta_k}}
\put(135,23){\ps{\id_g}}
\put(140,33){\ps{\id_\Id}}
\put(145,18){\ps{\phi}}
\put(150,3){\ps{\id_\Id}}
\put(155,13){\ps{\id_h}}
\put(155,28){\ps{\id_f}}

\end{picture}
$$
$$
\setlength{\unitlength}{.8mm}
\begin{picture}(160,40)

\put(0,18){\ps{=}}
\put(0,15){\ps{\scriptstyle (3)}}

\put(20,1){\line(0,1){8}}
\put(20,11){\line(0,1){18}}

\put(11,10){\line(1,0){18}}
\put(11,20){\line(1,0){8}}
\put(11,30){\line(1,0){18}}

\put(15,3){\ps{\epsilon_k}}
\put(15,13){\ps{\eta_k}}
\put(15,23){\ps{\id_g}}
\put(20,33){\ps{\id_\Id}}
\put(25,3){\ps{\id_\Id}}
\put(25,18){\ps{\phi}}

\put(40,18){\ps{=}}
\put(40,15){\ps{\scriptstyle (4)}}

\put(60,1){\line(0,1){28}}

\put(51,10){\line(1,0){8}}
\put(51,20){\line(1,0){8}}
\put(51,30){\line(1,0){18}}
\put(61,10){\line(1,0){8}}

\put(55,3){\ps{\epsilon_k}}
\put(55,13){\ps{\eta_k}}
\put(55,23){\ps{\id_g}}
\put(60,33){\ps{\id_\Id}}
\put(65,3){\ps{\id_\Id}}
\put(65,18){\ps{\phi}}

\put(80,18){\ps{=}}
\put(80,15){\ps{\scriptstyle (5)}}

\put(100,1){\line(0,1){28}}

\put(91,20){\line(1,0){8}}
\put(91,30){\line(1,0){18}}
\put(101,10){\line(1,0){8}}

\put(95,8){\ps{\id_k}}
\put(95,23){\ps{\id_g}}
\put(100,33){\ps{\id_\Id}}
\put(105,3){\ps{\id_\Id}}
\put(105,18){\ps{\phi}}

\put(120,18){\ps{=}}
\put(120,15){\ps{\scriptstyle (6)}}

\put(131,16){\line(1,0){8}}
\put(131,24){\line(1,0){8}}

\put(135,9){\ps{\id_\Id}}
\put(135,18){\ps{\phi}}
\put(135,27){\ps{\id_\Id}}

\put(150,18){\ps{=}}
\put(150,15){\ps{\scriptstyle (7)}}

\put(160,18){\ps{\phi}}

\end{picture}
$$
Equations (1) and (4) use the Gray property, (2) and (5) the companion equations, (3) and (6) use that $ \mu $ is the identity, and (7) that $ \tau $ is the identity.

The proof that $ FG(\psi) = \psi $ is similar.

The proofs that $ F $ and $ G $ are inverse on cubes is completely similar except that the diagrams now represent cubes rather than basic cells.

\end{proof}

A close examination of the proofs reveals that we only used the Gray condition for $ {\ic B} $ in the form
$$
 \chi \left(\bfig\square/```/<100,120>[{\scriptstyle \id}`*`*`*;```]\efig\right) = 1 \ \ \ \ \mbox{and}\ \ \ \ \ \ \chi \left(\bfig\square/```/<100,120>[*`*`*`{\scriptstyle \id};```]\efig\right) = 1 .
$$
Yet, as $ {\ic B} $ is of the form $ {\ic A}_s $, it also satisfies the corresponding conditions with $ \Id $ instead of $ \id $. Thus we get the following result.

\begin{corollary} In the presence of conditions (a), (b), (c), (3d), $ {\ic B} $ satisfies the Gray condition if and only if

\noindent
$$ \chi \left(\bfig\square/```/<100,120>[{\scriptstyle \id}`*`*`*;```]\efig\right) = 1 \ \ \ \ \mbox{and}\ \ \ \ \ \ \chi \left(\bfig\square/```/<100,120>[*`*`*`{\scriptstyle \id};```]\efig\right) = 1
$$
or equivalently
$$
\chi \left(\bfig\square/```/<100,120>[{\scriptstyle \Id}`*`*`*;```]\efig\right) = 1 \ \ \ \ \mbox{and}\ \ \ \ \ \ \chi \left(\bfig\square/```/<100,120>[*`*`*`{\scriptstyle \Id};```]\efig\right) = 1. 
$$
\end{corollary}

\begin{remark} In condition (1e) above, the $\Id$'s appear to be in the wrong place. This has to do with the following fact. If we start with a $2$-category $ {\cal A} $ and construct the double category of quintets $ {\mathbb Q}{\cal A} $
$$
\bfig\scalefactor{.7}

\square[A `B `C `D;l `g `h `k]

\morphism(230,220)/<=/<0,175>[`;{\scriptstyle \alpha}]

\efig
$$
and then take the horizontal $2$-category of $ {\mathbb Q}{\cal A} $ (vertical arrows identities) we get $ {\cal A} $ back. But if we take the vertical $2$-category of $ {\mathbb Q}{\cal A} $ we get $ {\cal A}^{co} $, i.e. the $2$-cells are reversed. Taking co-quintets we get the opposite situation. Starting with Gray's conventions leads naturally to co-quintets. So that $ {\ic A}_s $ corresponds directly to $ {\ic A}_c $ but corresponds to $ {\ic A}^{co}_l $, i.e. $ {\ic A}_l $ with the vertical direction reversed. This switches the two rows in $ \chi $ and explains the positioning of $ \Id $ in (1e).
\end{remark}

\section{Spans in double categories} \label{spans}

\subsection{The intercategory of spans in a double category}\label{spans.1} 

\ 

Let $ {\mathbb A} $ be a (weak) double category with a lax choice of pullbacks, i.e.\ the diagonal functor
$$
\Delta : {\mathbb A} \to {\mathbb A}^{\bf P}
$$
has a (lax) right adjoint \cite{GP}, where $ {\bf P} $ is the category
$$
\bfig\scalefactor{.7}

\Dtriangle/`>`<-/<300,200>[1`0`2;``]

\efig
$$

In elementary terms, this means the following. Let $ {\mathbb A} = {\bf A}_2\  \threepppp/>`>`>/<400>^{p_1} |{m} _{p_2} \ {\bf A}_1\  \three/>`<-`>/^{\partial_0} | {\id}_{\partial_1} \ {\bf A}_0 $. Then $ {\bf A}_0 $ and $ {\bf A}_1 $ have pullbacks preserved by $ \partial_0 $ and $ \partial_1 $. Furthermore a choice of pullback has been made also preserved by $ \partial_0 $ and $ \partial_1 $. So a chosen pullback in $ {\bf A}_1 $ will look like

$$
\bfig\scalefactor{.8}

\square/>`@{>}|{\bb}`@{>}|{\bb}`>/[A\times_C B ` B`\ov{A}\times_{\ov{C}} \ov{B}`\ov{B};`v\times_z w`w`]

\morphism(0,500)/>/<250,250>[A \times_C B`A;]

\morphism(500,500)/>/<250,250>[B`C;]

\morphism(250,750)/>/<500,0>[A`C;]

\morphism(750,750)|r|/@{>}|{\bb}/<0,-500>[C`\ov{C};z]

\morphism(500,0)|r|/>/<250,250>[\ov{B}`\ov{C};]

\place(400,600)[\scriptstyle \alpha]

\place(250,250)[\scriptstyle \pi_2]

\morphism(250,750)|r|/-/<0,-200>[A`;]

\place(625,400)[\scriptstyle \beta]

\efig
$$
The point of choosing pullbacks is not a question of the axiom of choice, which we use unashamedly, but of choosing them compatibly with $ \partial_0, \partial_1 $. Mere preservation of the pullbacks by $ \partial_0 $ and $ \partial_1 $ doesn't imply that we can make such a choice, but under weak conditions on $ {\mathbb A} $, e.g. every horizontal isomorphism has a companion, it is possible to do so.

Given such a compatible choice, it follows that $ {\bf A}_2 = {\bf A}_1 \times_{{\bf A}_0} {\bf A}_1 $ has a choice of pullbacks compatible with $ p_1 $ and $p_2 $, i.e.\ compatible pairs of chosen pullbacks in $ {\bf A}_1 $ give a pullback in $ {\bf A}_2 $. We are not assuming that $ m : {\bf A}_2 \to {\bf A}_1 $ preserves pullbacks, but there is always a universally given comparison cell
$$
\bfig\scalefactor{.7}

\square/`@{>}|{\bb}``=/<700,500>[\ov{A}\times_{\ov{c}} \ov{B} ``\ovv{A}\times_{\ov{c}} \ovv{B}`\ovv{A}\times_{\ovv{c}}\ovv{B};`\ov{v}\times_{\ov{z}} \ov{w}``]

\square(0,500)/=`@{>}|{\bb}``/<700,500>[A \times_C B ` A \times_C B`\ov{A} \times_{\ov{c}} \ov{B}`;`v\times_z w``]

\square(0,0)/=``@{>}|{\bb}`=/<700,1000>[A\times_C B `A \times_C B`\ovv{A}\times_{\ov{c}} \ovv{B}`\ovv{A}\times_{\ovv{c}}\ovv{B};``(v\cdot \ov{v})\times_{(z\cdot\ov{z})} (w\cdot \ov{w})`]

\place(350,500)[\scriptstyle \gamma]

\efig
$$
So if we do have a compatible choice of pullbacks it's always lax, as with any right adjoint. The word ``lax'' in ``lax choice of pullback'' is just to emphasize that it is not strong. If $ \id : {\bf A}_0 \to {\bf A}_1 $ preserves pullbacks (not necessarily the chosen ones) we say that $ {\mathbb A} $ has a lax {\em normal} choice of pullbacks. This is almost always the case and it is a weak condition to impose. (This is a special case of what we had called a lax functorial choice of $ {\mathbb I} $-limits in \cite{GP99} when our understanding of the notion was still evolving.) On the other hand, $ m : {\bf A}_2 \to {\bf A}_1 $ is just as likely to preserve pullbacks as not. If $ \id $ and $ m $ preserve them we say we have a strong choice of pullback.

Assume now that $ {\mathbb A} $ has a lax choice of pullbacks. Let $ {\bf \Lambda} $ be the category
$$ 0 \leftarrow 2 \rightarrow 1.
$$
Then $ {\mathbb A}^{\bf \Lambda} $ is the double category whose objects are spans of horizontal arrows
$$
A_0 \leftarrow A_2 \rightarrow A_1,
$$
whose horizontal arrows are commutative diagrams of horizontal arrows
$$
\bfig\scalefactor{.7}

\square/<-`>`>`<-/[A_0 `A_2 `B_0 `B_2;```]

\square(500,0)/>``>`>/[A_2`A_1`B_2`B_1;```]

\place(1650,250)[(a)]

\place(-650,250)[\ ]

\efig
$$
whose vertical arrows are spans of cells
$$
\bfig\scalefactor{.7}

\square/<-`@{>}|{\bb}`@{>}|{\bb}`<-/[A_0`A_2`\ov{A_0} `\ov{A_2};```]

\square(500,0)/>``@{>}|{\bb}`>/[A_2`A_1`\ov{A_2} `\ov{A_1};```]

\place(250,250)[\scriptstyle \alpha_0]

\place(750,250)[\scriptstyle \alpha_1]

\place(1650,250)[(b)]

\place(-650,250)[\ ]

\efig
$$
and whose double cells are commutative diagrams of cells
$$
\bfig\scalefactor{.7}

\square(0,300)/<-`@{>}|{\bb}``/[A_0`A_2`\ov{A_0}`;```]

\square(300,0)/<-`@{>}|{\bb}`@{>}|{\bb}`<-/[B_0`B_2`\ov{B_0}`\ov{B_2};```]

\morphism(0,800)|b|<300,-300>[A_0`B_0;]

\morphism(500,800)<300,-300>[A_2 ` B_2;]

\morphism(0,300)|b|<300,-300>[\ov{A_0}`\ov{B_0};]

\morphism(500,800)<500,0>[A_2`A_1;]

\morphism(1000,800)<300,-300>[A_1`B_1;]

\square(800,0)/>`@{>}|{\bb}`@{>}|{\bb}`>/[B_2`B_1`\ov{B_2}`\ov{B_1};```]

\place(550,250)[\scriptstyle \beta_0]

\place(1050,250)[\scriptstyle \beta_1]

\place(150,400)[\scriptstyle \phi_0]

\place(290,700)[\scriptstyle \alpha_0]

\place(600,600)[\scriptstyle \phi_2]

\morphism(500,780)/-/<0,-250>[`;]

\place(800,700)[\scriptstyle \alpha_1]

\place(1100,600)[\scriptstyle \phi_1]

\morphism(1000,780)/-/<0,-250>[`;]

\place(1800,250)[(c)]

\place(-500,250)[\ ]

\efig
$$
That is
$$ 
{\mathbb A}^{\bf \Lambda} = {\bf A}_2^{\bf \Lambda} \ \threepppp/>`>`>/<400>^{} |{} _{} \  {\bf A}_1^{\bf \Lambda} \   \three/>`<-`>/^{} | {}_{}\ {\bf A}_0^{\bf \Lambda}
$$
We have strict double functors $ \partial_0, \partial_1 : {\mathbb A}^{\bf \Lambda} \to {\mathbb A} $. $ \partial_0 $ picks out the $ 0 $ part of the diagram, and $ \partial_1 $ the $1 $ part. They are induced by the corresponding functors $ {\mathbbm 1} \two^{\ulcorner 0\urcorner}_{\ulcorner 1\urcorner} {\bf \Lambda} $. We also have a strict double functor $\id : {\mathbb A} \to {\mathbb A}^{\bf \Lambda} $ coming from $ {\bf \Lambda} \to {\mathbbm 1} $.

The pullback $ {\mathbb A}^{\bf \Lambda} \times_{\mathbb A} {\mathbb A}^{\bf \Lambda} $ is $ {\mathbb A}^{\bf M} $ where $ {\bf M} $ is the category
$$
0 \leftarrow 3 \rightarrow 1 \leftarrow 4 \rightarrow 2.
$$
The pullback (lax) functor $ {\mathbb A}^{\bf P} \to {\mathbb A} $ induces a lax functor $ m : {\mathbb A}^{\bf M} \to {\mathbb A}^{\bf \Lambda} $ and produces a pseudocategory
$$
{\mathbb A}^{\bf M} \ \threepppp/>`>`>/<400>^{p_1} |{m} _{p_2} \ {\mathbb A}^{\bf \Lambda} \three/>`<-`>/^{\partial_0} | {\id}_{\partial_1}\  {\mathbb A}
$$
in $ \Dlax $. In this way we get an intercategory, that we call ${\ic Span} ({\mathbb A}) $. As a double pseudocategory in ${\cal CAT} $, it looks like
$$
\bfig\scalefactor{.7}

\square/>`<-`<-`>/[{\bf A}_1^{\bf M} `{\bf A}_1^{\bf \Lambda}`{\bf A}_2^{\bf M} `{\bf A}_2^{\bf \Lambda};```]

\square|allb|/@{>}@<-3pt>`@{<-}@<-3pt>`@{<-}@<-3pt>`@{>}@<-3pt>/[{\bf A}_1^{\bf M} `{\bf A}_1^{\bf \Lambda}`{\bf A}_2^{\bf M} `{\bf A}_2^{\bf \Lambda};```]

\square|allb|/@{>}@<3pt>`@{<-}@<3pt>`@{<-}@<3pt>`@{>}@<3pt>/[{\bf A}_1^{\bf M} `{\bf A}_1^{\bf \Lambda}`{\bf A}_2^{\bf M} `{\bf A}_2^{\bf \Lambda};```]

\square(0,500)/>`>`>`>/[{\bf A}_0^{\bf M}`{\bf A}_0^{\bf \Lambda}`{\bf A}_1^{\bf M}`{\bf A}_1^{\bf \Lambda};```]

\square(0,500)|allb|/@{>}@<-3pt>`@{<-}@<-3pt>`@{<-}@<-3pt>`@{>}@<-3pt>/[{\bf A}_0^{\bf M}`{\bf A}_0^{\bf \Lambda}`{\bf A}_1^{\bf M}`{\bf A}_1^{\bf \Lambda};```]

\square(0,500)|allb|/@{>}@<3pt>`@{<-}@<3pt>`@{<-}@<3pt>`@{>}@<3pt>/[{\bf A}_0^{\bf M}`{\bf A}_0^{\bf \Lambda}`{\bf A}_1^{\bf M}`{\bf A}_1^{\bf \Lambda};```]

\square(500,0)/<-`<-`<-`<-/[{\bf A}_1^{\bf \Lambda} `{\bf A}_1`{\bf A}_2^{\bf \Lambda} `{\bf A}_2;```]

\square(500,0)|allb|/@{>}@<-3pt>`@{<-}@<-3pt>`@{<-}@<-3pt>`@{>}@<-3pt>/[{\bf A}_1^{\bf \Lambda} `{\bf A}_1`{\bf A}_2^{\bf \Lambda} `{\bf A}_2;```]

\square(500,0)|allb|/@{>}@<3pt>`@{<-}@<3pt>`@{<-}@<3pt>`@{>}@<3pt>/[{\bf A}_1^{\bf \Lambda} `{\bf A}_1`{\bf A}_2^{\bf \Lambda} `{\bf A}_2;```]

\square(500,500)/<-`>`>`<-/[{\bf A}_0^{\bf \Lambda} `{\bf A}_0`{\bf A}_1^{\bf \Lambda}`{\bf A}_1;```]

\square(500,500)|allb|/@{>}@<-3pt>`@{<-}@<-3pt>`@{<-}@<-3pt>`@{>}@<-3pt>/[{\bf A}_0^{\bf \Lambda} `{\bf A}_0`{\bf A}_1^{\bf \Lambda}`{\bf A}_1;```]

\square(500,500)|allb|/@{>}@<3pt>`@{<-}@<3pt>`@{<-}@<3pt>`@{>}@<3pt>/[{\bf A}_0^{\bf \Lambda} `{\bf A}_0`{\bf A}_1^{\bf \Lambda}`{\bf A}_1;```]

\efig
$$
whose rows are the double categories $\SSpan{\bf A}_0 $, $\SSpan{\bf A}_1 $ and $\SSpan{\bf A}_2 $ of \cite{DPP}. Recall from there that an arbitrary functor $ F : {\bf B} \to {\bf C} $ between categories with pullbacks induces a colax normal functor
$$
\SSpan F : \SSpan {\bf B} \to \SSpan{\bf C}.
$$
If $ F $ preserves pullbacks, then $ \SSpan F $ is a strong functor. To make $ \SSpan{\bf B} $ and $ \SSpan{\bf C} $ into double categories, a choice of pullback must be made (to define vertical composition). If $ F $ preserves these choices, then $ \SSpan F $ is a strict functor.

In this way we get the alternative description of ${\ic Span}({\mathbb A}) $ as a pseudocategory object
$$
\SSpan {\bf A}_2\  \threepppp/>`>`>/<400>^{} |{} _{} \ \SSpan {\bf A}_1 \ \three/>`<-`>/^{} | {}_{} \ \SSpan {\bf A}_0
$$
in $ \Dcolax $.

Referring to the table of Section 4 of \cite{ICI} we get a more detailed description of $ {\ic Span} ({\mathbb A}) $. Its

\noindent (1) objects are those of $ {\mathbb A} $,

\noindent (2) transversal arrows are the horizontal morphisms of $ {\mathbb A} $,

\noindent (3) horizontal arrows are spans of horizontal morphisms of $ {\mathbb A} $,

\noindent (4) vertical arrows are the vertical morphisms of $ {\mathbb A} $,

\noindent (5) horizontal cells are commutative diagrams as in (a) above,

\noindent (6) vertical cells are the double cells of $ {\mathbb A} $,

\noindent (7) basic cells are spans of double cells as in (b) above,

\noindent (8) cubes are commutative diagrams of double cells as in (c).

Compositions are obvious, either coming from $ {\mathbb A} $ or the composition of spans $ \otimes $. To see how interchange works, consider the following diagram in $ {\mathbb A} $:
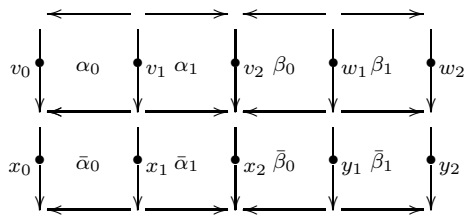
$$
\bfig\scalefactor{.9}

\square/<-`@{>}|{\bb}`@{>}|{\bb}`<-/<350,350>[```;`x_0`x_1`]

\place(175,175)[\scriptstyle \ov{\alpha}_0]

\square(0,350)/<-`@{>}|{\bb}`@{>}|{\bb}`<-/<350,350>[```;`v_0`v_1`]

\place(175,525)[\scriptstyle \alpha_0]

\square(350,0)/>``@{>}|{\bb}`>/<350,350>[```;``x_2`]

\place(525,175)[\scriptstyle \ov{\alpha}_1]

\square(350,350)/>``@{>}|{\bb}`/<350,350>[```;``v_2`]

\place(525,525)[\scriptstyle \alpha_1]

\square(700,0)/<-``@{>}|{\bb}`<-/<350,350>[```;``y_1`]

\place(875,175)[\scriptstyle \ov{\beta}_0]

\square(700,350)/<-``@{>}|{\bb}`/<350,350>[```;``w_1`]

\place(875,525)[\scriptstyle \beta_0]

\square(1050,0)/>``@{>}|{\bb}`>/<350,350>[```;``y_2`]

\place(1225,175)[\scriptstyle \ov{\beta}_1]

\square(1050,350)/>``@{>}|{\bb}`/<350,350>[```;``w_2`]

\place(1225,525)[\scriptstyle \beta_1]

\efig
$$
To calculate $ \dfrac{\alpha | \beta}{\ov{\alpha} | \ov{\beta}} $ we take the pullbacks
$$
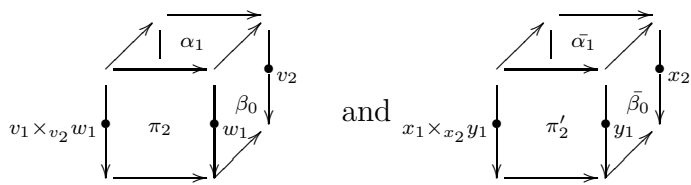
\bfig\scalefactor{.7}

\square/>`@{>}|{\bb}`@{>}|{\bb}`>/[```;`v_1\times_{v_2} w_1`w_1`]

\morphism(0,500)/>/<250,250>[`;]

\morphism(500,500)/>/<250,250>[`;]

\morphism(250,750)/>/<500,0>[`;]

\morphism(750,750)|r|/@{>}|{\bb}/<0,-500>[`;v_2]

\morphism(500,0)/>/<250,250>[`;]

\place(400,640)[\scriptstyle \alpha_1]

\place(250,250)[\scriptstyle \pi_2]

\morphism(250,750)|r|/-/<0,-200>[`;]


\place(650,350)[\scriptstyle \beta_0]

\place(1200,300)[\mbox{and}]

\square(1800,0)/>`@{>}|{\bb}`@{>}|{\bb}`>/[```;`x_1\times_{x_2} y_1`y_1`]

\morphism(1800,500)/>/<250,250>[`;]

\morphism(2300,500)/>/<250,250>[`;]

\morphism(2050,750)/>/<500,0>[`;]

\morphism(2550,750)|r|/@{>}|{\bb}/<0,-500>[`;x_2]

\morphism(2300,0)/>/<250,250>[`;]

\place(2450,350)[\scriptstyle \ov{\beta_0}]

\place(2080,250)[\scriptstyle \pi'_2]

\place(2200,640)[\scriptstyle \ov{\alpha_1}]


\morphism(2050,750)/-/<0,-200>[`;]

\efig
$$
and then compose the left half and right halves of
$$
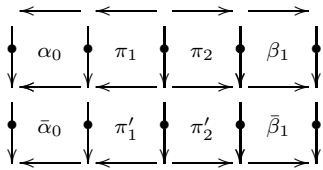
\bfig\scalefactor{.7}

\square/<-`@{>}|{\bb}`@{>}|{\bb}`<-/<350,350>[```;```]

\place(175,175)[\scriptstyle \ov{\alpha}_0]

\square(0,350)/<-`@{>}|{\bb}`@{>}|{\bb}`<-/<350,350>[```;```]

\place(175,525)[\scriptstyle \alpha_0]

\square(350,0)/<-``@{>}|{\bb}`<-/<350,350>[```;```]

\place(525,175)[\scriptstyle \pi'_1]

\square(350,350)/<-``@{>}|{\bb}`/<350,350>[```;```]

\place(525,525)[\scriptstyle \pi_1]

\square(700,0)/>``@{>}|{\bb}`>/<350,350>[```;```]

\place(875,175)[\scriptstyle \pi'_2]

\square(700,350)/>``@{>}|{\bb}`/<350,350>[```;```]

\place(875,525)[\scriptstyle \pi_2]

\square(1050,0)/>`@{>}|{\bb}`@{>}|{\bb}`>/<350,350>[```;```]

\place(1225,175)[\scriptstyle \ov{\beta}_1]

\square(1050,350)/>`@{>}|{\bb}`@{>}|{\bb}`/<350,350>[```;```]

\place(1225,525)[\scriptstyle \beta_1]

\efig
$$
whereas to calculate $ \left.\dfrac{\alpha}{\ov{\alpha}}\right|\dfrac{\beta}{\ov{\beta}} $ we take the pullback
$$
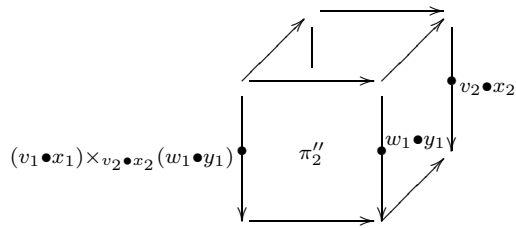
\bfig\scalefactor{.9}

\square/>`@{>}|{\bb}`@{>}|{\bb}`>/[```;`(v_1 \bullet x_1)\times_{v_2\bullet x_2} (w_1\bullet y_1)``]

\morphism(0,500)/>/<250,250>[`;]

\morphism(500,500)/>/<250,250>[`;]

\morphism(250,750)/>/<500,0>[`;]

\morphism(750,750)|r|/@{>}|{\bb}/<0,-500>[`;v_2\bullet x_2]

\morphism(500,0)/>/<250,250>[`;]


\place(250,250)[\scriptstyle \pi''_2]

\morphism(250,750)|r|/-/<0,-200>[`;]

\place(620,300)[\scriptstyle w_1\bullet y_1]

\efig
$$
and then compose left and right parts of
$$
\bfig\scalefactor{.7}

\square/<-`@{>}|{\bb}`@{>}|{\bb}`<-/<350,350>[```;```]

\place(175,175)[\scriptstyle \ov{\alpha}_0]

\square(0,350)/<-`@{>}|{\bb}`@{>}|{\bb}`<-/<350,350>[```;```]

\place(175,525)[\scriptstyle \alpha_0]

\square(350,0)/<-``>`<-/<350,700>[```;```]

\place(525,350)[\scriptstyle \pi''_1]

\square(700,0)/>```>/<350,700>[```;```]

\place(875,350)[\scriptstyle \pi''_2]

\square(1050,0)/>`@{>}|{\bb}`@{>}|{\bb}`>/<350,350>[```;```]

\place(1225,175)[\scriptstyle \ov{\beta}_1]

\square(1050,350)/>`@{>}|{\bb}`@{>}|{\bb}`/<350,350>[```;```]

\place(1225,525)[\scriptstyle \beta_1]

\efig
$$
The comparison
$$
\gamma : (v_1 \times_{v_2} w_1) \bullet (x_1 \times_{x_2} y_1) \to (v_1 \bullet x_1)\times_{v_2\bullet x_2} (w_1 \bullet y_1)
$$
gives a morphism of spans in $ {\bf A}_2 $ which is
$$
\chi : \dfrac{\alpha | \beta}{\ov{\alpha} | \ov{\beta}} \to \left.\dfrac{\alpha}{\ov{\alpha}}\right|\dfrac{\beta}{\ov{\beta}}.
$$
For the degenerate interchangers $ \mu,\delta, \tau $ we have the following. The horizontal identity $ \id_v $ is
$$
\bfig\scalefactor{.7}

\square/<-`@{>}|{\bb}`@{>}|{\bb}`<-/[A`A`\ov{A}`\ov{A};1`v`v`1]

\place(250,250)[\scriptstyle 1_v]

\square(500,0)/>``@{>}|{\bb}`>/[A`A`\ov{A}`\ov{A};1``v`1]

\place(750,250)[\scriptstyle 1_v]

\efig
$$
and $ \id_v \bullet \id_\ov{v} = \id_{v\bullet \ov{v}} $ so $\mu $ is the identity. The vertical identity $ \Id_f $ is
$$
\bfig\scalefactor{.7}

\square/<-`@{>}|{\bb}`@{>}|{\bb}`<-/[A_0 `A_2`A_0`A_2;f_o`\id `\id`f_0]

\place(250,250)[\scriptstyle \id_{f_0}]

\square(500,0)/>``@{>}|{\bb}`>/[A_2`A_1`A_2`A_1;f_1``\id`f_1]

\place(770,250)[\scriptstyle \id_{f_1}]

\efig
$$
and horizontal composition of two of these is done by taking the pullback of $ \id_{f_1} $ with $ \id_{g_0} $. If pullbacks in $ {\mathbb A} $ are not normal we get a nontrivial comparison
$$
\id_{A_2 \times_{A_1} B_2} \to \id_{A_2} \times_{\id_{A_1}} \id_{g_2}
$$
which gives a nontrivial
$$
\delta : \Id_{f\otimes g} \to \Id_f \otimes \Id_g.
$$
Finally $ \id_{\Id_A} $ is
$$
\bfig\scalefactor{.7}

\square/<-`@{>}|{\bb}`@{>}|{\bb}`<-/[A`A`A`A;1`\id`\id`1]

\place(250,250)[\scriptstyle 1_{\id_A}]

\square(500,0)/>``@{>}|{\bb}`>/[A`A`A`A;1``\id`1]

\place(750,250)[\scriptstyle 1_{\id_A}]

\efig
$$
and $ \Id_\id $ is
$$
\bfig\scalefactor{.7}

\square/<-`@{>}|{\bb}`@{>}|{\bb}`<-/[A`A`A`A;1`\id`\id`1]

\place(250,250)[\scriptstyle \id_{1_A}]

\square(500,0)/>``@{>}|{\bb}`>/[A`A`A`A;1``\id`1]

\place(750,250)[\scriptstyle \id_{1_A}]

\efig
$$
so they are equal and $ \tau : \Id_{\id_A} \to \id_{\Id_A} $ is the identity.

If $ {\mathbb A} $ and $ {\mathbb B} $ have a lax choice of pullbacks and $ F : {\mathbb A} \to {\mathbb B} $ is a lax (colax, pseudo) double functor (not necessarily preserving the pullbacks) then we get a colax-lax (resp. colax-colax, colax-pseudo) morphism of intercategories
$$
{\ic Span} (F) : {\ic Span} ({\mathbb A}) \to {\ic Span} ({\mathbb B})
$$
by applying $ F $ component-wise. The commutativity conditions in (a) and (c) above involve only horizontal composition which is strictly preserved by $ F $. If $ F $ preserves pullbacks then $ {\ic Span} (F) $ will be pseudo-lax (resp. pseudo-colax, pseudo-pseudo).

\subsection{Double spans}\label{spans.2}

\ 

We now consider some specific examples. The first is the double category $ \SSpan {\bf A} $ for a category with pullbacks. It has a strong choice of pullbacks. ${\ic Span}({\mathbb S}{\rm pan} {\bf A}) $ is an important construction and deserves a special name ${\ic Span^2}{\bf A} $. A general cube looks like
$$
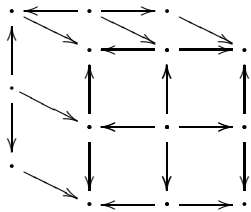
\bfig\scalefactor{.5}

\square/<-`>`>`<-/[\cdot `\cdot `\cdot `\cdot;```]

\square(500,0)[\cdot`\cdot`\cdot`\cdot;```]

\square(0,500)/<-`<-`<-`<-/[\cdot `\cdot `\cdot `\cdot;```]

\square(500,500)/>`<-`<-`>/[\cdot `\cdot `\cdot `\cdot;```]

\morphism(0,1250)/>/<-500,0>[\cdot`\cdot;]

\morphism(-500,1250)/>/<500,-250>[\cdot`\cdot;]

\morphism(0,1250)/>/<500,0>[\cdot`\cdot;]

\morphism(0,1250)/>/<500,-250>[\cdot`\cdot;]

\morphism(500,1250)/>/<500,-250>[\cdot`\cdot;]

\morphism(-500,1250)/<-/<0,-500>[\cdot`\cdot;]

\morphism(-500,750)/>/<500,-250>[`\cdot;]

\morphism(-500,750)/>/<0,-500>[`\cdot;]

\morphism(-500,250)/>/<500,-250>[\cdot`\cdot;]

\efig
$$
the front and back, the basic cells, being spans of spans. All the interchangers are isomorphisms.

If $ {\bf B} $ is another category with pullbacks and $ F : {\bf A} \to {\bf B} $ an arbitrary functor, then we get a colax-colax functor ${\ic Span^2} F : {\ic Span^2} {\bf A} \to {\ic Span^2} {\bf B} $. If $ U : {\bf B} \to {\bf A} $ is right adjoint to $ F $, then $ U $ preserves pullbacks and induces a pseudo-pseudo functor $ {\ic Span^2} U : {\ic Span^2} {\bf B} \to {\ic Span^2} {\bf A} $. We can consider it as a colax-lax functor and as such it is a conjoint to $ {\ic Span^2} F $ in $ {\ic ICat} $.

\subsection{Matrices in a monoidal category revisited}\label{spans.3}

\

A related example is the intercategory $ {\ic SM}({\bf V}) $ of Section \ref{duoidal}. For a monoidal category $({\bf V}, \otimes, I) $ with coproducts over which $ \otimes $ distributes, we have the double category $ {\bf V}$-${\mathbb S}{\rm et} $, introduced in \cite{P}, of sets, functions and $ {\bf V}$-matrices. If $ {\bf V} $ has pullbacks, then $ {\bf V}$-${\mathbb S}{\rm et} $ has a lax choice of pullbacks and $ {\ic Span}({\bf V}\mbox{-}{\mathbb S}{\rm et}) $ is $ {\ic SM}({\bf V}) $. 

\subsection{Spans of cospans}\label{spans.4}

\ 

An equally interesting example is the following. Let $ {\bf A} $ be a category with pullbacks and pushouts. Then the double category of cospans $ {\mathbb C}{\rm osp} {\bf A} $ has a lax (normal) choice of pullbacks. This gives an intercategory ${\ic Span} ({\mathbb C}{\rm osp} {\bf A}) $ which we will call ${\ic SpanCosp} {\bf A} $. A basic cell is a span of cospans and a general cube is a commutative diagram
$$
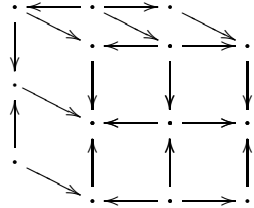
\bfig\scalefactor{.5}

\square/<-`<-`<-`<-/[\cdot `\cdot `\cdot `\cdot;```]

\square(500,0)/>`<-`<-`>/[\cdot`\cdot`\cdot`\cdot;```]

\square(0,500)/<-`>`>`<-/[\cdot `\cdot `\cdot `\cdot;```]

\square(500,500)/>`>`>`>/[\cdot `\cdot `\cdot `\cdot;```]

\morphism(0,1250)/>/<-500,0>[\cdot`\cdot;]

\morphism(-500,1250)/>/<500,-250>[\cdot`\cdot;]

\morphism(0,1250)/>/<500,0>[\cdot`\cdot;]

\morphism(0,1250)/>/<500,-250>[\cdot`\cdot;]

\morphism(500,1250)/>/<500,-250>[\cdot`\cdot;]

\morphism(-500,1250)/>/<0,-500>[\cdot`\cdot;]

\morphism(-500,750)/>/<500,-250>[`\cdot;]

\morphism(-500,750)/<-/<0,-500>[`\cdot;]

\morphism(-500,250)/>/<500,-250>[\cdot`\cdot;]

\efig
$$
Transversal composition is just composition in $ {\bf A} $, horizontal composition is span composition and given by pullback, and vertical composition is cospan composition given by pushout. The $ \chi $ is almost never an isomorphism but the other interchangers $ \mu, \delta, \tau $ always are. More details can be found in \cite{GP-Multiple}.

$ {\ic SpanCosp}{\bf A} $ is closely related to the product-coproduct duoidal category. Let ${\bf A} $ have finite limits and colimits and $ {\ic A} $ be the intercategory obtained from the duoidal category $ ({\bf A}, \times, 1, +, 0) $ as in Section \ref{duoidal}. There are two canonical inclusions of $ {\ic A} $ into $ {\ic SpanCosp} {\bf A} $. $ F_0 $ which takes a basic cell of $ {\ic A} $
$$
\bfig\scalefactor{.5}

\square/=`=`=`=/[*`*`*`*;```]

\place(250,250)[\scriptstyle A]

\efig
$$
to the basic cell
$$
\bfig\scalefactor{.5}

\square/<-`<-`<-`<-/[1`A`0`0;```]

\square(0,500)/<-`>`>`<-/[0`0`1`A;```]

\square(500,0)/>`<-`<-`>/[A`1`0`0;```]

\square(500,500)/>`>`>`>/[0`0`A`1;```]

\efig
$$
and $ F_1 $ which takes it to
$$
\bfig\scalefactor{.5}

\square/<-`<-`<-`<-/[1`A`1`0;```]

\square(0,500)/<-`>`>`<-/[1`0`1`A;```]

\square(500,0)/>`<-`<-`>/[A`1`0`1;```]

\square(500,500)/>`>`>`>/[0`1`A`1;```]

\efig
$$

$ F_0 $ is pseudo-lax and $ F_1 $ is colax-pseudo. There is also a canonical morphism $ G : {\ic SpanCosp} {\bf A} \to {\ic A} $ which picks out the middle object of a span of cospans
$$
\bfig\scalefactor{.5}

\square/<-`<-`<-`<-/[C`X`B`T;```]

\square(0,500)/<-`>`>`<-/[A`S`C`X;```]

\square(500,0)/>`<-`<-`>/[X`C'`T`B';```]

\square(500,500)[S`A'`X`C';```]

\place(1500,500)[\to/|->/^G]

\square(2000,250)/=`=`=`=/[*`*`*`*;```]

\place(2250,500)[\scriptstyle X]

\efig
$$

Although $ GF_0 $ and $ GF_1 $ are the identity on $ {\ic A} $, there is no conjointness or adjointness relationship between the $F$'s and $G$, in contrast with the situation of $ {\ic SM}({\bf V}) $ of Section \ref{duoidal}.

By duality we can start with a double category $ {\mathbb A} $ with a colax choice of pushouts and construct an intercategory of cospans. Because the interchanger $ \chi $ has a given direction, dualization in the transversal direction forces the switching of horizontal and vertical. So $ {\ic Cosp} {\mathbb A} $ has cospans of horizontal arrows of $ {\mathbb A} $ as its vertical arrows. The transversal arrows of $ {\ic Cosp} {\mathbb A} $ are the horizontal morphisms of $ {\mathbb A} $ and the horizontal arrows of $ {\ic Cosp} {\mathbb A} $ are the vertical arrows of $ {\mathbb A} $.

If $ {\bf A} $ is a category with pushouts and pullbacks then we can form $ {\ic Cosp} (\SSpan {\bf A}) $ and it is exactly the same as $ {\ic Span } ({\mathbb C}{\rm osp} {\bf A}) $, i.e.\ they are both ${\ic SpanCosp} {\bf A} $. De Francesco Albasini, Sabadini, Walters, have used $ {\ic SpanCosp} {\bf A} $ (for ${\bf A} $ the category of graphs) in their work on sequential and parallel composition \cite{dFASW}. 
The interchanger $ \chi $ is introduced and  its non invertibility is given a computer science justification in terms of synchronization. They mention 
``appropriate (lax monoidal) coherence equations'' but don't go into details.

If $ {\bf A} $ is a category with pushouts, we can form the intercategory of {\em double cospans}, $ {\ic Cosp^2}({\bf A}) = {\ic Cosp}({\mathbb C}{\rm osp}{\bf A}) $. Double cospans were introduced by Morton in the context of quantum field theory, first as an arXiv preprint in 2006 and later published as \cite{VM}. They were presented as Verity double bicategories so there were no transversal arrows. This was taken up in \cite{G1}, where higher cospans (and spans) were introduced, including their transversal morphisms, which are important for us here.

\subsection{Profunctors and spans in ${\cal C}{\it at} $}\label{spans.5}

\ 

An interesting example of a double category with a lax choice of pullbacks is $ \CCat $, the double category whose objects are small categories, horizontal arrows functors and vertical arrows profunctors. We choose the usual construction for pullbacks in $ {\bf Cat} $, viz.\ pairs of objects and pairs of arrows. Given double cells in $ \CCat $
$$
\bfig\scalefactor{.7}

\square/>`@{>}|{\bb}`@{>}|{\bb}`>/[{\bf B} ` {\bf A} `\ov{\bf B} `\ov{\bf A};F `R`P`\ov{F}]

\morphism(200,250)/=>/<175,0>[`;\rho]

\square(500,0)/<-``@{>}|{\bb}`<-/[{\bf A} `{\bf C} `\ov{\bf A} `\ov{\bf C};G``S`\ov{G}]

\morphism(700,250)/<=/<175,0>[`;\sigma]

\efig
$$
we take $ R \times_P S : {\bf B} \times_{\bf A} {\bf C} \tod \ov{\bf B} \times_{\ov{\bf A}} \ov{\bf C} $ to be
$$
R \times_P S ((B,C), (\ov{B},\ov{C})) = \{ (x,y) | x \in R(B,\ov{B}), y \in S (C,\ov{C}), \rho(x) = \sigma(y)\}
$$
We can represent such an element as
$$
(x, y) : (B, C) \tod (\ov{B}, \ov{C})
$$
where $ x :  B \todu{R} \ov{B}  $ and $ y : C \todu{S} \ov{C} $ and
$$
\left(\rho(x) : FB \todu{P} \ov{F}\ov{B}\right) = \left(\sigma(y) : GC \todu{P} \ov{G}\ov{C}\right).
$$
The identities are the hom functors and from the definition of morphisms in the pullbacks we see that $ \id_{\bf B} \times_{\id_{\bf A}} \id_{\bf C} = \id_{{\bf B} \times_{\bf A} {\bf C}} $, i.e.\ we have a normal choice of pullback.

On the other hand pullbacks are not strong but merely lax (normal) as the following example shows:
$$
\bfig\scalefactor{.7}

\square/>`@{>}|{\bb}`@{>}|{\bb}`>/[{\mathbbm 1} `{\mathbbm 2} `{\mathbbm 1} `{\mathbbm 1};0`\ov{R}`1`]

\morphism(200,250)/=>/<175,0>[`;]

\square(500,0)/<-``@{>}|{\bb}`<-/[{\mathbbm 2} `{\mathbbm 1} `{\mathbbm 1} `{\mathbbm 1};1``\ov{S}`]

\morphism(700,250)/<=/<175,0>[`;]

\square(0,500)/>`@{>}|{\bb}`@{>}|{\bb}`/[{\mathbbm 1}`{\mathbbm 1} `{\mathbbm 1} `{\mathbbm 2};`R`1`]

\morphism(200,750)/=>/<175,0>[`;]

\square(500,500)/<-``@{>}|{\bb}`<-/[{\mathbbm 1} `{\mathbbm 1} `{\mathbbm 2} `{\mathbbm 1};``S`]

\morphism(700,750)/<=/<175,0>[`;]

\efig
$$
A profunctor $ {\mathbbm 1} \tod {\mathbbm 1} $ is just a set and $ R \otimes \ov{R} = R \times \ov{R}$, $ S \otimes \ov{S} = S \times \ov{S} $. $ 1 : {\mathbbm 1} \tod {\mathbbm 2} $ and $ 1 : {\mathbbm 2} \tod {\mathbbm 1} $ are the constant profunctors with value $ 1 $ and $ 1 \otimes 1 = 1 $. So $ (R \otimes \ov{R}) \times_{1\otimes 1} (S \otimes \ov{S}) = R \times \ov{R} \times S \times \ov{S}$. But $ (R\times_1 S) \otimes (\ov{R} \times_1 \ov{S}) $ is the composite
$$
{\mathbbm 1} \tod {\bf 0} \tod {\mathbbm 1}
$$
which is $ 0 $. Thus the intercategory $ {\ic Span}(\CCat) $ has a non invertible interchanger $ \chi $.

A profunctor $ P : {\bf A} \tod {\bf B} $ may be viewed as a discrete fibration from $ {\bf A} $ to $ {\bf B} $ (also called a discrete bifibration over $ {\bf A} \times {\bf B} $)
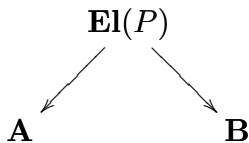
$$
\bfig\scalefactor{.7}

\Atriangle/>`>`/[{\bf El}(P)`{\bf A}`{\bf B};``]

\efig
$$
which is a span in $ {\bf Cat} $. The ``element construction'' is in fact a lax embedding of double categories
$$
\CCat \to/^{ (}->/ \SSpan ({\bf Cat})
$$
$$
\bfig\scalefactor{.7}

\square(0,500)/>`@{>}|{\bb}`@{>}|{\bb}`>/[{\bf A} `{\bf B} `{\bf C} `{\bf D};F `P `Q `G]

\morphism(200,750)/=>/<175,0>[`;\scriptstyle \phi]

\place(850,750)[\longmapsto]

\square(1300,250)<700,500>[{\bf El}(P) `{\bf El}(Q) `{\bf C} `{\bf D};\phi ```]

\square(1300,750)/>`<-`<-`/<700,500>[{\bf A} `{\bf B} `{\bf El}(P) `{\bf El}(Q);F```]

\efig
$$ 
which preserves pullbacks. This gives an embedding of intercategories
$$
{\ic Span}(\CCat) \to/^{ (}->/ {\ic Span}\  \SSpan ({\bf Cat}) = {\ic Span}^2 ({\bf Cat})
$$
which is pseudo-lax. It is transversally full and faithful.

Looking at things from the other side, taking elements gives an embedding of intercategories from quintets into spans
$$
{\ic Q} (\CCat) \to/^{ (}->/ {\ic Span} (\CCat)
$$
$$
\bfig\scalefactor{.7}

\square/@{>}|{\bb}`@{>}|{\bb}`@{>}|{\bb}`@{>}|{\bb}/[{\bf A} `{\bf B} `{\bf C} `{\bf D};P `R `S `Q]

\morphism(250,200)/<=/<0,175>[`;\scriptstyle \phi]

\place (850,250)[\longmapsto]

\square(1200,0)/<-`@{>}|{\bb}`@{>}|{\bb}`<-/<700,500>[{\bf A} `{\bf El}(P) `{\bf C} `{\bf El}(Q);``{\bf El}(Q)`]

\square(1900,0)/>``@{>}|{\bb}`>/<700,500>[{\bf El}(P) `{\bf B} `{\bf El}(Q) `{\bf D};``S`]

\morphism(1450,250)/<=/<175,0>[`;]

\morphism(2260,250)/=>/<175,0>[`;]

\efig
$$
$ {\bf El}(Q) : {\bf El}(P) \tod {\bf El} (Q) $ is given by
$$
{\bf El}(P) (A \todo{f} B, C \todo{q} D) = \{ (A \todo{r} C, B \todo{s} D) | \phi (p \otimes s) = r \otimes q \}
$$
We write the condition $ \phi (p \otimes s) = r \otimes q $ symbolically as
$$
\bfig\scalefactor{.7}

\square/@{>}|{\bb}`@{>}|{\bb}`@{>}|{\bb}`@{>}|{\bb}/[A `B `C `D;p `r `s `q]

\morphism(245,160)/<-|/<0,200>[`;\scriptstyle \phi]

\place(2000,250)[(*)]

\place(-1500,250)[\ ]

\efig
$$
This embedding is lax-lax. That it is horizontally lax is not surprising as the category of elements construction is itself lax, but at first glance it might appear to be pseudo in the vertical direction. However, for a vertical composite, one finds that an element of $ {\bf El}(\phi) \otimes {\bf El}(\psi) $ is an equivalence class of squares
$$
\bfig\scalefactor{.7}

\square/@{>}|{\bb}`@{>}|{\bb}`@{>}|{\bb}`@{>}|{\bb}/[```;```]

\morphism(245,160)/<-|/<0,200>[`;\scriptstyle \psi]

\morphism(245,660)/<-|/<0,200>[`;\scriptstyle \phi]

\square(0,500)/@{>}|{\bb}`@{>}|{\bb}`@{>}|{\bb}`@{>}|{\bb}/[```;```]

\efig
$$
whereas an element of $ {\bf El}(\phi \otimes \psi) $ is a square
$$
\bfig\scalefactor{.7}

\square/@{>}|{\bb}`@{>}|{\bb}`@{>}|{\bb}`@{>}|{\bb}/[```;```]

\morphism(245,160)/<-|/<0,200>[`;\scriptstyle \theta]

\efig
$$
whose vertical sides are equivalence classes of pairs, thus giving a canonical comparison
$$
{\bf El}(\phi) \otimes {\bf El} (\psi) \to {\bf El}(\phi \otimes \psi)
$$
which is not an isomorphism in general.

This embedding, composed with the previous one, gives a lax-lax embedding
$$
{\ic Q} (\CCat)\to/^{ (}->/ {\ic Span}^2 ({\bf Cat})
$$
which is transversally full and faithful:

$$
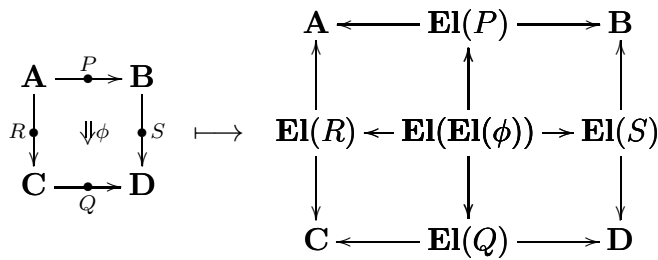
\bfig\scalefactor{.7}

\square(0,250)/@{>}|{\bb}`@{>}|{\bb}`@{>}|{\bb}`@{>}|{\bb}/[{\bf A} `{\bf B} `{\bf C} `{\bf D};P `R `S `Q]

\morphism(250,450)/<=/<0,175>[`;\scriptstyle \phi]

\place (850,500)[\longmapsto]

\square(1300,0)/<-`>`>`<-/<700,500>[{\bf El}(R) `{\bf El}({\bf El}(\phi)) `{\bf C} `{\bf El}(Q);```]

\square(1300,500)/<-`<-`<-`<-/<700,500>[{\bf A} `{\bf El}(P) `{\bf El}(R) `{\bf El}({\bf El}(\phi));```]

\square(2000,0)<700,500>[{\bf El}({\bf El}(\phi)) `{\bf El}(S) `{\bf El}(Q) `{\bf D};```]

\square(2000,500)/>`<-`<-`>/<700,500>[{\bf El}(P) `{\bf B} `{\bf El}({\bf El}(\phi)) `{\bf El}(S);```]

\efig
$$
An object of the middle category, $ {\bf El}({\bf El}(\phi)) $ is precisely a square of elements as in $(*)$.

A profunctor $ R : {\bf A} \tod {\bf C} $ may alternatively be viewed as a discrete cofibration from $ {\bf A} $ to $ {\bf C} $,
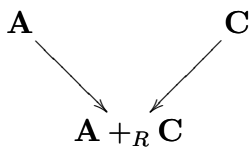
$$
\bfig\scalefactor{.7}

\Vtriangle/`>`>/[{\bf A} `{\bf C} `{\bf A} +_R {\bf C};``]

\efig
$$
i.e. cospan. This gives a colax embedding of $ \CCat $ into the double category of cospans
$$
\CCat \to/^{ (}->/  \Cosp ({\bf Cat})
$$
$$
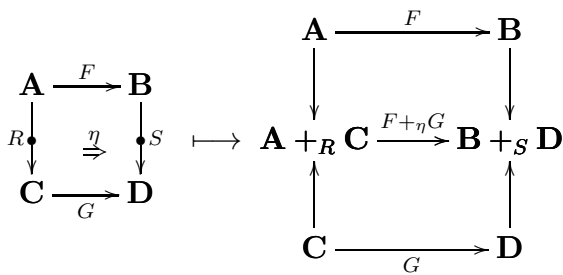
\bfig\scalefactor{.7}

\square(0,250)/>`@{>}|{\bb}`@{>}|{\bb}`>/[{\bf A} `{\bf B} `{\bf C} `{\bf D};F `R `S `G]

\morphism(200,450)/=>/<175,0>[`;\scriptstyle \eta]

\place(850,500)[\longmapsto]

\square(1300,0)/>`<-`<-`>/<900,500>[{\bf A} +_R {\bf C} ` {\bf B} +_S {\bf D} `{\bf C} `{\bf D};F +_\eta G  ```G]

\square(1300,500)/>`>`>`>/<900,500>[{\bf A} `{\bf B} `{\bf A} +_R {\bf C} `{\bf B} +_S {\bf D};F```]

\efig
$$
The double category $ \CCat $ has pushouts. The one-dimensional pushouts are the usual thing. For the two-dimensional ones, note that the category of profunctors with cells like $ \eta $ above can be viewed as $ {\bf Cat}/{\bf 2} $, with $ R $ corresponding to
$$
\bfig\scalefactor{.9}

\morphism(0,0)/<-/<0,300>[{\bf 2}`{\bf A} +_R {\bf C};]

\efig
$$
$ {\bf Cat}/{\bf 2} $ has pushouts and pulling back along the two functors $ {\bf 1} \to {\bf 2} $ preserves them. So $ \CCat $ has pushouts and furthermore the embedding
$$
\CCat \to/^{ (}->/  \Cosp ({\bf Cat})
$$
preserves them. Thus we get a colax-pseudo embedding of intercategories
$$
{\ic Cosp} (\CCat) \to/^{ (}->/  {\ic Cosp} \Cosp ({\bf Cat}) = {\ic Cosp}^2 ({\bf Cat}).
$$
(Remember that for a double category $ {\mathbb A} $, the horizontal arrows of $ {\ic Cosp} ({\mathbb A}) $ are the vertical arrows of $ {\mathbb A} $ and the vertical arrows of $ {\ic Cosp} ({\mathbb A}) $ are cospans of horizontal arrows in $ {\mathbb A} $, which is the only way it can be. This explains why the above embedding is pseudo in the {\em vertical} direction.)

There is also an embedding
$$
{\ic Q} (\CCat) \to/^{ (}->/  {\ic Cosp} (\CCat)
$$
$$
\bfig\scalefactor{.7}

\square(0,250)/@{>}|{\bb}`@{>}|{\bb}`@{>}|{\bb}`@{>}|{\bb}/[{\bf A} `{\bf B} `{\bf C} `{\bf D};P `R `S `Q]

\morphism(220,450)/<=/<0,175>[`;\scriptstyle \phi]

\place(850,500)[\longmapsto]

\square(1300,0)/@{>}|{\bb}`<-`<-`@{>}|{\bb}/<900,500>[{\bf A} +_R {\bf C} ` {\bf B} +_g {\bf D} `{\bf C} `{\bf D};P+_\phi  Q```Q]

\square(1300,500)/@{>}|{\bb}`>`>`/<900,500>[{\bf A} `{\bf B} `{\bf A} +_R {\bf C} `{\bf B} +_g {\bf D};P```]

\morphism(1750,200)/=>/<0,175>[`;]

\morphism(1750,750)/<=/<0,175>[`;]

\efig
$$
where the profunctor in the middle is given by

\noindent $ (P+_\phi Q)(A, B) = P(A, B) $

\noindent $ (P+_\phi Q) (A, D) = (R \otimes Q) (A, D) $

\noindent $ (P+_\phi Q) (C, B) = \emptyset $

\noindent $ (P+_\phi Q) (C, D) = Q (C, D). $

\noindent The action by morphisms is, for the most part, obvious. The only interesting case is for a morphism $ B \to D $ in $ {\bf B} +_S {\bf D} $, i.e. an element of $ S $, $ B \todo{s} D $. Then

$$
(P+_\phi Q) (A, s) : (P+_\phi Q) (A, B) \to (P+_\phi Q)(A, D)
$$
$$ A \todo{p} B \longmapsto \phi (p \otimes s).
$$
A calculation will show that this embedding is pseudo-colax.

Combining these two constructions gives us a colax-colax embedding
$$
{\ic Q} (\CCat) \to/^{ (}->/  {\ic Cosp}^2 ({\bf Cat})
$$
which the reader is invited to write out explicitly.

It looks like there could be an embedding
$$
{\ic Q} (\CCat) \to/^{ (}->/  {\ic Span} {\ic Cosp} ({\bf Cat})
$$
combining the fibration and cofibration constructions and formalizing how they interact. Indeed one could consider
$$
\bfig\scalefactor{.7}

\square(0,250)/@{>}|{\bb}`@{>}|{\bb}`@{>}|{\bb}`@{>}|{\bb}/[{\bf A} `{\bf B} `{\bf C} `{\bf D};P `R `S `Q]

\morphism(250,450)/<=/<0,175>[`;\scriptstyle \phi]

\place (850,500)[\longmapsto]

\square(1300,0)/<-`<-`<-`<-/<1000,500>[{\bf A}+_R {\bf C} `{\bf El} P+_{{\bf El}(\phi)}{\bf El}Q `{\bf C} `{\bf El} Q;```]

\square(1300,500)/<-`>`>`<-/<1000,500>[{\bf A} `{\bf El} P `{\bf A}+_R {\bf C} `{\bf El} P+_{{\bf El}(\phi)}{\bf El}Q;```]

\square(2300,0)/>`<-`<-`>/<1000,500>[{\bf El} P+_{{\bf El}(\phi)}{\bf El}Q `{\bf B}+_S{\bf D} `{\bf El}Q `{\bf D};```]

\square(2300,500)/>`>`>`>/<1000,500>[{\bf El}P `{\bf B} `{\bf El} P+_{{\bf El}(\phi)}{\bf El}Q `{\bf B}+_S{\bf D};```]

\efig
$$
This looks nice but it would be lax-colax, which is the combination that doesn't work, so something must go wrong. And indeed, this assignment is neither lax nor colax in the vertical direction.

The situation is this. $ \Cosp ({\bf Cat}) $ has pullbacks and $ \CCat \to/^{ (}->/ \Cosp ({\bf Cat}) $ preserves them so we get a pseudo-colax embedding
$$
{\ic Span} (\CCat) \to/^{ (}->/  {\ic Span} {\ic Cosp} ({\bf Cat})
$$
which can't be composed with the lax-lax
$$
{\ic Q} (\CCat) \to/^{ (}->/  {\ic Span} (\CCat)
$$
which is what we were trying to do above.

$ \SSpan ({\bf Cat}) $ also has pushouts so we can apply the $ {\ic Cosp} $ construction to the lax embedding
$$
\CCat \to/^{ (}->/  \SSpan ({\bf Cat})
$$
to get a lax-lax embedding
$$
{\ic Cosp} (\CCat) \to/^{ (}->/  {\ic Cosp} (\SSpan({\bf Cat})) = {\ic Span} {\ic Cosp} ({\bf Cat}).
$$
Thus we get a diagram of embeddings, though neither of the composites has any of the admissible laxity-colaxity structures:
$$
\bfig\scalefactor{.7}

\square/^{ (}->`^{ (}->`^{ (}->`^{ (}->/<1300,500>[{\ic Q} (\CCat) `{\ic Span} (\CCat) `{\ic Cosp} (\CCat) `{\ic Span}{\ic Cosp}({\bf Cat});lax-lax `pseudo-colax `pseudo-colax `lax-lax]

\efig
$$
We can still compose them as functions and we find that the diagram doesn't commute! Going around the top gives a span of cospans for which the middle category is $ {\bf El} P+_{{\bf El}(\phi)} {\bf El} Q $ which has two kinds of objects $ A \todo{p} B $ and $ C \todo{q} D $ and obvious morphisms. On the other hand, going around the bottom has as middle category $ {\bf El}(P+_\phi Q) $ which has three kinds of objects $ A \todo{p} B $, $ C \todo{q} D $, and $[ A \todo{r} C \todo{q} D] $ with ``obvious'' morphisms. The inclusion $ {\bf El} P +_{{\bf El}\phi} {\bf El} Q \to/^{ (}->/  {\bf El}(P+_\phi Q) $ is a basic cell in the triple category $ {\ic I}{\ic Cat} $ of intercategories.

\subsection{Generalized spans and cospans}\label{spans.6} 

\ 

All of our examples so far, except for those gotten from duoidal categories, have at least one of the interchangers $ \tau $, $ \delta $, $ \mu $, $ \chi $ being an isomorphism, and the duoidal category ones are quite special in that their horizontal and vertical cells are identities. One is left wondering if there might not be some sort of tension between the actual triple structure and the interchangers forcing one or the other to be trivial. Of course we could take the product of two of our examples, the duoidal category ones of Section \ref{duoidal.1} and the spans in a double category of Section \ref{spans.1} say, to get examples where nothing is degenerate, but that's still not very satisfying.

In this section we introduce a class of intercategories, parametrized by quintets in $ \Cat $, which generalize the spans of cospans of Section \ref{spans.4}. Our main interest here is in nice examples of intercategories in which the interchangers are non trivial.

Consider a quintet in $ \Cat $
$$
\bfig\scalefactor{.9}

\square<400,400>[{\bf A} `{\bf B} `{\bf C} `{\bf D};F ` H `K `G]

\morphism(270,250)/=>/<-85,-85>[`;\phi]

\efig
$$
in which $ {\bf B} $ and $ {\bf D} $ have pullbacks, and ${\bf C} $ and $ {\bf D} $ have pushouts. (No preservation properties are assumed on $ G $ or $ K $.)

We construct an intercategory $ {\ic SC}_\phi $ as follows. A basic cell consists of 9 objects and 12 morphisms as in
$$
\bfig\scalefactor{.7}
\place(0,0)[A_{21}]
\node a2(400,0)[FA_{21}]
\node a3(900,0)[B_2]
\node a4(1400,0)[FA_{22}]
\place(1800,0)[A_{22}]
 
\node b1(0,300)[HA_{21}]
\node b3(900,300)[KB_2]
\node b5(1800,300)[HA_{22}]

\node c1(0,800)[C_1]
\node c2(400,800)[GC_1]
\node c3(900,800)[D]
\node c4(1400,800)[GC_2]
\node c5(1800,800)[C_2]

\node d1(0,1300)[HA_{11}]
\node d3(900,1300)[KB_1]
\node d5(1800,1300)[HA_{12}]

\place(0,1600)[A_{11}]
\node e2(400,1600)[FA_{11}]
\node e3(900,1600)[B_1]
\node e4(1400,1600)[FA_{12}]
\place(1800,1600)[A_{12}]

\arrow[a3`a2;]
\arrow[a3`a4;]
\arrow[b1`c1;]
\arrow[b3`c3;]
\arrow[b5`c5;]
\arrow[c3`c2;]
\arrow[c3`c4;]
\arrow[d1`c1;]
\arrow[d3`c3;]
\arrow[d5`c5;]
\arrow[e3`e2;]
\arrow[e3`e4;]

\efig
$$
satisfying four commutativities $(i, j = 1, 2)$
$$
\bfig\scalefactor{.8}

\node a(0,0)[D]
\node b(0,500)[KB_i]
\node c(500,500)[KFA_{ij}]
\node d(1100,300)[GHA_{ij}]
\node e(1100,0)[GC_j]

\arrow[b`a;]
\arrow[b`c;]
\arrow[c`d;\phi A_{ij}]
\arrow[d`e;]
\arrow[a`e;]

\efig
$$
From this one may infer that the objects of $ {\ic SC}_\phi $ are the objects of $ {\bf A} $, and its horizontal arrows are $F$-spans in $ {\bf B} $, i.e. spans whose domain and codomain are $ F $ of something in $ {\bf A} $. Similarly, the vertical arrows are $H$-cospans in $ {\bf C} $. The transversal arrows are natural families of morphisms, i.e. $ a_{ij} : A_{ij} \to A'_{ij} $, $ b_i : B_i \to B'_i $, $ c_j : C_j \to C'_j $, $ d: D \to D' $ commuting with the structural morphisms. When $ F, G, H, K, \phi $ are all identities, this is just the ${\ic SpanCosp} {\bf A} $ of \ref{spans.4}.

Horizontal composition is span composition, i.e. given by pulling back:
$$
\bfig\scalefactor{.7}
\place(-100,0)[A_{21}]
\node a2(300,0)[FA_{21}]
\node a3(1100,0)[B_2 \times_{FA_{22}} B'_2]
\node a4(1900,0)[FA_{23}]
\place(2300,0)[A_{23}]
 
\node b1(-100,300)[HA_{21}]
\node b3(1100,300)[K(B_2 \times_{FA_{22}} B'_2)]
\node b5(2300,300)[HA_{23}]

\node c1(-100,800)[C_1]
\node c2(300,800)[GC_1]
\node c3(1100,800)[D \times_{GC_2} D']
\node c4(1900,800)[GC_3]
\node c5(2300,800)[C_3]

\node d1(-100,1300)[HA_{11}]
\node d3(1100,1300)[K(B_1 \times_{FA_{12}} B'_1)]
\node d5(2300,1300)[HA_{13}]

\place(-100,1600)[A_{11}]
\node e2(300,1600)[FA_{11}]
\node e3(1100,1600)[B_1 \times_{FA_{12}} B'_1]
\node e4(1900,1600)[FA_{13}]
\place(2300,1600)[A_{13}]

\arrow[a3`a2;]
\arrow[a3`a4;]
\arrow[b1`c1;]
\arrow[b3`c3;]
\arrow[b5`c5;]
\arrow[c3`c2;]
\arrow[c3`c4;]
\arrow[d1`c1;]
\arrow[d3`c3;]
\arrow[d5`c5;]
\arrow[e3`e2;]
\arrow[e3`e4;]

\efig
$$
The vertical arrows in the middle are pullback induced by
$$
\bfig\scalefactor{.7}

\node a1(0,400)[K(B_i \times_{FA_{i2}} B'_i)]
\node a2(450,800)[KB_i]
\node a3(1600,800)[D]
\node a4(2100,400)[GC_2]
\node a5(1600,0)[D']
\node a6(450,0)[KB'_i]
\node a7(850,400)[KFA_{i2}]
\node a8(1500,400)[GHA_{i2}]

\arrow[a1`a2;]
\arrow[a2`a3;]
\arrow[a3`a4;]
\arrow[a5`a4;]
\arrow[a6`a5;]
\arrow[a1`a6;]
\arrow[a2`a7;]
\arrow[a6`a7;]
\arrow[a7`a8;]
\arrow[a8`a4;]

\efig
$$
and are the composites
$$
K(B_i \times_{FA_{i2}} B'_i) \to KB_i \times_{KFA_{i2}} KB'_i \to KB_i \times_{GHA_{i2}} KB'_i \to D\times_{GC_2} D'.
$$

Vertical composition is dual.

The horizontal identities $ \id_A $ and $ \id_C $ are shown in the diagram
$$
\bfig\scalefactor{.7}
\place(0,0)[\ov{A}]
\node a2(300,0)[F\ov{A}]
\node a3(800,0)[F\ov{A}]
\node a4(1300,0)[F\ov{A}]
\place(1600,0)[\ov{A}]
 
\node b1(0,300)[H\ov{A}]
\node b3(800,300)[KF\ov{A}]
\node b5(1600,300)[H\ov{A}]

\node c1(0,800)[C]
\node c2(300,800)[GC]
\node c3(800,800)[GC]
\node c4(1300,800)[GC]
\node c5(1600,800)[C]

\node d1(0,1300)[HA]
\node d3(800,1300)[KFA]
\node d5(1600,1300)[HA]

\place(0,1600)[A]
\node e2(300,1600)[FA]
\node e3(800,1600)[FA]
\node e4(1300,1600)[FA]
\place(1600,1600)[A]

\arrow|b|[a3`a2;1]
\arrow|b|[a3`a4;1]
\arrow[b1`c1;]
\arrow[b3`c3;]
\arrow[b5`c5;]
\arrow[c3`c2;1]
\arrow[c3`c4;1]
\arrow[d1`c1;]
\arrow[d3`c3;]
\arrow[d5`c5;]
\arrow[e3`e2;1]
\arrow[e3`e4;1]

\efig
$$
where the vertical arrows in the middle are
$$
KFA \to^{\phi A} GHA \to GC \quad \mbox{and} \quad KF\ov{A} \to^{\phi\ov{A}} GH\ov{A} \to GC.
$$
The vertical identities, $ \Id_A $ and $ \Id_B $, which are dual, are displayed in
$$
\bfig\scalefactor{.7}
\place(-75,0)[A]
\node a2(300,0)[FA]
\node a3(800,0)[B]
\node a4(1300,0)[FA']
\place(1675,0)[A']
 
\node b1(-75,300)[HA]
\node b3(800,300)[KB]
\node b5(1675,300)[HA']

\node c1(-75,800)[HA]
\node c2(300,800)[GHA]
\node c3(800,800)[KB]
\node c4(1300,800)[GHA']
\node c5(1675,800)[HA']

\node d1(-75,1300)[HA]
\node d3(800,1300)[KB]
\node d5(1675,1300)[HA']

\place(-75,1600)[A]
\node e2(300,1600)[FA]
\node e3(800,1600)[B]
\node e4(1300,1600)[FA']
\place(1675,1600)[A']

\arrow[a3`a2;]
\arrow[a3`a4;]
\arrow|l|[b1`c1;1]
\arrow[b3`c3;1]
\arrow[b5`c5;]
\arrow[c3`c2;]
\arrow[c3`c4;]
\arrow[d1`c1;1]
\arrow|r|[d3`c3;1]
\arrow[d5`c5;]
\arrow[e3`e2;]
\arrow[e3`e4;]

\efig
$$
where the horizontal arrows in the middle are
$$
KB \to KFA \to^{\phi A} GHA \quad \mbox{and} \quad KB \to KFA' \to^{\phi A'} GHA'.
$$

We are interested in the interchangers $ \tau $, $ \delta $, $ \mu $, and $ \chi $, which are transversally special cubes by definition, so that, of the nine arrows making up such a cube, only the middle one is not an identity. Thus everything of interest is concentrated there and we abuse notation by denoting a span or cospan and even a span of cospans by its middle object, the arrows being understood. So we write $ \id_A = FA $, $ \id_C = GC $, $ \Id_A = HA$, and $ \Id_B = KB $. Then $ \Id_{\id_A} = \Id_{FA} = KFA $ and $ \id_{\Id_A} = \id_{HA} = GHA $, and $ \tau A : \Id_{\id_A} \to \id_{\Id_A} $ is exactly
$$
\hspace{5cm}\phi A : KFA \to GHA \hspace{5cm} (1)
$$

For a composable pair of $F$-spans $ B $ and $ B' $ we have $ B | B' = B \times_{FA'} B' $ so $ \Id_{B | B'} = K (B \times_{FA'} B') $ whereas $ \Id_B | \Id_{B'} $ is $ KB \times_{GHA'} KB' $ and $ \delta(B, B') : \Id_{B|B'} \to \Id_B | \Id_{B'} $ is given by
$$
\hspace{1cm} K(B \times_{FA'} B') \to^{\lambda} KB\times_{KFA'} K B' \to^{1 \times_{\phi A'} 1} KB \times_{GHA'} KB' \hspace{1cm} (2)
$$

The $ \mu $ is dual to this. $ \mu (C, C') : \dfrac{\id_C}{\id_{C'}} \to \id_{\frac{C}{C'}} $ is the composite
$$
\hspace{2cm} GC +_{KF\ov{A}} G C' \to^{1 +_{\phi \ov{A}} 1} GC +_{GH\ov{A}}  G C' \to^{\gamma} G(C +_{H\ov{A}} C') \hspace{2cm} (3)
$$

The formula for $\chi $ is similar but more complicated notationally. Consider four basic cells $ (i, j = 1, 2) $
$$
\bfig\scalefactor{.9}

\square/@{>}|{\circ}`@{>}|{\bb}`@{>}|{\bb}`@{>}|{\circ}/<700,400>[A_{ij}  `A_{i, j+1}  `A_{i+1, j}  `A_{i + 1, j+1};B_{ij} `C_{ij} `C_{i, j+1} `B_{i+1, j}]

\place(350,200)[\scriptstyle D_{ij}]

\efig
$$
with structure morphisms as above. Then $ (D_{11} \circ D_{12}) \bullet (D_{21} \circ D_{22}) $ is the pushout $ P $ of
$$
\bfig\scalefactor{.7}

\node a(0,350)[K(B_{21} \times_{FA_{22}} B_{22})\ \ \ \ \ \ \ \ ]
\node b(800,700)[\ \ \ \ \ \ \ \ D_{11} \times_{GC_{12}} D_{12}]
\node c(800,0)[\ \ \ \ \ \ \ \ D_{21} \times_{GC_{22}} D_{22}]

\arrow[a`b;]
\arrow[a`c;]

\efig
$$
whereas $ (D_{11} \bullet D_{21}) \circ (D_{12} \bullet D_{22}) $ is the pullback $ Q $ of 
$$
\bfig\scalefactor{.7}

\node a(0,0)[D_{12} +_{KB_{22}} D_{21}\ \ \ \ \ \ \ \ ]
\node b(0,700)[D_{11} +_{KB_{21}} D_{21}\ \ \ \ \ \ \ \ ]
\node c(800,350)[\ \ \ \ \ \ \ \ \ \ G(C_{12} +_{HA_{22}} C_{22})]

\arrow[a`c;]
\arrow[b`c;]

\efig
$$
A morphism from a pushout to a pullback is given by four morphisms satisfying four equations. The ones giving $ \chi $ are:
$$
D_{i1} \times_{GC_{i2}} D_{i2} \to^{proj_j} D_{ij} \to^{inj_i} D_{1j} +_{KB_{2j}} D_{2j}
$$
It can be difficult to determine when such a $ \chi $ is an isomorphism.

The input data for an instance of $ \chi $ is a $ 5 \times 5 $ array of objects connected by a number (40) of arrows, but the pullbacks and pushouts needed only involve the $ 3 \times 3 $ part at its centre (with its 12 arrows), which gives the following cospan of spans:
$$
\bfig\scalefactor{.7}
\node a1(0,0)[D_{21}]
\node a2(600,0)[GC_{22}]
\node a3(1200,0)[D_{22}]

\node b1(0,400)[KB_{21}]
\node b2(600,400)[KFA_{22}]
\node b3(1200,400)[KB_{22}\rlap{\phantom{the resulting diagram} (4)}]

\node c1(0,800)[D_{11}]
\node c2(600,800)[GC_{12}]
\node c3(1200,800)[D_{12}]

\arrow[a1`a2;]
\arrow[a3`a2;]

\arrow[b1`a1;]
\arrow[b2`a2;]
\arrow[b3`a3;]

\arrow[b1`b2;]
\arrow[b3`b2;]

\arrow[b1`c1;]
\arrow[b2`c2;]
\arrow[b3`c3;]

\arrow[c1`c2;]
\arrow[c3`c2;]

\efig
$$
We can take the pullback of each row and then the pushout of the resulting diagram or the other way around and we get the canonical lim-colim comparison morphism $ \theta $ from the pushout $ P' $ of 
$$
\bfig\scalefactor{.7}

\node a(0,350)[KB_{21} \times_{KFA_{22}} KB_{22}\ \ \ \ \ \ \ \ ]
\node b(800,700)[\ \ \ \ \ \ \ \ D_{11} \times_{GC_{12}} D_{12}]
\node c(800,0)[\ \ \ \ \ \ \ D_{21} \times_{GC_{22}} D_{22}]

\arrow[a`b;]
\arrow[a`c;]

\efig
$$
to the pullback $ Q' $ of
$$
\bfig\scalefactor{.7}

\node a(0,0)[D_{12} +_{KB_{22}} D_{22}\ \ \ \ \ \ \ ]
\node b(0,700)[D_{11} +_{KB_{21}} D_{21}\ \ \ \ \ \ \ ]
\node c(800,350)[\ \ \ \ \ \ \ GC_{12} +_{KFA_{22}} GC_{22}]

\arrow[a`c;]
\arrow[b`c;]

\efig
$$
Let us say that {\em $G$-pullbacks commute with $K$-pushouts} if this $ \theta $ is an isomorphism for all cospans of spans of the above form, i.e. commutative diagrams in $ {\bf D} $ where the horizontal arrows in the middle are $ K{b_1} $ and $ K{b_2} $  and the vertical ones in the middle are $ Gc_1 \cdot \phi A_{22} $ and $ Gc_2 \cdot \phi A_{22} $, not necessarily coming from cells of $ {\ic SC}_\phi $. For example, if $ F, G, H, K, \phi $ are all identities, this says that pullbacks commute with pushouts, a condition that almost never holds.

Now the $ P' $ and $ Q' $ are almost the same as $ P $ and $ Q $ but the pushout is over the wrong object, and similarly for the pullback. We have the canonical comparison
$$
\kappa : K(B_{21} \times_{FA_{22}} B_{22}) \to KB_{21} \times_{KFA_{22}} KB_{22}
$$
which induces an epimorphism:
$$
1 +_\kappa 1 : P \to P'
$$
Similarly we have
$$
GC_{12} +_{KFA_{22}} GC_{22} \to^{1 +_{\phi A_{22}} 1}  GC_{12} +_{GHA_{22}} GC_{22} \to^\gamma G(C_{12} +_{HA_{22}} C_{22})
$$
inducing a monomorphism
$$
Q' \to^{1 \times_\psi 1} Q'' \to^{1 \times_\gamma 1} Q
$$
where $ \psi = 1 +_{\phi A_{22}} 1 $ and $ Q'' $ is the pullback over the middle object. Then
$$
\hspace{3cm} \chi = (P \to^{1 +_\kappa 1} P' \to^\theta Q' \to^{1 \times_\psi 1} Q'' \to^{1 \times_\gamma 1} Q) \hspace{3cm} (5)
$$

\begin{theorem} In $ {\ic SC}_\phi $ we have:
\vspace{.2cm}

\noindent (1) $ \tau $ is an isomorphism if and only if $ \phi $ is;

\vspace{.2cm}

\noindent (2) $ \delta $ is an isomorphism if and only if $ K$ preserves $F$-pullbacks and $ \phi $ is a monomorphism;

\vspace{.2cm}

\noindent (3) $ \mu $ is an isomorphism if and only if $ G $ preserves $H$-pushouts and $ \phi $ is an epimorphism;

\vspace{.2cm}

\noindent (4) $ \chi $ is an isomorphism if and only if

\noindent\hspace{.4cm} (a) $ \phi $ is an isomorphism,

\noindent\hspace{.4cm} (b) $ K $ transforms $F$-pullbacks into quasi-pullbacks,

\noindent\hspace{.4cm} (c) $ G $ transforms $H$-pushouts into quasi-pushouts,

\noindent\hspace{.4cm} (d) $G$-pullbacks commute with $K$-pushouts.

\end{theorem}

\begin{proof}

\noindent (1) $ \tau A = \phi A $ so (1) is obvious.

\vspace{.2cm}

\noindent (2) $ \delta $ is given in formula (2) above as $ (1\times_{\phi A'} 1) \lambda $, and $ 1 \times_{\phi A'} 1 $ is a monomorphism, so $ \delta $ is an isomorphism if and only if both $ \lambda $ and $ 1 \times_{\phi A'} 1 $ are. That $ \lambda $ is an isomorphism is the definition of $ K $ preserving the $ F $-pullback in question, and every $ F $-pullback arises in this way. Indeed, the $ F $-pullback
$$
\bfig\scalefactor{.7}

\Ctriangle/<-``>/[B`B\times_{FA'} B' `B';``]

\Dtriangle(500,0)/`>`<-/[B `FA' `B';`b`b']

\efig
$$
will come up if $ \delta $ is applied to the $ F $-spans
$$
[b] = FA'  \to/<-/^b B \to^b FA'  \ \ \mbox{and} \ \ [b'] = FA' \to/<-/^{b'}  B' \to^{b'} FA'
$$
Consider
$$
\bfig\scalefactor{.8}

\node a1(0,0)[KB']
\node a2(0,600)[KB]
\node a3(700,600)[KB]
\node a4(1100,300)[GHA']
\node a5(700,0)[KB']
\node a6(300,300)[KFA']

\arrow[a1`a6;]
\arrow[a2`a6;]
\arrow[a2`a3;1_{KB}]
\arrow[a3`a4;]
\arrow[a6`a4;\phi A']
\arrow|b|[a1`a5;1_{KB'}]
\arrow[a5`a4;]

\efig
$$
$ 1 \times_{\phi A'} 1$ is the morphism this induces from the pullback of the three objects on the left to the pullback of the other three. So, if $ \phi A' $ is monic then $ 1 \times_{\phi A'} 1 $ is an isomorphism. Conversely, we can take $ B \to FA' $ and $ B' \to FA' $ to be the identity $ 1_{FA'} $, and then $ 1\times_{\phi A'} 1 $ is the diagonal for the kernel pair of $ \phi A' $ and that's invertible if and only if $ \phi A' $ is monic.

\vspace{.2cm}

\noindent (3) This is the dual to (2).

\vspace{.2cm}

\noindent (4) $ \chi $ is given by (5) above in which $ 1 +_\kappa 1 $ is an epimorphism and $1 \times_\psi 1 $ and $ 1 \times_\gamma 1 $ are monomorphisms, so $ \chi $ is invertible if and only each of these morphisms as well as $ \theta $ are.

First of all, direct calculation shows that $ \chi $ applied to ($\Id_{id_A}$, $\id_{\Id_A}$; $\id_{\Id_A}$, $\Id_{\id_A}$) is $ \phi A : KFA \to GHA $, so if $ \chi $ is an isomorphism, then $ \phi $ is necessarily one too.

Assume now that $ \chi $ is an isomorphism. Then so is $ \phi $ which gives (a).

To say that $ K $ transforms $F$-pullbacks into quasi-pullbacks means that the canonical morphism
$$
\kappa : K(B \times_{FA} B') \to KB \times_{KFA} KB'
$$
is an epimorphism, for every diagram
$$
\bfig\scalefactor{.9}

\node a(0,0)[B']
\node b(0,400)[B]
\node c(300,200)[FA]

\arrow|b|[a`c;b']
\arrow[b`c;b]

\efig
$$
If we evaluate $ \chi $ at ($\Id_{[b]}$, $\Id_{[b']} $; $\Id_{[b]} $, $\Id_{[b']} $), then the legs of the diagram whose pushout is $ P' $ are both equal to $ 1 \times_{\phi A} 1 $, which is an isomorphism. Thus $ 1 \times_\kappa 1 : P \to P' $ is the codiagonal for the cokernel pair of $ \kappa $, and as $ 1 \times_\kappa 1 $ is invertible, $ \kappa $ is an epimorphism. This proves (b).

(c) is dual to (b).

(d) means that $ \theta $ is an isomorphism for all diagrams (4) coming from cells of $ {\ic SC}_\phi $, but in fact all such diagrams come from cells of $ {\ic SC}_\phi $. Indeed, one can take the cells
$$
\bfig\scalefactor{.8}

\square/@{>}|{\cc}`@{>}|{\bb}`@{>}|{\bb}`@{>}|{\cc}/[A_{22} `A_{22} `A_{22} `A_{22};{[}b_i{]} `{[}c_i{]} `{[}c_i{]} `{[b}_i{]}]

\place(250,250)[\scriptstyle D_{ij}]

\efig
$$

Conversely, assume that (a)-(d) are satisfied. Then from (a) we know that $ \phi $ is an isomorphism so $ \psi = 1+_{\phi_{A_{22}}} 1  $ is also and thus $ 1 \times_\psi 1 $ is too. From (b) we get that $ \kappa $ is an epimorphism so $ 1 +_\kappa 1 $ is an isomorphism. Similarly (c) shows that $ 1 \times_\gamma 1 $ is an isomorphism. (d) is just a special case of $ \theta $ being an isomorphism.

\end{proof}

This theorem gives many examples of intercategories in which none of the interchangers are invertible. Any quintet in which the $ \phi $ is neither monic nor epic gives one, for example.

There is a similar class of examples generalizing the double spans of Section \ref{spans.2}. It starts again with a quintet
$$
\bfig\scalefactor{.8}

\square<400,400>[{\bf A} `{\bf B} `{\bf C} `{\bf D};F ` H `K `G]

\morphism(270,250)/=>/<-95,-95>[`;\phi]

\efig
$$
but now $ {\bf B}, {\bf C}, {\bf D} $ are required to have pullbacks. The details are more complicated and will appear elsewhere.

\section{The intercategory ${\ic Set} $} \label{Set}

\subsection{Intermonads}\label{Set.1}

\ 

As is well known, a small category is a monad in the bicategory of spans, or better, a vertical monad in the double category $ {\mathbb S}{\rm et} $ of sets, functions and spans. Better because it is here that functors appear naturally. So a small category corresponds to a lax functor $ {\mathbbm 1} \to {\mathbb S}{\rm et} $.

Examining the notion of small double category, which has two kinds of morphisms, cells and various domains and codomains, we see a span of spans. We wish to code up the compositions and identities as a sort of double monad. This will live in the intercategory $ {\ic Span} ({\mathbb S}{\rm pan} ({\bf Set})) = {\ic Span} ({\mathbb S}{\rm et})= {\ic Span}^2 ({\bf Set}) $ which we call {\em the intercategory of sets} and denote $ {\ic Set} $. A small double category will then turn out to be a lax-lax morphism of intercategories $ {\mathbbm 1} \to {\ic Set}$. 

Let us examine the structure of a lax-lax functor from $ {\mathbbm 1} $ to an arbitrary intercategory $ {\ic A} $. The unique basic cell of $ {\mathbbm 1} $ gives
$$
\bfig\scalefactor{.6}

\square[*`*`*`*;\id`\Id`\Id`\id]

\place(250,250)[\scriptstyle \id_\Id]

\place(1000,250)[\to/|->/]

\square(1500,0)[A`A`A`A;t`T`T`t]

\place(1750,250)[\scriptstyle D]

\efig
$$
$ t $ is a horizontal monad whose structure is given by special horizontal cells
$$
\bfig\scalefactor{.6}

\morphism(0,400)|a|/>/<600,0>[A`A;\id]

\morphism(400,0)|b|/>/<600,0>[A`A;t]

\morphism(0,400)/=/<400,-400>[A`A;]

\morphism(600,400)/=/<400,-400>[A`A;]

\place(500,200)[\scriptstyle u]

\morphism(1500,400)|a|<600,0>[A`A;t]

\morphism(1900,0)|b|<1200,0>[A`A;t]

\morphism(2100,400)|a|/>/<600,0>[A`A;t]

\morphism(1500,400)|b|/=/<400,-400>[A`A;]

\morphism(2700,400)/=/<400,-400>[A`A;]

\place(2300,200)[\scriptstyle m]

\efig
$$
composed in the transversal direction. $ T $ is a vertical monad whose structure is given by vertical cells, $ U $ and $ M $, also composed in the transversal direction.

The main structure is on $ D $. It is a horizontal and vertical monad whose structural morphisms are cubes
$$
\bfig\scalefactor{.8}

\square(0,300)/>`>``/[A ` A `A `;\id_A `T ` `]

\square(300,0)[A ` A`A `A; t ` T `T`t]

\morphism(0,800)|b|/=/<300,-300>[A` A;]

\morphism(500,800)/=/<300,-300>[A ` A;]

\morphism(0,300)|b|/=/<300,-300>[A`A;]

\place(550,250)[\scriptstyle D]

\morphism(130,450)/=/<80,-80>[`;]

\place(400,650)[\scriptstyle u]

\place(500,-200)[u : \id_T \to D]

\square(1500,300)/>`>``/[A ` A `A `;t `\Id_A ` `]

\square(1800,0)[A ` A `A `A;t ` T `T `t]

\morphism(1500,800)|b|/=/<300,-300>[A` A;]

\morphism(2000,800)/=/<300,-300>[A` A;]

\morphism(1500,300)|b|/=/<300,-300>[A`A;]

\place(2050,250)[\scriptstyle D]

\place(1650,400)[\scriptstyle U]

\morphism(1900,700)/=/<80,-80>[`;]

\place(2000,-200)[U : \Id_t \to D]

\efig
$$

$$
\bfig\scalefactor{.8}

\square(0,300)/>`>``/[A ` A `A `;t ` T ` `]

\morphism(500,800)/>/[A`A;t]

\square(300,0)|arrb|<1000,500>[A ` A `A `A; t `T `T`t]

\morphism(0,800)/=/<300,-300>[A` A;]

\morphism(1000,800)/=/<300,-300>[A ` A;]

\morphism(0,300)/=/<300,-300>[A`A;]

\place(800,250)[\scriptstyle D]

\morphism(150,400)/=/<50,-50>[`;]

\place(600,650)[\scriptstyle m]

\place(800,-250)[m: D|D \to D]

\square(2000,300)/>```/[A ` A ``;t `  ` `]

\square(2300,0)[A ` A`A`A; t` T `T `t]

\morphism(2000,800)|l|/>/<0,-250>[`A;T]

\morphism(2000,550)|l|/>/<0,-190>[`;T]

\morphism(2000,800)|b|/=/<300,-300>[A` A;]

\morphism(2500,800)/=/<300,-300>[A ` A;]

\morphism(2000,300)|b|/=/<300,-300>[A`A;]

\place(2550,250)[\scriptstyle D]

\place(2150,400)[\scriptstyle M]

\morphism(2370,700)/=/<50,-50>[`;]

\place(2350,-250)[M: \dfrac{D}{D} \to D]

\efig
$$
These must satisfy the following conditions.

\noindent (1) (Horizontal monad)
$$
\bfig\scalefactor{.7}

\qtriangle<700,500>[\id_T | D ` D | D `D;u | D`\lambda `m]

\ptriangle(700,0)/<-`>`>/<700,500>[D | D ` D | \id_T `D; D | u ``\rho]

\square(2500,0)<800,500>[D | D | D `D | D `D | D `D;D | m `m | D `m `m]

\efig
$$

\noindent (2) (Vertical monad)
$$
\bfig\scalefactor{.7}

\qtriangle<700,500>[\frac{\Id_t}{D}  `\frac{D}{D} `D;\frac{U}{D}`\lambda`M]

\ptriangle(700,0)/<-`>`>/<700,500>[\frac{D}{D}`\frac{D}{\Id_t} `D;\frac{D}{U}``\rho]

\square(2500,0)<700,700>[\dfrac{D}{\dfrac{D}{D}}`\dfrac{D}{D}`\dfrac{D}{D} `D; \frac{D}{M}`\frac{M}{D}`M`M]

\efig
$$

\noindent (3) (Horizontal/vertical compatibility)

$$
\bfig

\morphism(150,700)[\dfrac{D | D}{D | D}`\left.\dfrac{D}{D}\right|\dfrac{D}{D};\chi]

\morphism(150,700)<-150,-400>[\dfrac{D | D}{D | D}`\dfrac{D}{D};\frac{m}{m}]

\morphism(650,700)<150,-400>[\left.\dfrac{D}{D}\right|\dfrac{D}{D}`D | D;M | M]

\morphism(0,300)|b|<400,-300>[\dfrac{D}{D}`D;M]

\morphism(800,300)|b|<-400,-300>[D | D`D;m]

\morphism(1650,700)[\dfrac{\id_T}{\id_T}`\id_{\frac{T}{T}};\mu]

\morphism(1650,700)<-150,-400>[\dfrac{\id_T}{\id_T}`\dfrac{D}{D};\frac{u}{u}]

\morphism(2150,700)<150,-400>[\id_{\frac{T}{T}}`\id_T;\id_M]

\morphism(1500,300)|b|<400,-300>[\dfrac{D}{D}`D;M]

\morphism(2300,300)|b|<-400,-300>[\id_T`D;u]

\efig
$$

$$
\bfig

\morphism(150,700)[\Id_{t | t}`\Id_t | \Id_t;\delta]

\morphism(150,700)<-150,-400>[\Id_{t | t}`\Id_t;\Id_m]

\morphism(650,700)<150,-400>[\Id_t | \Id_t`D | D;U | U]

\morphism(0,300)|b|<400,-300>[\Id_t`D;U]

\morphism(800,300)|b|<-400,-300>[D | D`D;m]

\morphism(1650,700)[\Id_{\id_A}`\id_{\Id_A};\tau]

\morphism(1650,700)<-150,-400>[\Id_{\id_A}`\Id_t;\Id_u]

\morphism(2150,700)<150,-400>[\id_{\Id_A}`\id_T;\id_U]

\morphism(1500,300)|b|<400,-300>[\Id_t`D;U]

\morphism(2300,300)|b|<-400,-300>[\id_T`D;u]

\efig
$$

\begin{definition} We call a quintuple $ (D, u, m, U, M) $ as above satisfying conditions (1)-(3) an {\em intermonad} in $ {\ic A} $. \end{definition}

Not surprisingly, an intermonad in $ {\ic Set} $ is a small (strict) double category. Conditions (3) express the interchange law. More generally, for a category $ {\bf A} $ with pullbacks, an intermonad in $ {\ic Span^2} {\bf A} $ is a double category object in $ {\bf A} $.

If $ {\cal A} $ is a bicategory, then an intermonad in the intercategory $ {\ic Q}({\cal A}) $ of quintets, as introduced in Section \ref{Verity}, reduces to a pair of monads with a distributive law between them.

Thus we see that distributive laws and double categories are both special cases of the same construction! This formalizes the idea that a double category is ``just two categories with the same objects plus some way of relating the two kinds of arrows''.

If $ {\bf V} $ is a duoidal category, considered as an intercategory as in Section \ref{duoidal}, then an intermonad in $ {\bf V} $ is what Aguiar and Mahajan \cite{AM} call a double monoid.

\subsection{Hom functors for intercategories}\label{Set.2}

\ 

We end with an example of morphism of intercategories reinforcing the idea that $ {\ic Span^2} ({\bf Set}) $ is really the intercategory of sets. Let $ {\ic A} $ be an intercategory and $ X $ a fixed object of $ {\ic A} $. Define the hom functor $ \H : {\ic A} \to {\ic Set} $ as follows.

\noindent (1) $ \H (A) $ is the set of transversal arrows $ X \to A $.

\noindent (2) For a transversal arrow $ f : A \to A' $, $ \H (f) : \H (A) \to \H (A') $ is defined by composing with $ f $ as usual. This is a function, so a transversal arrow of $ {\ic Set} $, and is strictly functorial in $ f $.

\noindent (3) For a horizontal arrow $ h : A \tod B $, $ \H (h) $ is the span
$$
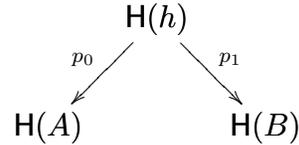
\bfig\scalefactor{.7}

\Atriangle/>`>`/[\H (h) `\H (A) `\H (B);p_0`p_1`]

\efig
$$
$\H (h) $ is the set of all horizontal cells
$$
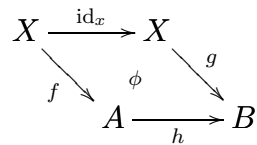
\bfig\scalefactor{.7}

\morphism(0,400)|a|/>/<600,0>[X`X;\id_x]

\morphism(400,0)|b|/>/<600,0>[A`B;h]

\morphism(0,400)|l|/>/<400,-400>[X`A;f]

\morphism(600,400)|r|/>/<400,-400>[X`B;g]

\place(500,200)[\scriptstyle \phi]

\efig
$$
with $ p_0 (\phi) = f $ and $ p_1 (\phi) = g $. We consider $ \H (h) $ as a horizontal arrow in $ {\ic Set} $.

\noindent (4) For vertical arrows $ v : A \tod \ov{A} $ we define $ \H (v) $ to be the span of all vertical cells
$$
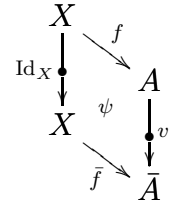
\bfig\scalefactor{.7}

\morphism(0,500)|r|/>/<400,-300>[X`A;f]

\morphism(0,0)|l|/@{<-}|{\bb}/<0,500>[X`X;\Id_X]

\morphism(0,0)|b|/>/<400,-300>[X`\ov{A};\ov{f}]

\morphism(400,200)|r|/@{>}|{\bb}/<0,-500>[A`\ov{A};v]

\place(200,100)[\scriptstyle \psi]

\efig
$$
now considered as a vertical arrow of $ {\ic Set} $.

\noindent (5) The action of $ \H $ on horizontal (resp.\ vertical) cells is supposed to be a morphism of spans and is given by transversal composition. This is strictly functorial as it should be.

So far this is just like the hom functors for double categories in \cite{P}. In particular $ \H $ will be lax on horizontal (resp.\ vertical) arrows.

However some care is needed in the definition of $ \H $ on basic cells
$$
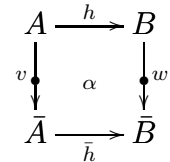
\bfig\scalefactor{.7}

\square/>`@{>}|{\bb}`@{>}|{\bb}`>/[A`B`\ov{A}`\ov{B};h`v`w`\ov{h}]

\place(250,250)[\scriptstyle \alpha]

\efig
$$
It will be a span of spans
$$
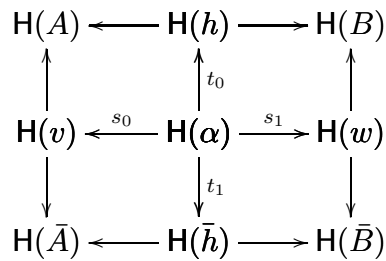
\bfig\scalefactor{.7}

\square/<-`>`>`<-/<700,500>[\H(v)`\H(\alpha) `\H(\ov{A}) `\H(\ov{h});``t_1`]

\square(0,500)|alra|/<-`<-`<-`<-/<700,500>[\H(A) `\H(h)`\H(v)`\H(\alpha);``t_0`s_0]

\square(700,0)<700,500>[\H(\alpha)`\H(w)`\H(\ov{h})`\H(\ov{B});s_1```]

\square(700,500)/>`<-`<-`>/<700,500>[\H(h)`\H(B)`\H(\alpha)`\H(w);```]

\efig
$$
$ \H (\alpha) $ will be the set of cubes
$$
\bfig\scalefactor{.7}

\square(0,300)/>`>``/[X `X `X `;\id_x `\Id_x ` `]

\square(300,0)[A ` B `\ov{A} `\ov{B}; h ` v `w' `\ov{h}]

\morphism(0,800)|b|<300,-300>[X` A;]

\morphism(500,800)<300,-300>[X ` B;]

\morphism(0,300)|b|<300,-300>[X`\ov{A};]

\place(550,250)[\scriptstyle \alpha]

\place(150,400)[\scriptstyle \psi]

\place(400,650)[\scriptstyle \phi]

\efig
$$
But there is a choice for the (hidden) back face, either $ \Id_{\id_x} $ or $ \id_{\Id_x} $. Only $ \Id_{\id_x} $ works and we'll see why below.

\noindent (6) $ \H (\alpha) $ is the set of cubes $ c : \Id_{\id_x} \to \alpha $ with the projections $ s_0 (c) = \psi $, $ s_1 (c) = \theta $ (the right face of $ c $), $ t_0 (c) = \phi $, $ t_1 (c) = \ov{\phi} $.

\noindent (7) The horizontal laxity morphisms of $ \H $ are as follows:
$$
{\H}_h : \id_{\H}(v) \to \H (\id_v)
$$
$$
(\psi : \Id_X \to v) \longmapsto (\Id_{\id_X} \to^\tau \id_{\Id_X} \to^{\id_\psi} \id_v)
$$
and
$$
{\H}_h : \H (\alpha) | \H (\beta) \to \H (\alpha | \beta)
$$
$$
(c : \Id_{\id_X} \to \alpha, d : \Id_{\id_X} \to \beta) \longmapsto (\Id_{\id_X} \to^{\Id_{\lambda^{-1}}} \Id_{\id_{X | \id_X}} \to^\delta \Id_{\id_X} | \Id_{\id_X} \to^{c|d} \alpha | \beta)
$$

\noindent (8) The vertical laxity morphisms are as follows.
$$
{\H}_v : \Id_{\H (h)} \to \H (\Id_h)
$$
$$
(\phi : \id_X \to h) \longmapsto (\Id_{\id_X} \to^{\Id_\phi} \Id_h)
$$
$$
{\H}_v : \dfrac{\H (\alpha)}{\H (\ov{\alpha})} \ \to \ \H (\dfrac{\alpha}{\ov{\alpha}})
$$
$$
(c : \Id_{\id_X} \to \alpha, \ov{c} : \Id_{\id_X} \to \ov{\alpha}) \longmapsto \left(\Id_{\id_X} \to^{\lambda^{-1}} \frac{\Id_{\id_X}}{\Id_{\id_X}} \to^{\frac{c}{\ov{c}}} \frac{\alpha}{\ov{\alpha}} \right) .
$$

This completes the description of $ \H $. Note that in (7) we had to use $ \tau $ and $ \delta $ whereas in (8) we only used the structural isomorphisms of $ {\ic A} $. Had we defined $ \H $ using $ \id_{\Id_X} $ as domain, (7) would only use the structural isomorphism whereas (8) would need $ \tau $ and $ \mu $, both of which go in the wrong direction.

$ \H $ has to satisfy a number of conditions, namely (5)-(14) of Section 5, \cite{ICI}. This is merely a question of working through the definitions above in the context of (5)-(14) and using the coherence conditions of Section 4 in \cite{ICI}. We do a few representative examples in detail.

In all of these diagrams, the objects are spans of spans of sets and the arrows are morphisms of such. In order to show commutativity it is sufficient to take an element of the middle set and follow its paths around the diagram and verify that we get the same thing in both cases.

Let's take (5) for example. For a basic cell
$$
\bfig\scalefactor{.7}

\square[A`B`\ov{A}`\ov{B};h`v`\ov{v}`\ov{h}]

\place(250,250)[\scriptstyle \alpha]

\efig
$$
we have to verify commutativity of 
$$
\bfig\scalefactor{.7}

\square/>`>`>`<-/<800,500>[\frac{\Id_{\H (h)}}{\Id_{\H (\alpha)}}
      `\frac{\H (\Id_h)}{\H (\alpha)}
      `\H (\alpha)
      `\H \left(\frac{\Id_h}{\alpha}\right);
   \frac{{\H}_v}{\H (\alpha)} `\lambda'  `{\H}_v  `\H (\lambda')]

\efig
$$
The upper left corner is the span of spans
$$
\bfig\scalefactor{.7}

\square/<-`>`>`<-/<1300,500>[\H (A) \times_{\H (A)} \H (v)
     `\H (h) \times_{\H (h)} \H (\alpha)
      `\H (\ov{A})
      `\H (\ov{h});```]

\square(0,500)/<-`<-`<-`<-/<1300,500>[\H (A) `\H (h) 
   `\H (A) \times_{\H (A)} \H (v)
    `\H (h) \times_{\H (h)} \H (\alpha);```]

\square(1300,0)<1300,500>[\H (h) \times_{\H (h)} \H (\alpha)
    `\H (B) \times_{\H (B)} \H (w)
    `\H (\ov{h}) `\H (\ov{B});```]

\square(1300,500)/>`<-`<-`>/<1300,500>[\H (h) ` \H (B)
    `\H (h) \times_{\H (h)} \H (\alpha)
   `\H (B) \times_{\H (B)} \H (w);```]

\efig
$$
and an element of $ \H (h) \times_{\H (h)} \H (\alpha) $ is a pair consisting of a horizontal cell $ \phi : \id_X \to h $ and a cube $ c : \Id_{\id_X} \to \alpha $ whose vertical domain is $ \phi $
$$
\bfig\scalefactor{.7}

\square(0,300)/>`>``/[X `X `X `; `` `]

\square(300,0)[A ` B `\ov{A}`\ov{B};  `  ` `]

\morphism(0,800)|b|<300,-300>[X` A;]

\morphism(500,800)<300,-300>[X ` B;]

\morphism(0,300)|b|<300,-300>[X`\ov{A};]

\place(550,250)[\scriptstyle \alpha]

\place(-200,500)[c:]

\place(400,650)[\scriptstyle \phi]

\morphism(0,1300)/>/<500,0>[X`X;]

\morphism(300,1000)/>/<500,0>[A`B;]

\morphism(0,1300)/>/<300,-300>[X`A;]

\morphism(500,1300)/>/<300,-300>[X`B;]

\place(400,1150)[\scriptstyle \phi]

\efig
$$

The $ \lambda' $ on the left is the isomorphism which takes $ (\phi, c) $ to $ c $. Going around the square gives
$$
\bfig\scalefactor{.7}

\square/|->``|->`<-|/<1000,500>[(\phi, c) `(\Id_\phi , c)
    `\lambda' \left(\frac{\Id_\phi}{c}\right) \lambda'^{-1}
     `\lambda' \left(\frac{\Id_\phi}{c}\right);```]

\efig
$$
Naturality of $ \lambda' $ 
$$
\bfig\scalefactor{.7}

\square<700,500>[\frac{\Id_{\id_X}}{\Id_{\id_c}}
      `\frac{\Id_h}{\alpha}
    `\Id_{\id_X}  `\alpha;
\frac{\Id_\phi}{c} `\lambda' `\lambda' `c]

\efig
$$
gives $ \lambda' \left(\frac{\Id_\phi}{c}\right) \lambda'^{-1} = c $.

Conditions (6) and (7) are virtually the same, using only the vertical double category coherence.

Condition (8) is the transpose of (5), but because it is about horizontal composition it will involve $ \tau $ and $ \delta $ and intercategory coherence. We must verify commutativity of
$$
\bfig\scalefactor{.7}

\square<1300,500>[\id_{\H (v)} | \H (\alpha)
    `\H (\id_h) | \H (\alpha)
     `\H (\alpha)  
     `\H (\id_v | \alpha);
{\H}_h | \H (\alpha) `\lambda `{\H}_h `\H (\lambda)]

\efig
$$
An element of the top left corner is a pair consisting of a vertical cell $ \psi : \Id_X \to v $ and a cube $ c : \Id_{\id_X} \to \alpha $ whose horizontal domain is $ \psi $
$$
\bfig\scalefactor{.7}

\morphism(400,750)|r|/>/<400,-300>[X`A;]

\morphism(400,250)|l|/<-/<0,500>[X`X;\Id_X]

\morphism(400,250)|b|/>/<400,-300>[X`\ov{A};]

\morphism(800,450)|r|/>/<0,-500>[A`\ov{A};v]

\place(600,350)[\scriptstyle \psi]

\square(1300,300)/>`>``/[X `X `X `; ` ` `]

\square(1600,0)[A ` B `\ov{A} `\ov{B};  `  ` `]

\morphism(1300,800)|b|<300,-300>[X` A;]

\morphism(1800,800)<300,-300>[X ` B;]

\morphism(1300,300)|b|<300,-300>[X`\ov{A};]

\place(1850,250)[\scriptstyle \alpha]

\place(1450,400)[\scriptstyle \psi]

\place(1100,550)[c:]

\efig
$$
$ \lambda $ takes $ (\psi, c) $ to $ c $, whereas going around the square we get first of all $ ((\id_\psi) \tau, c) $, then $ ((\id_\psi) (\tau) | c) \cdot \delta \Id_{\lambda^{-1}} $ and finally we multiply by $ \lambda $ to get the long way around the diagram
$$
\bfig\scalefactor{.7}

\qtriangle(0,500)/>`=`>/<1000,500>[\Id_{\id_X}
     `\Id_{\id_X | \id_X}
     `\Id_{\id_X};
  \Id_{\lambda^{-1}}``]

\square(1000,0)/<-`>`>`<-/<1000,500>[\Id_{\id_X}
   `\id_{\Id_X} | \Id_{\id_X}
 `\alpha
  `\id_v | \alpha;\lambda `c `\id_\psi | c `\lambda]

\square(1000,500)|alra|/>`>`>`<-/<1000,500>[\Id_{\id_X | \id_X}
  `\Id_{\id_X} | \Id_{\id_X}
   `\Id_{\id_X}
   `\id_{\Id_X} | \Id_{\id_X};
\delta `\Id_\lambda `\tau | \Id_{\id_X} `\lambda]

\efig
$$
The top square is condition (30) from Section 4 of \cite{ICI}.

Conditions (9) and (10) are very much the same.

Condition (11) reduces to naturality of $ \tau $.

For condition (12) we must verify the commutativity of
$$
\bfig\scalefactor{.7}

\square/>`>``>/<1000,500>[\frac{\id_{\H (v)}}{\id_{\H (\ov{v})}}
    `\frac{\H(\id_v)}{\H(\id_{\ov{v}})}
     `\id_{\frac{\H (v)}{\H (\ov{v})}}
    `\id_{\H\left(\frac{v}{\ov{v}}\right)};
\frac{{\H}_h}{{\H}_h}
`\mu``\id_{{\H}_v}]

\square(1000,0)/>``>`>/<1000,500>[\frac{\H(\id_v)}{\H(\id_{\ov{v}})}
`\H \left(\frac{\id_v}{\id_{\ov{v}}}\right)
   `\id_{\H\left(\frac{v}{\ov{v}}\right)}
    `\H \left(\id_{\frac{v}{\ov{v}}}\right);
{\H}_v``\H (\mu) `{\H}_h]

\efig
$$
This reduces to checking that the following diagram commutes for any two vertically composable vertical cells $ \psi : \Id_X \to v $ and $ \ov{\psi} : \Id_X \to \ov{v} $.
$$
\bfig\scalefactor{.7}

\qtriangle(0,0)/>`=`>/<1000,500>[\Id_{\id_X}`\dfrac{\Id_{\id_X}}{\Id_{\id_X}}`\Id_{\id_X};\lambda'^{-1}``\lambda']

\square(1000,0)/>``>`>/<1000,500>[\dfrac{\Id_{\id_X}}{\Id_{\id_X}}`\dfrac{\Id_{\id_X}}{\id_{\Id_X}}`\Id_{\id_X}`\id_{\Id_X};\frac{\Id_{\id_X}}{\tau}``\lambda'`\tau]

\square(2000,0)<1000,500>[\dfrac{\Id_{\id_X}}{\id_{\Id_X}}`\dfrac{\id_{\Id_X}}{\id_{\Id_X}}`\id_{\Id_X}`\id_{\frac{\Id_X}{\Id_X}};\frac{\tau}{\id_{\Id_X}}``\mu`\id_{\lambda^{-1}}]

\square(3000,0)<1000,500>[\dfrac{\id_{\Id_X}}{\id_{\Id_X}}`\dfrac{\id_v}{\id_{\ov{v}}}`\id_{\frac{\Id_X}{\Id_X}}`\id_{\frac{v}{\ov{v}}};\frac{\id_\psi}{\id_{\ov{\psi}}}``\mu`\id_{\frac{\psi}{\ov{\psi}}}]

\efig
$$
where the middle square is (22) of Section 5 in \cite{ICI} and the other two squares are naturality.

Conditions (13) and (14) are similar and left to the reader.

\section{Acknowledgements}

We thank the anonymous referee for a thoughtful reading of the paper and many helpful suggestions. Section \ref{Verity.4}, exploring notions of double functor inspired by intercategories, was added. Much of Section \ref{Gray.4} was rewritten to ``tighten up'' the statement and proof of Theorem \ref{TrueGray}. Section \ref{spans.5} was greatly expanded to include an intercategory perspective on profunctors.

These changes, along with a number of minor ones, have made for a much better paper.

\begin{references*}

\bibitem{AM} M.\ Aguiar, S.\ Mahajan, Monoidal functors, species and Hopf algebras, CRM Monograph Series 29, American Mathematical Society, Providence, RI, 2010.

\bibitem{laxtransf} J.\ B\'enabou, Introduction to bicategories, Reports of the Midwest Category Seminar, Lecture Notes in Mathematics, Vol. 47, Springer, Berlin 1967, pp.\ 1-77.

\bibitem{BCZ} G.\ B\"{o}hm, Y.\ Chen, L.\ Zhang,  On Hopf monoids in duoidal categories,  Journal of Algebra,
Vol.~394, 2013, pp.~139-172.

\bibitem{BS} T.\ Booker, R.\ Street, Tannaka duality and convolution for duoidal categories, Theory Appl. Categ. 28 (2013), No.\ 6, pp.\ 166-205.

\bibitem{DPP} R.\ Dawson, R.\ Par\'e, D.\ Pronk, The Span Construction, Theory and Applications of Categories, Vol.\ 24, No.\ 13, 2010, pp.\ 302-377.

\bibitem{dFASW} L. de Francesco Albasini, N. Sabadini, R.F.C. Walters, Cospans and spans of graphs: a categorical algebra for the sequential and parallel composition of discrete systems, arXiv:0909.4136v1 [math.CT] 23 Sep 2009.

\bibitem{GaGu} R.\ Garner, N.\ Gurski, The low-dimensional structures formed by tricategories, Mathematical Proceedings of the Cambridge Philosophical Society 146 (2009), no. 3, pp.\ 551-589.

\bibitem{GPS} R.\ Gordon, A.\ J.\ Power, R.\ Street, Coherence for Tricategories, Mem.  Amer. Math. Soc., 1995, Vol.\ 117, no.\ 558.

\bibitem{G1} M.\ Grandis, Higher cospans and weak cubical categories (Cospans in Algebraic Topology, I), Theory Appl. Categ. 18 (2007), no.\ 12, pp.\ 321-347.

\bibitem{GP99} M.\ Grandis, R.\ Par\'e, Limits in double categories, Cahiers Topologie G\'eom.\ Diff\'erentielle Cat\'eg.\ 40 (1999), no.\ 3, pp.\ 162-220.

\bibitem{GP} M.\ Grandis, R.\ Par\'e, Adjoint for double categories, Cahiers Topologie G\'eom.~Diff\'erentielle Cat\'eg.\ 45 (2004), pp.\ 193-240.

\bibitem{GP3} M.\ Grandis, R.\ Par\'e, Kan extensions in double categories (On weak double categories, Part III), Theory and Applications of Categories, Vol.\ 20, No.\ 8, 2008, pp.\ 152-185.

\bibitem{ICI} M.\ Grandis, R.\ Par\'e, Intercategories, Theory and Applications of Categories, Vol. 30, 2015, No. 38, pp 1215-1255. 

\bibitem{GP-Multiple} M.\ Grandis, R.\ Par\'e, An introduction to multiple categories (On weak and lax multiple categories, I), accepted for Cahiers Topologie G\'eom.~Diff\'erentielle Cat\'eg., preprint available at Pubbl. Mat. Univ. Genova, Preprint {\bf 604} (2015), http://www.dima.unige.it/\raisebox{-4pt}{\~{}}grandis/Mlc1.pdf

\bibitem{GP2} M.\ Grandis, R.\ Par\'e, Limits in multiple categories (On weak and lax multiple categories, II), accepted for Cahiers Topologie G\'eom.~Diff\'erentielle Cat\'eg., preprint available at Pubbl. Mat. Univ. Genova, Preprint {\bf 605} (2015), http://www.dima.unige.it/\raisebox{-4pt}{\~{}}grandis/Mlc2.pdf

\bibitem{Gray} J.\ W.\ Gray, Formal category theory: adjointness for 2-categories, Lecture Notes in Mathematics, Vol.\ 391, Springer, Berlin, 1974.

\bibitem{JS} A.\ Joyal, R.\ Street, Braided Tensor Categories, Advances in Mathematics 102, pp.\ 20-78 (1993).

\bibitem{VM} J.\ C.\ Morton, Double Bicategories and Double Cospans, Journal of Homotopy and Related Structures, Vol. 4(2009), no.~1, pp.\ 389-428.

\bibitem{P} R.\ Par\'e, Yoneda theory for double categories, Theory and Applications of Categories, Vol.\ 25, No.\ 17, 2011, pp.\ 436-489.

\bibitem{Sh} M. Shulman, Constructing symmetric monoidal bicategories, arXiv:1004.0993 [math.CT], 7 Apr 2010.

\bibitem{Shl2011} M. Shulman, Comparing composites of left and right derived functors, New York J. Math 17 (2011), 75-125.

\bibitem{V} D.\ Verity, Enriched categories, internal categories and change of base, Reprints in Theory and Applications of Categories, no.\ 20, 2011, pp.\ 1-266.

\end{references*}

\end{document}